\documentclass[11pt,openright,oneside,letterpaper,onecolumn]{report}  %% USE THIS FOR SINGLE SIDED

	\pdfoutput=1

\newcommand{\thesistitle}{Dualization and deformations of the Bar-Natan--Russell skein module}
\newcommand{\thesisauthor}{Andrea Heyman}
\newcommand{\thesisyear}{2016}

\usepackage{amsmath}
\usepackage{named}
\usepackage{fancyhdr}
\usepackage{afterpage}
\usepackage{epstopdf}
\usepackage[pdftex]{graphicx}
\usepackage{comment}
\usepackage{cite}
\usepackage{float}
\usepackage{morefloats}
\usepackage[all,cmtip]{xy}

\usepackage[pdftitle={\thesistitle},pdfauthor={\thesisauthor},pdfpagemode={UseOutlines},letterpaper,bookmarks=false,bookmarksopen=true,pdfstartview={FitH},bookmarksnumbered=true,]{hyperref}

\newcommand{\doublespace}{\renewcommand{\baselinestretch}{1.5} \small \normalsize}
\newcommand{\normalspace}{\doublespace}
\newcommand{\thesistitlepage}{
    \normalspace
    \thispagestyle{empty}
    \begin{center}
        \textbf{\LARGE \thesistitle} \\[1cm]
        \textbf{\LARGE \thesisauthor} \\[8cm]
        Submitted in partial fulfillment of the \\
        requirements for the degree \\
        of Doctor of Philosophy \\
        in the Graduate School of Arts and Sciences \\[4cm]
        \textbf{\Large COLUMBIA UNIVERSITY} \\[5mm]
        \thesisyear
    \end{center}
    \clearpage
}

\newcommand{\thesiscopyrightpage}{
    \thispagestyle{empty}
    \strut \vfill
    \begin{center}
      \copyright \thesisyear \\
      \thesisauthor \\
      All Rights Reserved
    \end{center}
    \cleardoublepage
}

\newtheorem{theorem}{Theorem}[section]
\newtheorem{lemma}[theorem]{Lemma}
\newtheorem{proposition}[theorem]{Proposition}
\newtheorem{corollary}[theorem]{Corollary}

\newenvironment{proof}[1][Proof]{\begin{trivlist}
\item[\hskip \labelsep {\bfseries #1}]}{\end{trivlist}}
\newenvironment{definition}[1][Definition]{\begin{trivlist}
\item[\hskip \labelsep {\bfseries #1}]}{\end{trivlist}}
\newenvironment{example}[1][Example]{\begin{trivlist}
\item[\hskip \labelsep {\bfseries #1}]}{\end{trivlist}}
\newenvironment{remark}[1][Remark]{\begin{trivlist}
\item[\hskip \labelsep {\bfseries #1}]}{\end{trivlist}}
\newenvironment{conjecture}[1][Conjecture]{\begin{trivlist}
\item[\hskip \labelsep {\bfseries #1}]}{\end{trivlist}}
\newenvironment{observation}[1][Observation]{\begin{trivlist}
\item[\hskip \labelsep {\bfseries #1}]}{\end{trivlist}}

\newcommand{\heart}{\,\heartsuit \,}
\newcommand{\diam}{\, \lozenge \,}

\usepackage{amssymb}
\begin{document}

% For the first pages we do not have numbering 
\pagestyle{empty}

\thesistitlepage
\thesiscopyrightpage

    \thispagestyle{empty}
    \begin{center}
    \textbf{\LARGE ABSTRACT} \\[1cm]
     \textbf{\LARGE \thesistitle} \\[1cm]
     \textbf{\LARGE \thesisauthor} \\[1cm]
    \end{center}
    This thesis studies the Bar-Natan skein module of the solid torus with a particular boundary curve system, and in particular a diagrammatic presentation of it due to Russell. This module has deep connections to topology and categorification: it is isomorphic to both the total homology of the $(n,n)$-Springer variety and the 0th Hochschild homology of the Khovanov arc ring $H^n$.

We can also view the Bar-Natan--Russell skein module from a representation-theoretic viewpoint as an extension of the Frenkel--Khovanov graphical description of the Lusztig dual canonical basis of the $U_q(\mathfrak{sl}_2)$-representation $V_1^{\otimes 2n}$. One of our primary results is to extend a dualization construction of Khovanov using Jones--Wenzl projectors from the Lusztig basis to the Russell basis.

We also construct and explore several deformations of the Russell skein module. One deformation is a quantum deformation that arises from embedding the Russell skein module in a space that obeys Kauffman--Lins diagrammatic relations. Our quantum version recovers the original Russell space when $q$ is specialized to $-1$ and carries a natural braid group action that recovers the symmetric group action of Russell and Tymoczko. We also present an equivariant deformation that arises from replacing the TQFT algebra $\mathcal{A}$ used in the construction of the rings $H^n$ by the equivariant homology of the two-sphere with the standard action of $U(2)$ and taking the 0th Hochschild homology of the resulting deformed arc rings. We show that the equivariant deformation has the expected rank.

Finally, we consider the Khovanov two-functor $\mathcal{F}$ from the category of tangles. We show that it induces a surjection from the space of cobordisms of planar $(2m, 2n)$-tangles to the space of $(H^m, H^n)$-bimodule homomorphisms and give an explicit description of the kernel. We use our result to introduce a new quotient of the Russell skein module.

    \cleardoublepage

% In the "roman-numbered" section of the thesis, we have numbers at the bottom
% and we have to reduce the textheight of the text to make space for the number

\pagenumbering{roman}
\pagestyle{plain}

\setlength{\footskip}{0.5in}

\setcounter{tocdepth}{2}
\renewcommand{\contentsname}{Table of Contents}
\tableofcontents
\cleardoublepage

\addcontentsline{toc}{chapter}{List of Figures}
\listoffigures
\cleardoublepage

\begin{comment}
\addcontentsline{toc}{chapter}{List of Tables}
\listoftables 
\cleardoublepage
\end{comment}

%%%
%%% Acknowledgments
%%%
~\\[1in] % hack to put space at top.
\textbf{\Huge Acknowledgments}\\

%\noindent 
Thank you to my advisor, Mikhail Khovanov, for your unique perspective and patience.

Thank you to my thesis committee, Mohammed Abouzaid, Melissa Liu, Josh Sussan, and especially Heather Russell, for your detailed and constructive comments.

Thank you to my amazing husband, Will, for loving and supporting me always.

Thank you to my family, Ellen, David, Ben, Vic, Reba, Ed, and Sylvia, for a lifetime of support and for being as proud of this accomplishment as I am.

Thank you to my classmates, St\'{e}phane, Rob, Karsten, Jo\~{a}o, Connor, Vivek, Vlad, Natasha, and Andrey, for making it fun to come to work each day for the last five years.

Thank you to Terrance, who went above and beyond to make the department a better place to work.

I was partially supported by NSF grants DMS-1406065 and DMS-1005750.

\cleardoublepage

\begin{comment}
%%%
%%% Dedication page
%%%
\thispagestyle{plain}
\strut \vfill
\centerline{\LARGE 
Dedication text
}
\vfill \strut
\cleardoublepage
\end{comment}

%\draft   % Generates a draft version in single-space

%%%
%%% BODY
%%%
\pagestyle{headings}
\pagenumbering{arabic}

%
% In the "arabic" section of the thesis, we do not have numbers at  the
% bottom and we want to use the full length of the page to avoid vbox
% underfulls. We use the fancyheaders package to adapt the headers
% according to the  Columbia requirements.
%
\setlength{\textheight}{8.5in}
\setlength{\footskip}{0.5in}

% We change the pagestyle 
\fancypagestyle{plain} {%
\fancyhf{}
\cfoot{\thepage}
\fancyhead[RE,LO]{\itshape \leftmark}
\renewcommand{\headrulewidth}{0pt}
}
\pagestyle{plain}

\chapter{Introduction}
\label{section:intro}
%\section{Introduction and Outline}

The primary object of study in this thesis will be the Bar-Natan skein module of the solid torus with a particular boundary curve system depending on a nonnegative integer $n$. A convenient set of diagrammatics for this skein module was provided by Russell \cite{russ}, and when using these diagrammatics we refer to the equivalent space as the Russell skein module $R_n$. In Russell's graphical calculus, diagrams consist of dotted crossingless matchings, subject to certain Type I and Type II relations (see Figure \ref{Russell_relns}), and a basis of diagrams is given by those that have dots on outer arcs only.

In \cite{khfrenkel}, Frenkel and Khovanov introduce a graphical calculus for the Lusztig dual canonical basis in tensor powers of irreducible representations of the quantum group $U_q(\mathfrak{sl}_2)$. In this thesis, we will be primarily concerned with even tensor powers of the fundamental representation $V_1$. In the Frenkel--Khovanov calculus, the basis of the invariant subspace of $V_1^{\otimes 2n}$, denoted $\mbox{Inv}(n)$, is exactly given by crossingless matchings of $2n$ points. In this sense, we view the basis of $R_n$ as an extension of the graphical basis of $\mbox{Inv}(n)$. Khovanov provides a description of the Lusztig canonical basis using a graphical approach in \cite{khthesis}. He first gives a graphical interpretation of the traditional bilinear form on $V_1^{\otimes n}$ and then constructs duals to the graphical dual Lusztig canonical basis of \cite{khfrenkel} using Jones--Wenzl projectors. In Chapter \ref{ch:Dual-Russell}, specializing to $q=-1$, we extend several of the results of Frenkel and Khovanov to the Russell basis.

In particular, in Section \ref{sec:bilin_form} we extend the graphical description of the Khovanov bilinear form on $\mbox{Inv}(n)$ to $R_n$ and show that it is well-defined, symmetric, and non-degenerate. As in the Khovanov case, our bilinear form admits a diagrammatic description. It also allows us to introduce a new graphical calculus for the dual of the Russell space (see Section \ref{sec:dual-calc}). The primary focus of Chapter \ref{ch:Dual-Russell} is to graphically describe dual elements to Russell basis elements with respect to this bilinear form. In Section \ref{sec:no-dots}, we review Khovanov's construction of the Lusztig canonical basis and tailor it to our set-up, which involves projecting to the invariant subspace $\mbox{Inv}(n)$ and specializing the value of $q$ to $-1$. The main result lies in Section \ref{sec:dual-dots}, Theorem \ref{dualbasisthm}, which extends the graphical construction of the Lusztig canonical basis to construct the dual Russell basis of $R_n$ using Jones--Wenzl projectors. The proof of this result occupies the majority of Chapter \ref{ch:Dual-Russell}.

While admitting a purely combinatorial description, the Russell skein module has strong topological significance. Recall that the $(n,n)$-Springer variety is the variety of complete flags in $\mathbb{C}^{2n}$ fixed by a nilpotent matrix with two Jordan blocks of size $n$. The following key result is due to Russell:

\begin{theorem}[Russell]
$R_n$ is isomorphic to the total homology of the $(n,n)$-Springer variety.
\end{theorem}

In \cite{russ-2}, Russell and Tymoczko describe a natural, combinatorial action of the symmetric group on a space isomorphic to the Russell skein module. They identify basis elements of that space, again given by ``standard'' crossingless matchings with dots on outer arcs only, with homology generators of $X_n$ and show that the $S_{2n}$ action they define is the Springer representation. Their action can be extended to the full Russell skein module by first rewriting any diagram with dots on inner arcs in terms of standard dotted matchings using Type I and Type II relations and then applying the $S_{2n}$ action previously defined. However, it should be noted that this extension does not have a local description, in the sense that the action of the symmetric group generator $s_i$ in general affects more arcs than just those with endpoints numbered $i$ or $i+1$.

In Chapter \ref{ch:Quantum-Russell}, we present a quantized version of the Russell skein module, denoted $R_{n,k}^q$, that is a deformation of the original skein module now considered over the ring $\mathbb{Z}[q, q^{-1}]$ for some fixed element $q$ of the ground ring $\Bbbk$ with deformed Type I and Type II relations depending on this parameter $q$ (see Figure \ref{quantum_Russell}). Specializing to $q = -1$ recovers the original Russell space. The deformed Type I and Type II relations are not local in the same sense as the original ones. Instead, they are semi-local, in the sense that they involve an additional term for each undotted arc containing the two arcs involved in the traditional Russell relations. In Section \ref{sec:no-local}, we explain that purely local quantum Russell relations are not possible.

In Section \ref{sec:KL-def} we first define spaces $\widetilde{S}^q_{n,k}$ that obey traditional Kauffman--Lins diagrammatic relations. We then introduce spaces $S^q_{n,k}$ that are quotients of $\widetilde{S}^q_{n,k}$ and turn out to be isomorphic to $R^q_{n,k}$. Sections \ref{sec:psi-tilde} and \ref{sec:psi} describe the embedding of the quantum Russell space inside the isomorphic spaces $S^q_{n,k}$ for generic $q$. As a consequence, we are able to identify a convenient basis for $R^q_{n,k}$ and compute its dimension.

The embedding of $R_{n,k}^q$ in a quotient of a Kauffman--Lins space is advantageous, as Kauffman--Lins diagrammatics are well-understood. In particular, the Kauffman--Lins space carries a natural action of the braid group, where the action of the generator $\sigma_i$ is given by attaching a positive crossing between strands numbered $i$ and $i+1$. In Section \ref{sec:Bn-act} we pull back the well-defined action of $B_{2n}$ on $S^q_{n,k}$ to the quantum Russell space. We observe that when we consider the subspace $\overline{R}^q_{n,k}$ of $\widetilde{R}^q_{n,k}$ with a basis given by diagrams with dots on outer arcs only, which is isomorphic to the quantum Russell space $R_{n,k}^q$, the braid group action admits a fully local description (Figure \ref{braid_group_act}), as it does in the Russell--Tymoczko case. That is, the spaces fit into a commutative diagram

\centerline{
\xymatrix{
\overline{R}_{n,k}^q \ar@{^{(}->}[r]  \ar@/_1pc/[rr]_{\cong} & \widetilde{R}_{n,k}^q \ar@{->>}[r] & R_{n,k}^q \\
}} 
\noindent where the leftmost and rightmost spaces both carry a $B_{2n}$ action, and the bent arrow represents an isomorphism intertwining these actions, but only the action on the leftmost space is local.

In Section \ref{sec:symm_gp_act-1} we explain that, when $q=1$, the Kauffman--Lins-induced braid group action descends to a symmetric group action on $\overline{R}^{q=1}_{n,k}$, and this action is identical to that of Russell and Tymoczko. We note that the space on which Russell and Tymoczko define their action actually corresponds with $\overline{R}^{q=-1}_{n,k}$ in our set-up, but there is no conflict because these spaces are naturally isomorphic. When $q = \pm 1$, the obstruction to locality on the full quantum Russell space disappears, so we extend the Russell--Tymoczko symmetric group action to the full $q=1$ skein module in a local manner in Section \ref{sec:symm_gp_act-2}.

Russell's result connecting $R_n$ to the homology of the $(n,n)$-Springer variety came by working with an alternate topological space $\widetilde{S}$, whose homology and cohomology rings are isomorphic to those of the Springer variety (and which was later shown by Wehrli \cite{wehrli} to actually be homeomorphic to the Springer variety.) $\widetilde{S}$ was introduced by Khovanov in \cite{crossmatch}, where he showed a parallel result concerning its cohomology:

\begin{theorem} [Khovanov]
The cohomology of the $(n,n)$-Springer variety is isomorphic to the center of the arc ring $H^n$.
\end{theorem}

The rings $H^n$ were invented by Khovanov in \cite{khov} when he extended his categorification of the Jones polynomial to the categorification of a tangle invariant. The combination of these two results of Russell and Khovanov provides an important link between the Russell skein module and the world of categorification: $R_n$ is isomorphic to the $0$th Hochschild homology of the ring $H^n$. This statement follows from the fact that the center of a ring is isomorphic to its $0$th Hochschild cohomology, the rings $H^n$ are symmetric, and the following:

\begin{proposition}
For a finite-dimensional symmetric algebra $A$ over a field $\Bbbk$, Hochschild homology and cohomology are dual vector spaces, i.e., for $n \geq 0$,
\[ HH^n(A) \cong \mbox{Hom}_{\Bbbk}(HH_n(A), \Bbbk).\]
\end{proposition}

In Chapter \ref{ch:equiv}, Section \ref{sec:HH0}, we give a direct proof of the isomorphism between $R_n$ and $HH_0(H^n)$. This isomorphism motivates our definition of what we call the ``equivariant'' deformation of the Russell skein module. In Section \ref{equi-Hn}, we first describe the equivariant deformation of the arc rings $H^n$, which we call $H^n_{h,t}$. This deformation comes by replacing the ring $\mathcal{A}$ used to define the TQFT from the construction of the original rings $H^n$, which is isomorphic to the cohomology ring of the two-sphere $S^2$, by the equivariant cohomology of $S^2$ with the standard action of $U(2)$.

In Section \ref{sec:equi-Russell}, we consider the 0th Hochschild homology of the rings $H^n_{h,t}$. We also present a graphical deformation of the Russell skein module depending on $h$ and $t$, $R_n^{h,t}$, and show that
\[ HH_0(H^n_{h,t}) \cong R_n^{h,t}.\]
In Section \ref{sec:equi-rank}, we show that $R_n^{h,t}$ is a free $\mathbb{Z}[h,t]$-module and has the same rank over $\mathbb{Z}[h,t]$ that $R_n$ has over $\mathbb{Z}$, justifying our description of $R_n^{h,t}$ as a deformation of $R_n$.

Finally, in Chapter \ref{ch:tangle-cob}, we return to the set-up of the rings $H^n$. We recall that the Khovanov two-functor $\mathcal{F}$ associates to cobordisms of planar $(2m, 2n)$-tangles homomorphisms of $(H^m, H^n)$-bimodules. We define two hom-spaces for any planar $(2m,2n)$-tangles $T_1, T_2$: $\mbox{Hom}_{BN}(T_1, T_2)$, the space of tangle cobordisms from $T_1$ to $T_2$ modulo the local Bar-Natan relations, and $\mbox{Hom}_{(m,n)}(T_1, T_2)$, the space of homomorphisms from $\mathcal{F}(T_1)$ to $\mathcal{F}(T_2)$ as $(H^m, H^n)$-bimodules. The main result of this chapter is Theorem \ref{thm:bimod-kernel}, which says that the map between these hom-spaces is in fact surjective: every bimodule homomorphism arises from a tangle cobordism in this way. Further, our theorem gives an explicit description of its kernel.

Summing over all planar $(m,n)$-tangles $T_1, T_2$, we get a surjection
\[ H^{m+n} \cong \mbox{Hom}_{BN}(m,n) \xrightarrow{\phi_{m.n}} \mbox{Hom}(m,n).\]
The composition of these maps takes the center of $H^{m+n}$ to the center of $\mbox{Hom}(m,n)$. While we do not have a description of $Z(\mbox{Hom}(m,n))$, it contains $\mbox{im}(\phi_{m,n}|_Z)$, which we show in Proposition \ref{prop:Z-Hom} is isomorphic to $Z(H^m) \otimes Z(H^n)$. We conjecture that $Z(\mbox{Hom}(m,n))$ is in fact isomorphic to $Z(H^m) \otimes Z(H^n)$.

Thinking of the rings $\mbox{Hom}(m,n)$ as quotients of $H^{m+n}$, in Section \ref{sec:quotient-Russell} for any integers $m,n$ we define a quotient of the Russell space $R_{m+n}$ by considering $HH_0(\mbox{Hom}(m,n))$ as a quotient of $HH_0(H^{m+n}) \cong R_{m+n}$. In Section \ref{sec:x^2=t}, we extend the surjectivity result to the equivariant case, where $x^2 = t$.

%\section{Outline}

\chapter{Preliminaries}

\section{The rings $H^n$}

\subsection{Definition}
\label{sec:Hn-def}
The rings $H^n$, sometimes referred to as arc rings, were introduced by Khovanov in \cite{khov} in the context of the categorification of a tangle invariant that extends the Jones polynomial.

Their construction involves the two-dimensional topological quantum field theory (TQFT) functor $\mathcal{F}$ used in the definition of Khovanov homology \cite{kh-hom}, the original of the link homology theories. By a two-dimensional TQFT we mean a functor from the category of two-dimensional cobordisms between closed one-manifolds to the category of abelian groups and group homomorphisms. It was shown by Abrams \cite{abrams} that two-dimensional TQFTs exactly correspond to commutative Frobenius algebras. In our case, $\mathcal{F}$ will be defined by a Frobenius algebra $\mathcal{A}$ given as follows.

As a (graded) free abelian group, let $\mathcal{A}$ have rank two, spanned by $1$ and $X$, with $1$ in degree $-1$ and $X$ in degree $1$. Then introduce a commutative, associative multiplication map $m: \mathcal{A} \otimes \mathcal{A} \to \mathcal{A}$, graded of degree 1, by
\[ 1^2 = 1, \quad 1X = X1 = X, \quad X^2 = 0.\]
The unit map $\iota: \mathbb{Z} \to \mathcal{A}$ is defined by $\iota(1) = 1$. The trace map $\epsilon: \mathcal{A} \to \mathbb{Z}$ is defined by
\[ \epsilon(1) = 0, \quad \epsilon(X) = 1.\]

The functor $\mathcal{F}$ associates to a disjoint union of $k$ circles the abelian group $\mathcal{A}^{\otimes k}$. For the elementary cobordisms, $\mathcal{F}$ associates $m$ to the ``pair of pants'' cobordism from two circles to one circle, $\iota$ to the cup cobordism from the empty manifold to a single circle, and $\epsilon$ to the cap cobordism from a single circle to the empty manifold (see Figure \ref{TQFT}).

\begin{figure}[htbp]
   \centering
   \includegraphics[scale=.8]{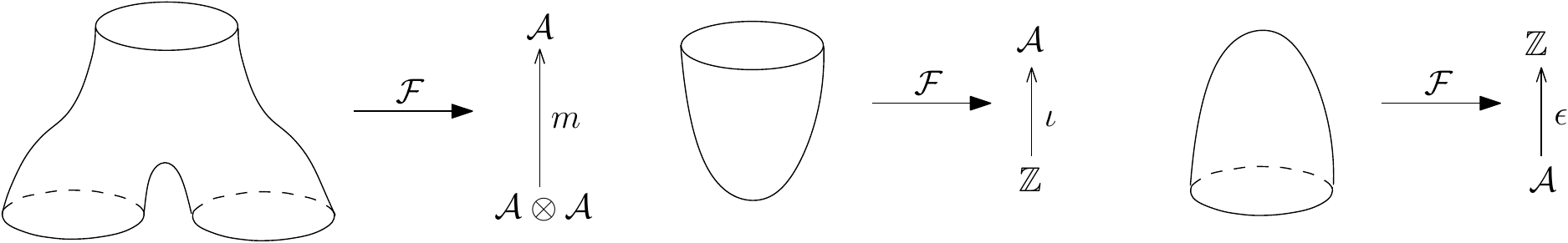}
   \caption{The TQFT functor $\mathcal{F}$.}
   \label{TQFT}
\end{figure}

The cobordism from one circle to two circles is different from that from two circles to one circle and is associated the map $\Delta: \mathcal{A} \to \mathcal{A} \otimes \mathcal{A}$:
\[ \Delta(1) = 1 \otimes X + X \otimes 1, \quad \Delta(X) = X \otimes X.\]

\begin{definition}
Let $B^n$ denote the set of isotopy classes of pairwise disjoint embeddings of $n$ arcs in $\mathbb{R} \times [0,1]$ connecting in pairs $2n$ points on $\mathbb{R} \times \{1\}$. Elements of $B^n$ will be referred to as \emph{crossingless matching cups}, or sometimes just \emph{crossingless matchings}, of $2n$ points.
\end{definition}

Given a crossingless matching cup $m \in B^n$, we let $W(m)$ denote the reflection of $m$ about the line $\mathbb{R} \times \{\frac{1}{2} \}$, so that $W(m)$ is a crossingless matching cap connecting $2n$ points on $\mathbb{R} \times \{0 \}$.

We are now able to describe the space underlying the finite-dimensional graded ring $H^n$, for $n \geq 0$. As a graded abelian group, $H^n$ decomposes into the direct sum
\[ H^n = \bigoplus_{a, b \in B^n} {}_b(H^n)_a,\]
where
\[{}_b (H^n)_a := \mathcal{F}(W(b)a)\{n\}. \]

Here $W(b)a$ represents the closed one-manifold formed by gluing together the diagrams $W(b)$ and $a$ along their $2n$ fixed points to get a disjoint union of manifolds that are isotopic to circles. The notation $\{n\}$ means that we shift the grading up by $n$, that is, if a graded $G$-module has summand $G_k$ in degree $k$, then $G\{n\}$ has summand $G_{k-n}$ in degree $k$.

To give $H^n$ a ring structure, we must define its multiplication. First define $uv$ to be $0$ if $u \in {}_d(H^n)_c$ and $v \in {}_b (H^n)_a$ where $c \neq b$. If this is not the case, then the multiplication maps
\[{}_c (H^n)_b \otimes {}_b(H^n)_a	\to {}_c(H^n)_a\]
will be defined as follows.

Given the one-manifolds $W(c)b$ and $W(b)a$ in $\mathbb{R} \times [0,1]$, we form the one-manifold $W(c)bW(b)a$ by vertically stacking $W(c)b$ on top of $W(b)a$ and scaling the second coordinate by a factor of $\frac{1}{2}$ to get a configuration of circles in $\mathbb{R} \times [0,1]$. Consider the ``simplest'' cobordism, denoted $S(b)$, from $bW(b)$ to $\mbox{Vert}_{2n}$, the one-manifold of $2n$ arcs embedded vertically in $\mathbb{R} \times [0,1]$. More precisely, $S(b)$ is a surface in $\mathbb{R} \times [0,1] \times [0,1]$ with bottom boundary equal to $W(b)b$, top boundary equal to $\mbox{Vert}_{2n}$, and $S(b)$ is diffeomorphic to a disjoint union of $n$ discs. Let $\mbox{Id}_{W(c)}S(b)\mbox{Id}_{a}$ be the cobordism
\[ W(c)bW(b)a \to W(c)a\]
given by composing $S(b)$ with the identity cobordisms from $W(c)$ to itself and $a$ to itself. By applying the TQFT functor, we get a map
\[ \mathcal{F}(W(c)bW(b)a) \to \mathcal{F}(W(c)a).\]
Composing with the canonical isomorphism $\mathcal{F}(W(c)b) \otimes \mathcal{F}(W(b)a) \to \mathcal{F}(W(c)bW(b)a)$, we get a map
\[\mathcal{F}(W(c)b) \otimes \mathcal{F}(W(b)a) \to \mathcal{F}(W(c)a). \]
Note that the surface $\mbox{Id}_{W(c)}S(b)\mbox{Id}_{a}$ has $n$ saddle points, and both $m$ and $\Delta$ have degree $1$, so that the above map has degree $n$. Therefore, after shifting, the map
\[\mathcal{F}(W(c)b)\{n\} \otimes \mathcal{F}(W(b)a)\{n\} \to \mathcal{F}(W(c)a)\{n\}\]
is grading-preserving. This is the map that defines the associative multiplication 
\[m_{c,b,a}: {}_c (H^n)_b \otimes {}_b(H^n)_a	\to {}_c(H^n)_a.\]

The unit $1$ in $H^n$ is the sum over all crossingless matchings $a \in B^n$ of idempotent elements $1_a \in {}_aH^n_a$ defined as the element $1^{\otimes n} \in \mathcal{A}^{\otimes n}\{n\} = {}_aH^n_a$.

With underlying abelian group structure, multiplication, and unit as above, $H^n$ is now a graded, associative, unital ring.

\begin{example}
To clarify the ring structure of $H^n$, we examine the case where $n=2$. $B^2$ consists of two crossingless matchings: that of two unnested adjacent arcs, which we call $a$, and that of two nested arcs, which we call $b$. The diagrams $W(a)a, W(a)b, W(b)a,$ and $W(b)b$ are shown in Figure \ref{H2-ex}.

\begin{figure}[htbp]
   \centering
   \includegraphics[scale=1]{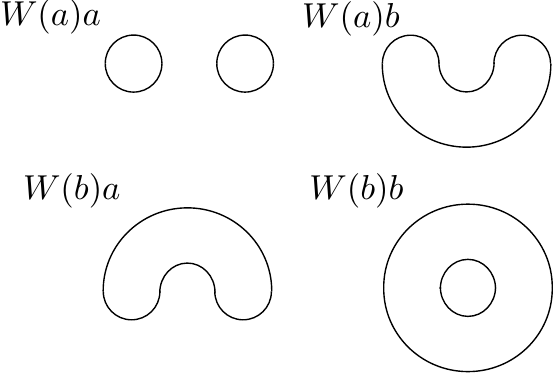}
   \caption{The diagrams underlying the ring $H^2$.}
   \label{H2-ex}
\end{figure}

After applying $\mathcal{F}$ to those diagrams, we see that the structure of $H^2$ as a graded abelian group is given by
\begin{eqnarray*}
H^2 &=& {}_a(H^2)_a \oplus {}_a(H^2)_b \oplus {}_b(H^2)_a \oplus {}_b(H^2)_b \\
&=& \mathcal{A}^{\otimes 2} \{2 \} \oplus \mathcal{A} \{ 2 \} \oplus \mathcal{A} \{2\} \oplus \mathcal{A}^{\otimes 2} \{2\}.
\end{eqnarray*}
As an example of the multiplication in $H^2$, consider the map ${}_a(H^2)_b \otimes {}_b(H^2)_a \to {}_a(H^2)_a$. This map is induced by the ``simplest'' cobordism from $W(a)bW(b)a$ to $W(a)a$, which involves two saddles that first merge the two circles and then splits them:
\[ \mathcal{A} \{2\} \otimes \mathcal{A}\{2\} \xrightarrow{m} \mathcal{A}\{3\} \xrightarrow{\Delta} \mathcal{A}^{\otimes 2}\{2\}. \]
\end{example}

\subsection{Role in categorification}
To place the rings $H^n$ in their appropriate context, it is important to mention that they were constructed by Khovanov as part of the categorification of an extension of the Jones polynomial to tangles \cite{khov}. We briefly explain that process here.

\begin{definition}
A $(m,n)$-tangle is a one-dimensional cobordism in $\mathbb{R}^2 \times [0,1]$ from the 0-manifold of $n$ points lying on the bottom boundary $\mathbb{R}^2 \times \{0\}$ to the 0-manifold of $m$ points lying on the top boundary $\mathbb{R}^2 \times \{1\}$.
\end{definition}

The extended Jones polynomial associates to a $(2m, 2n)$-tangle $T$ a map $J(T): \mbox{Inv}(n) \to \mbox{Inv}(m)$, where $\mbox{Inv}(k)$ is defined to be the $U_q(\mathfrak{sl_2})$-invariant subspace of $V_1^{\otimes 2k}$, with $V_1$ the fundamental two-dimensional representation of $U_q(\mathfrak{sl}_2)$. We can consider tangles as the one-morphisms of a two-category in which objects correspond to nonnegative integers (the number of fixed points on the boundary of a tangle) and two-morphisms are tangle cobordisms. We will restrict to the subcategory of even tangles, in which objects are even integers and one-morphisms are tangles with an even number of top and bottom endpoints. Khovanov's work defines a functor from the two-category of even tangles that turns an object $2n$ into the ring $H^n$, a $(2m,2n)$-tangle $T$ into a chain complex of $(H^m, H^n)$-bimodules, and a tangle cobordism into a map of such chain complexes. After categorification, $J(T)$ becomes a functor from $\mathcal{K}_P^m$ to $\mathcal{K}_P^n$, where $\mathcal{K}_P^n$ is the category of bounded complexes of graded projective $H^n$-modules.

\section{The Bar-Natan--Russell skein module}

In this section we will present the definition of the Bar-Natan--Russell skein module, which will be our primary object of study, and survey previous results related to it. Key references are \cite{russ} and \cite{russ-2}.

\subsection{Bar-Natan skein module of solid torus}
In foundational work \cite{BN}, Bar-Natan gives an alternate proof of Khovanov's link homology, and more generally its extension to tangles, coming from a more topological viewpoint. Khovanov's original construction forms a cube of resolutions, which is a complex of tensor products of the algebra $\mathcal{A}$ formed by applying the functor $\mathcal{F}$ to closed one-manifolds obtained by resolving all crossings in the tangle in all possible ways, with maps between then involving $m$ or $\Delta$ according to whether circles are being merged or split.

Bar-Natan's construction delays the application of $\mathcal{F}$ until later in the process. His cube of resolutions consists of the one-mainfolds themselves, before the application of $\mathcal{F}$, with the ``maps'' between them replaced by marked cobordisms, where marked means that sheets in the cobordisms may carry dots. In the Bar-Natan set-up, cobordisms are subject to the local relations of Figure \ref{BN-relns}, commonly referred to as the \emph{Bar-Natan relations}. The fourth relation is commonly called the ``neck-cutting'' relation.

\begin{figure}[htbp]
\[
  \begin{array}{c}
    \includegraphics[height=1cm]{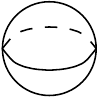}
  \end{array}\hspace{-2mm}=0
  \qquad\qquad
  \begin{array}{c}
    \includegraphics[height=1cm]{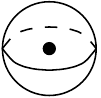}
  \end{array}\hspace{-2mm}=1
  \qquad\qquad
  \begin{array}{c}\includegraphics[height=10mm]{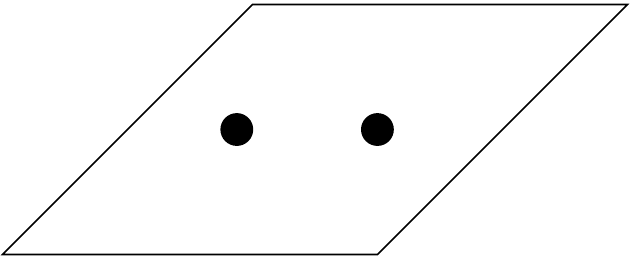}\end{array}
  \hspace{-4mm}=0
\]
\[
  \begin{array}{c}\includegraphics[height=10mm]{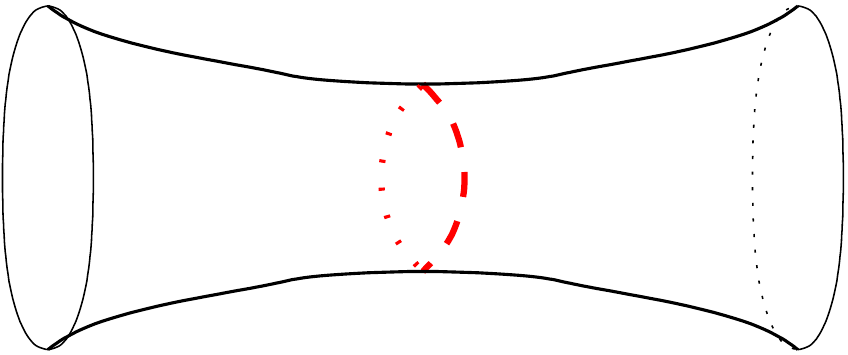}\end{array}
  =\begin{array}{c}\includegraphics[height=10mm]{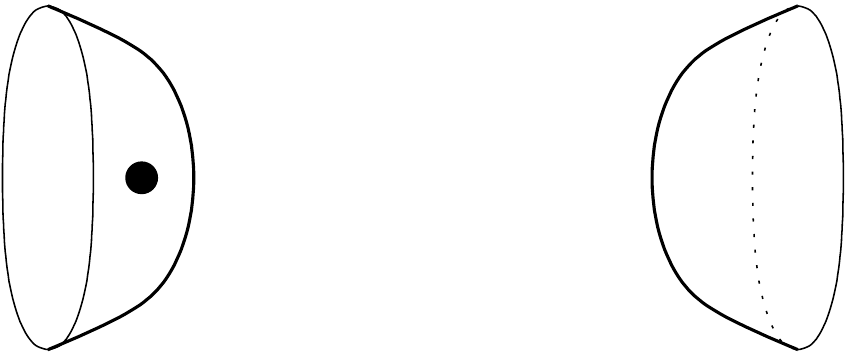}\end{array}
  +\begin{array}{c}\includegraphics[height=10mm]{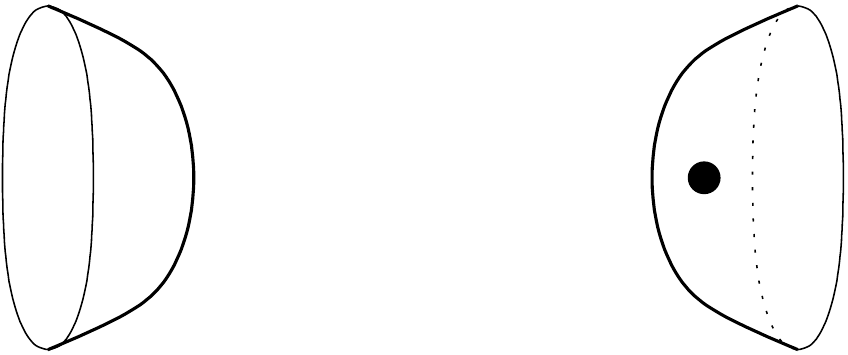}\end{array}
\]
\caption{Local Bar-Natan relations}
\label{BN-relns}
\end{figure}

These relations will be essential to the definition of what we will call the Bar-Natan skein module. While such a skein module can be defined for any three-manifold $M$, we will only be interested in the special case in which $M = A \times I$, where $A$ is the standard planar annulus. Also, for $n \geq 0$, we fix a boundary curve system $c_{n}$ consisting of $2n$ disjoint copies of the longitude of the solid torus, considered to be embedded in $A \times \{1\} \subset M$, which we think of as the ``top'' of the torus.

\begin{definition}
For $n \geq 0$, define $\mathcal{BN}_n$ to be the $\mathbb{Z}$-module generated by marked surfaces $S \subset A \times I$ modulo isotopy with boundary $c_n$ subject to the local Bar-Natan relations of Figure \ref{BN-relns}. We refer to $\mathcal{BN}_n$ as the $n$th \emph{Bar-Natan skein module}.
\end{definition}

\begin{example}
Figure \ref{BN1} shows two surfaces which generate the skein module $\mathcal{BN}_1$. Each surface is a half-torus with boundary $c_1 \subset A \times \{1\}$. The surface on the left carries 0 dots, while the surface on the right carries 1 dot. There are no Bar-Natan relations between these two surfaces.
\begin{figure}[htbp]
   \centering
   \includegraphics[scale=.8]{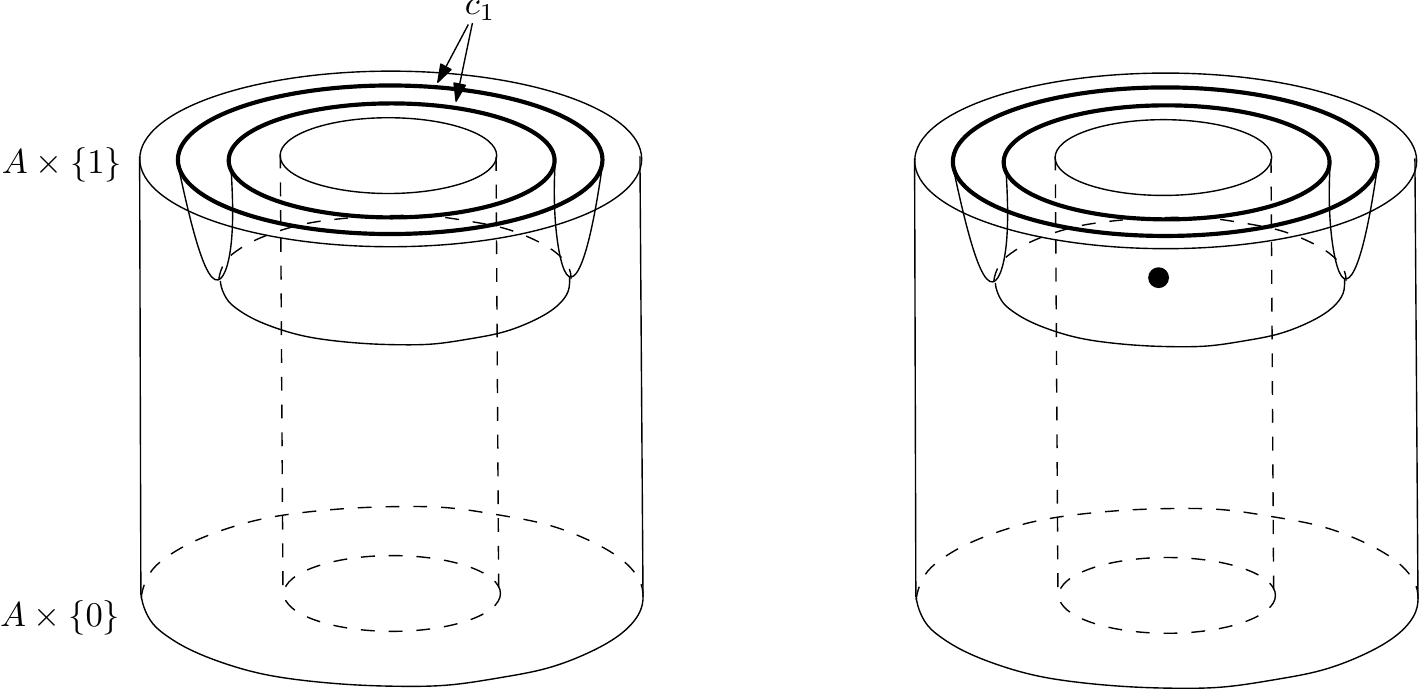}
   \caption{Generators for $\mathcal{BN}_1$.}
   \label{BN1}
\end{figure}
\end{example}

\subsection{Russell diagrammatics}
In this thesis, we will primarily consider the Bar-Natan skein module in terms of a diagrammatic calculus due to Russell \cite{russ}.

\begin{definition}
For $0 \leq k \leq n$, let $\widetilde{R}_{n,k}$ be the space of formal linear combinations with coefficients in $\mathbb{Z}$ of diagrams, where a diagram is defined to be a crossingless matching of $2n$ fixed points on a line decorated with $k$ dots such that each arc carries at most one dot.
\end{definition}

The $\mathbb{Z}$-module $R_{n,k}$ is defined to be the quotient of $\widetilde{R}_{n,k}$ by certain relations which we now describe. Let $\alpha$ and $\beta$ be crossingless matchings of $2n$ points that have identical arcs except that for some fixed points numbered $a < b < c < d$, $\alpha$ has arcs with endpoints $(a,b)$ and $(c,d)$ while $\beta$ has arcs with endpoints $(a,d)$ and $(b,c)$ (where we number the $2n$ fixed points 1 through $2n$ from left to right.)
\begin{enumerate}
\item \textbf{Type I relations: } Let $m_1$ and $m_2$ be diagrams in $\widetilde{R}_{n,k}$ that have the arc structure of $\alpha$, where $m_1$ has the arc $(a,b)$ dotted and the arc $(c,d)$ undotted while $m_2$ has the arc $(a,b)$ undotted and the arc $(c,d)$ dotted. Similarly let $m_1'$ and $m_2'$ have the arc structure of $\beta$, where $m_1'$ has $(a,d)$ dotted and $(b,c)$ undotted while $m_2'$ has $(a,d)$ undotted and $(b,c)$ dotted. Suppose that $m_1, m_2, m_1',$ and $m_2'$ are identical away from arcs with endpoints $a,b,c,d$. Then we impose the Type I relation
\[ m_1 + m_2 - m_1' - m_2' = 0.\]

\item \textbf{Type II relations: } Let $m_3 \in \widetilde{R}_{n,k}$ have the arc structure of $\alpha$ with dots on the arcs $(a,b)$ and $(c,d)$ and $m_3'$ have the arc structure of $\beta$ with dots on the arcs $(a,d)$ and $(b,c)$, where $m_3$ and $m_3'$ are identical away from $a,b,c,d$. Then we impose the Type II relation
\[m_3 - m_3' = 0.\]
\end{enumerate}

See Figure \ref{Russell_relns} for pictures of the Russell relations. Due to the locality of the relations, only arcs with endpoints at $a,b,c,$ or $d$ are shown, since any arcs not ending at $a,b,c,d$ are identical in each diagram of the relation.

\begin{figure}[H]
\centering
\includegraphics[scale=1]{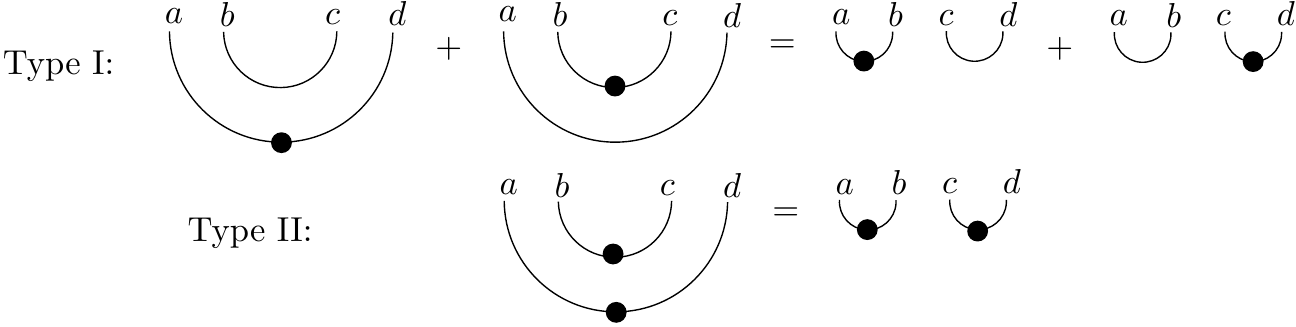}
\caption{Type I and Type II Russell relations}
\label{Russell_relns}
\end{figure}

\begin{definition}
We define the $(n,k)$-\emph{Russell skein module} $R_{n,k}$ to be the quotient of $\widetilde{R}_{n,k}$ by all Type I and Type II relations. We define the $n$th Russell skein module $R_n$ to be $\oplus_{0 \leq k \leq n} R_{n,k}$.
\end{definition}

\begin{theorem}[Russell]
$\mathcal{BN}_n$ and $R_n$ are isomorphic as $\mathbb{Z}$-modules.
\end{theorem}

\begin{proof}
A sketch of the proof goes as follows. Given a marked surface in $\mathcal{BN}_n$, it can be reduced using the Bar-Natan relations into a configuration of $n$ half-tori where each sheet carries at most one dot. Taking a vertical cross-section of such a configuration gives a dotted crossingless matching of $2n$ points, where we place a dot on an arc if the corresponding half-torus in $\mathcal{BN}_n$ was dotted.

Given any pair of half-tori, neither of which is nested inside any other half-torus, a tube can be inserted between the two. Then, a neck-cutting relation can be performed in one of two ways, either on the compressing disk inside the tube or on that which goes around the puncture and has boundary along the tube and two half-tori. The results of these neck-cutting relations exactly correspond with the Type I and Type II relations. $\blacksquare$
\end{proof}

Because they are isomorphic, we frequently refer to the Bar-Natan and Russell skein modules jointly as the Bar-Natan--Russell skein module.

\begin{comment}
It was shown in \cite{russ-2}, Corollary 2.16, that there is a basis of $R_{n,k}$ given by standard crossingless matchings, which are those diagrams of $\widetilde{R}_{n,k}$ that only have dots on ``outer'' arcs. An outer arc $\alpha$, with endpoints $(b,c)$ is defined to be one such that there is no other arc $\beta$ in the diagram with endpoints $(a,d)$ such that $a<b$ and $d>c$. If such an arc $\beta$ does exist, we say that $\alpha$ is an inner arc that is nested inside $\beta$. Any dotted arc inside an undotted arc may be reexpressed using a Type I relation, and any dotted arc inside a dotted arc may be reexpressed using a Type II relation.
\end{comment}

\section{The quantum group $U_q(\mathfrak{sl}_2)$}
\subsection{Definition}
The representation theory of the quantum group $U_q(\mathfrak{sl}_2)$ plays a prominent role in the categorification of certain low-dimensional topological invariants. We recall the necessary pieces of that story here, following \cite{kassel} and \cite{CFS}.

\begin{definition}
Let $q$ be an indeterminate with a fixed value in $\mathbb{C}$ different from $0, 1, -1$. The quantum group $U_q(\mathfrak{sl}_2)$ is an associative algebra over $\mathbb{C}(q)$, the field of complex-valued rational functions in $q$, with four generators labeled $E, F, K, K^{-1}$ subject to the relations
\begin{gather*}
KK^{-1} = K^{-1}K = 1, \\
KE = q^2 EK, \quad KF = q^{-2}FK, \\
\lbrack E,F\rbrack = \frac{K-K^{-1}}{q-q^{-1}}.
\end{gather*}
\end{definition}

$U_q(\mathfrak{sl}_2)$ can be equipped with the structure of a Hopf algebra with comultiplication $\Delta$ and counit $\varepsilon$ defined on generators by
\begin{gather*}
\Delta(E) = E \otimes 1 + K^{-1} \otimes E \\
\Delta(F) = F \otimes K + 1 \otimes F \\
\Delta(K^{\pm 1}) = K^{\pm 1} \otimes K^{\pm 1}.
\end{gather*}
and
\[ \varepsilon(E) = \varepsilon(F) = 0, \quad \varepsilon(K) = \varepsilon(K^{-1}) = 1.\]
The antipode will not be needed here.

\subsection{Representation theory}

The representation theory of $U_q(\mathfrak{sl}_2)$ is well-known. For any nonnegative integer $n$ there is a unique irreducible representation of $U_q(\mathfrak{sl}_2)$ of dimension $n+1$, denoted $V_n$. $V_n$ has a basis labeled
\[
\{v^m\}, \quad -n \leq m \leq n, \quad m \equiv n (\mbox{mod } 2)
\]
and the action of $E, F,$ and $K$ is given by
\begin{gather*}
Ev^m = \left[ \frac{n-m}{2} \right] v^{m+2} \\
Fv^m = \left[ \frac{n+m}{2} \right] v^{m-2} \\
K^{\pm 1} v^m = q^{\pm m} v^{m}
\end{gather*}
\noindent where $\lbrack n \rbrack$, sometimes referred to as ``quantum $n$,'' is defined as
\[ \lbrack n \rbrack = \frac{q^n - q^{-n}}{q-q^{-1}}\]
and $v^{n+2}$ and $v^{-n-2}$ are 0.

We refer to the one-dimensional representation $V_0 \cong \mathbb{C}(q)$ as the \emph{trivial representation} and the two-dimensional representation $V_1 \cong \mathbb{C}(q)v^1 \oplus \mathbb{C}(q)v^{-1}$ as the \emph{fundamental representation}. We explicitly write out the actions of $E,F, $ and $K^{\pm 1}$ on the basis elements of $V_1$, since they will be used frequently:
\begin{gather*}
Ev^{1} = 0, \quad Ev^{-1} = v^1 \\
Fv^1 = v^{-1}, \quad Fv^{-1} = 0 \\
K^{\pm 1} v^1 = q^{\pm 1}v^1, \quad K^{\pm 1}v^{-1} = q^{\mp 1}v^{-1}.
\end{gather*}

The comultiplication $\Delta$ defined above determines the action of $U_q(\mathfrak{sl}_2)$ on tensor products of representations. We will need the following maps of representations intertwining the $U_q(\mathfrak{sl}_2)$ action:
\begin{gather*}
\varepsilon_1: V_1 \otimes V_1 \to V_0 \\
\varepsilon_1(v^1 \otimes v^1) = \varepsilon_1(v^{-1} \otimes v^{-1}) = 0, \quad \varepsilon_1(v^{-1} \otimes v^1) = 1, \quad \varepsilon_1(v^1 \otimes v^{-1}) = -q \\
\delta_1: V_0 \to V_1 \otimes V_1 \\
\delta_1(1) = v^1 \otimes v^{-1} -q^{-1}v^{-1} \otimes v^1.
\end{gather*}

\begin{lemma}
\label{edrelns}
The maps $\varepsilon_1$ and $\delta_1$ satisfy the relations
\begin{gather*}
(1 \otimes \varepsilon_1) \circ (\delta_1 \otimes 1) = 1 = (\varepsilon_1 \otimes 1) \circ (1 \otimes \delta_1) \\
\varepsilon_1 \circ \delta_1 = -q-q^{-1}.
\end{gather*}
\end{lemma}
\begin{proof}
By computation.
\end{proof}

We define an intertwining map $R_{1,1}: V_1 \otimes V_1 \to V_1 \otimes V_1$ in terms of $\delta_1$ and $\epsilon_1$ by
\[ R_{1,1} = q^{1/2} (\delta_1 \circ \epsilon_1) + q^{-1/2} \mbox{Id}.\]
We also define $T_i: V_1^{\otimes n} \to V_1^{\otimes n}$ for $1 \leq i \leq n-1$ by
\[ T_i = 1^{\otimes (i-1)} \otimes R_{1,1} \otimes 1^{\otimes (n-i-1)}.\]

\begin{lemma}
\label{T-braid}
The elements $T_1, \ldots, T_{n-1}$ satisfy the relations
\begin{gather*}
T_iT_{i+1}T_i = T_{i+1}T_iT_{i+1} \\
T_iT_j = T_jT_i, \quad |i-j| > 1.
\end{gather*}
\end{lemma}

\begin{proof}
By computation.
\end{proof}

\begin{proposition}
The elements $T_1, \ldots, T_{n-1}$ define an action of the $n$th braid group on $V_1^{\otimes n}$.
\end{proposition}

\begin{proof}
This follows immediately from the previous lemma.
\end{proof}

Finally, we will need the notion of Jones--Wenzl projectors. Recall that the symmetric group $S_n$ is generated by the elementary transpositions $s_1, \ldots, s_{n-1}$, with the relations
\[ s_i^2 = 1, \quad s_is_{i+1}s_i = s_{i+1}s_is_{i+1}, \quad s_is_j = s_js_i \quad |i-j| > 1.\]
For any permutation $s \in S_n$, define $l(s)$ to be the number of pairs $(i,j)$, $1 \leq i < j \leq n$, such that $s(i) > s(j)$. Then there exists a presentation of $s$ given by $s_{i_1} \cdots s_{i_{l(s)}}$. Such a presentation is not unique, but any two must be related by a sequence of relations of the second and third type of those above. We call such a presentation a \emph{reduced representation} of $s$. Then it follows from Lemma \ref{T-braid} that the following $U_q(\mathfrak{sl}_2)$-intertwining endomorphism of $V_1^{\otimes n}$ is well-defined:
\[ T(s) := T_{i_1} \cdots T_{i_{l(s)}}.\]

\begin{definition}
The $n$th \emph{Jones--Wenzl projector} $p_n$ is defined by
\[p_n = \frac{1}{[n]_{-}!} \sum_{s \in S_n} q^{-3l(s)/2}T(s),\]
where $[n]_{-}! = [n]_- \cdots [1]_-$, and $[i]_- = (q^{-2i}-1)/(q^{-2}-1).$
\end{definition}

For $1 \leq i \leq n-1$, define the $U_q(\mathfrak{sl}_2)$-intertwining endomorphism of $V_1^{\otimes n}$ by
\[ U_i := 1^{\otimes (i-1)} \otimes (\delta_1 \circ \varepsilon_1) \otimes 1^{\otimes (n-i-1)}.\]

\begin{proposition}
\label{JW-prop}
The Jones--Wenzl projectors satisfy the properties
\begin{gather*}
p_n^2 = p_n \\
p_nU_i = U_ip_n = 0, \quad 1 \leq i \leq n-1.
\end{gather*}
\end{proposition}

\begin{theorem}
\label{JW-thm}
Jones--Wenzl projectors satisfy the inductive relation
\[ p_{n+1} = (p_n \otimes 1) - \mu_n (p_n \otimes 1) \circ U_n \circ (p_n \otimes 1),\]
where $\mu_1 = 1/(-q-q^{-1})$ and $\mu_{k+1} = (-q-q^{-1}-\mu_k)^{-1}$.
\end{theorem}

For proofs of the previous proposition and theorem, see \cite{KL}.

\subsection{Graphical calculus}
\label{sec:graphical_calculus}
A well-known diagrammatic description of the $U_q(\mathfrak{sl}_2)$ intertwiners discussed in the previous subsection was invented by Penrose, generalized by Kauffman, and utilized by Kauffman and Lins to construct invariants of three-manifolds. This graphical calculus motivates our description of the dual Russell basis in Chapter \ref{ch:Dual-Russell} and our construction of the quantum deformation of the Russell skein module in Chapter \ref{ch:Quantum-Russell}. We review its key pieces here, following \cite{khfrenkel} and the standard references \cite{CFS} and \cite{KL}.

In the graphical calculus, an intertwiner from $V_1^{\otimes m}$ to $V_1^{\otimes n}$ is drawn as a diagram on $(m+n)/2$ strands with $m$ bottom and $n$ top endpoints. Composition of maps corresponds to the vertical stacking of diagrams, while tensor products correspond to horizontal placement.

The identity map on $V^{\otimes n}$ is drawn as $n$ vertical lines. The intertwining maps $\varepsilon_1, \delta_1$, and $R_{1,1}$ correspond to the diagrams in Figure \ref{intertwiners}, where the relation in the second row comes from the definition of $R_{1,1}$.

\begin{figure}[htbp]
\centering
\includegraphics[scale=1]{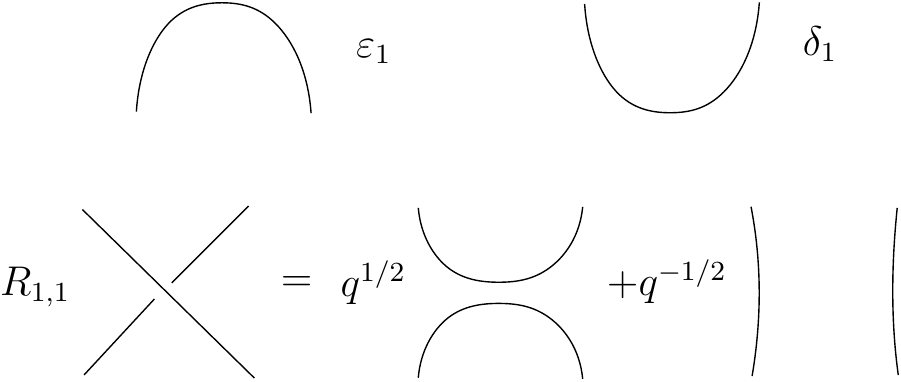}
\caption{Graphical $U_q(\mathfrak{sl}_2)$ intertwiners.}
\label{intertwiners}
\end{figure}

The relations of Lemma \ref{edrelns} give us isotopy and evaluation of the closed circle:
\begin{figure}[H]
\centering
\includegraphics[scale=1]{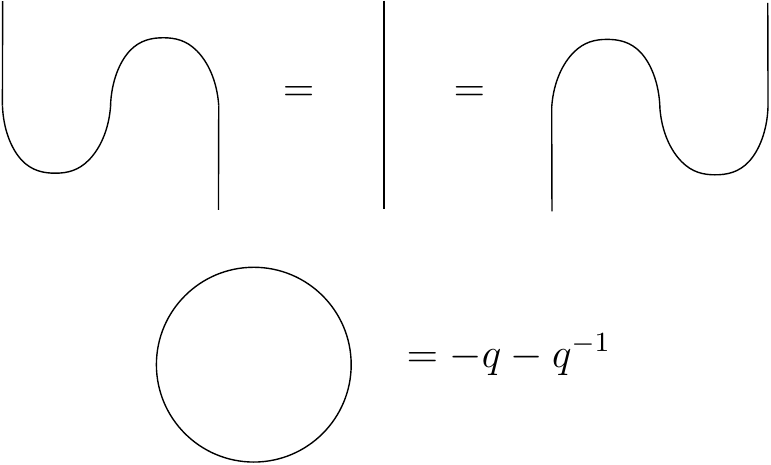}
\end{figure}

We frequently refer to the expansion of the crossing as in Figure \ref{intertwiners} and the evaluation of the circle to $-q-q^{-1}$ together as the \emph{Kauffman--Lins} relations.

The intertwiner $T_i$ is given by the following diagram.
\begin{figure}[H]
\centering
\includegraphics[scale=1]{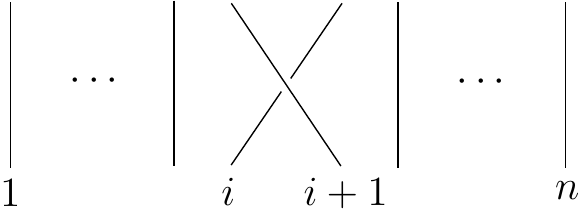}
\end{figure}

The Jones--Wenzl projector $p_n$ is depicted
\begin{figure}[H]
\centering
\includegraphics[scale=1]{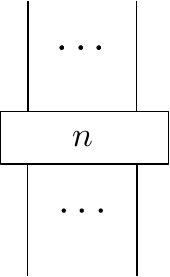}
\end{figure}
with the label ``$n$'' occasionally omitted when the size is clear. From the definition of $p_n$ and the expansion of the crossing $R_{1,1}$, $p_n$ can always be expressed as a $\mathbb{C}(q)$-linear combination of planar tangles. For example:
\begin{figure}[H]
\centering
\includegraphics[scale=1.2]{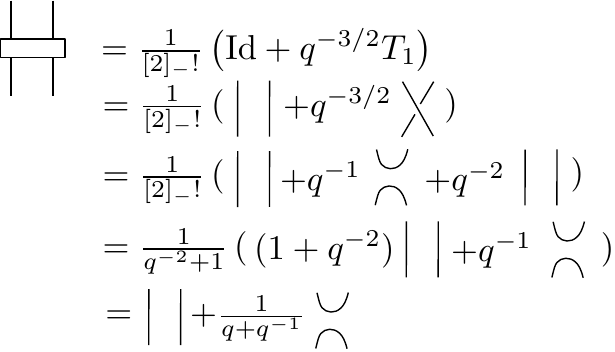}
\end{figure}

Proposition \ref{JW-prop} and Theorem \ref{JW-thm} translate into Figures \ref{JW-prop-fig} and \ref{JW-thm-fig}, respectively.

\begin{figure}[htbp]
\centering
\includegraphics[scale=1]{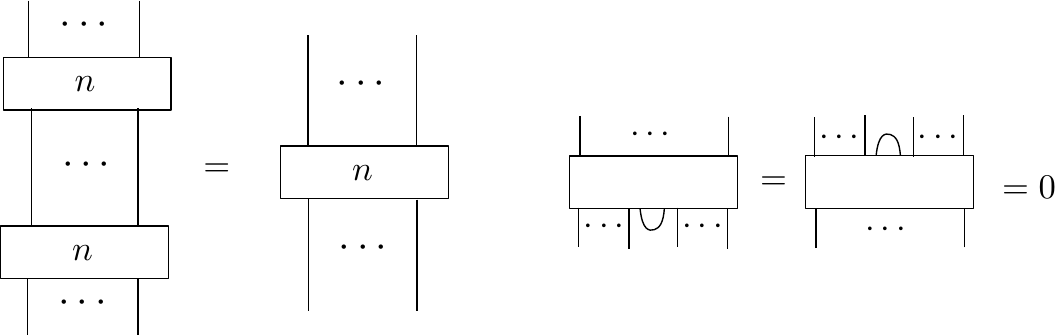}
\caption{Graphical relations involving projectors.}
\label{JW-prop-fig}
\end{figure}

\begin{figure}[htbp]
\centering
\includegraphics[scale=1]{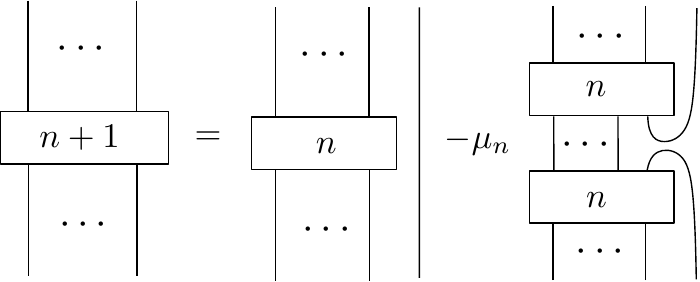}
\caption{Graphical description of inductive property of Jones--Wenzl projectors.}
\label{JW-thm-fig}
\end{figure}

\subsection{Lusztig canonical and dual canonical bases}
Lusztig defined a canonical basis in tensor products of irreducible representations of quantum groups. In particular, we consider tensor powers of the fundamental representation $V_1$ of $U_q(\mathfrak{sl}_2)$. A diagrammatic construction of the dual of his basis in $(V_1^{\otimes n})^{\ast} \cong V_1^{\otimes n}$ was obtained by Frenkel and Khovanov in \cite{khfrenkel}. In \cite{khthesis}, Khovanov gives an explicit formula for the Lusztig canonical basis using a diagrammatic approach. He describes the graphical interpretation of the natural bilinear form on $V_1 ^{\otimes n}$ and constructs the duals to the Lusztig dual canonical basis elements with respect to this form. The construction of the duals uses Jones--Wenzl projectors, and the details are reviewed in Chapter \ref{ch:Dual-Russell}. Also in that chapter, we extend Khovanov's graphical Lusztig canonical basis for the invariant subspace $\mbox{Inv}(V_1^{\otimes 2n})$ under the $U_q(\mathfrak{sl}_2)$ action in the case $q=-1$ to a basis in the Russell skein module.

%changes: fixed error in coefficient of middle case of main theorem/lemma

%\documentclass[12pt]{article}
%\usepackage[margin = 1in]{geometry}                % See geometry.pdf to learn the layout options. There are lots.
%\geometry{letterpaper}                   % ... or a4paper or a5paper or ... 
%\geometry{landscape}                % Activate for for rotated page geometry
%\usepackage[parfill]{parskip}    % Activate to begin paragraphs with an empty line rather than an indent
%\usepackage{amsmath}
%\usepackage{graphicx}
%\usepackage{amssymb}
%\usepackage{epstopdf}
%\usepackage{epsfig}
%\usepackage{graphicx}
%\usepackage{cite}
%\usepackage{url}
%\usepackage{color}
%\usepackage{comment}
%\usepackage{float}
%\usepackage{marvosym}
%\usepackage{afterpage,lipsum}
%\usepackage{setspace}

%\DeclareGraphicsRule{.tif}{png}{.png}{`convert #1 `dirname #1`/`basename #1 .tif`.png}

%\input{xy}
%\xyoption{all}

\chapter{Dualizing the Russell skein module}
\label{ch:Dual-Russell}

\begin{comment}
\section{Introduction}

In \cite{khov}, Khovanov defines the arc rings $H^n$, which play a fundamental role in the extension of Khovanov homology from links to tangles. Heather Russell has a graphical calculus for the 0th Hochschild homology $HH_0(H^n)$ consisting of crossingless matchings with dots subject to certain relations. A graphical basis for $HH_0(H^n)$ can be obtained by taking Russell elements of a certain form. We would like to create a basis that is dual to hers, thereby giving a graphical basis for $HH^0(H^n) \cong Z(H^n)$.
\end{comment}

\section{The dual Russell space}

\subsection{A bilinear form on the Russell space}
\label{sec:bilin_form}

Recall that $R_n$ denotes the $\mathbb{Z}$-module spanned by crossingless matchings of $2n$ points with at most one dot on each arc subject to Russell's Type I and Type II relations described in the previous section. Elements of this space can be expressed as linear combinations of diagrams, drawn as crossingless matching cups (with dots). In this section, we define a bilinear form on the space $R_n$, that is, a map 
\[ \langle \cdot, \cdot \rangle: R_n \otimes R_n \to \mathbb{Z}. \]

We define the pairing $\langle a,b \rangle$ of diagrams $a, b$ in $R_n$ as follows. First, rotate the diagram $b$ by 180 degrees and replace each dot by an X. Then match the endpoints of the rotated diagram with the endpoints of $a$ to form a diagram whose connected components are closed circles decorated with dots and X's. We allow lines in the diagram to move up to isotopy, so we may deform our circles to being round, and we allow dots and X's to slide freely around the circle they are on but not past one another. Finally, evaluate this diagram to an integer by multiplying the evaluations of each component, defined according to the following rules: 
\begin{itemize}
\item A closed circle with no dots or X's evaluates to 2.
\item A closed circle with the same number of dots and X's (at least one of each) arranged in alternating fashion evaluates to 1.
\item All other circles evaluate to 0.
\end{itemize}

Figure \ref{pairing1} shows the evaluation of circles in which no more than one dot and one X are present.

\begin{figure}[ht]
   \centering
   \includegraphics[scale=.7]{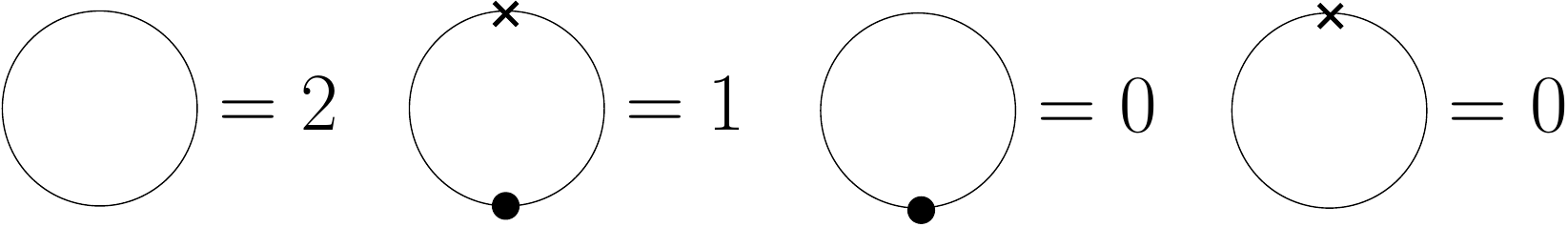} 
   \caption{The evaluation of diagrams with at most one dot and one X.}
   \label{pairing1}
\end{figure}

Given the above definition of the pairing of diagrams of $R_n$, we extend to a bilinear form on all of $R_n$ by linearity.

\begin{lemma}
\label{nondeg-sym}
The bilinear form $\langle \cdot, \cdot \rangle$ is nondegenerate and symmetric.
\end{lemma}

\begin{proof}
It is clear that $\langle \cdot, \cdot \rangle$ is nondegenerate because $\langle a, \overline{a} \rangle$, where $\overline{a}$ is the horizontal reflection of $a$, is non-zero for any diagram $a \in R_n$.

To see that $\langle \cdot, \cdot \rangle$ is symmetric, for $a, b$ diagrams in $R_n$, consider the closed circles constructed in the definitions of $\langle a, b \rangle$ and $\langle b, a \rangle$ The closed diagrams are identical except that they have been reflected across the horizontal axis and have dots and X's interchanged. Observe that the evaluation rules in the definition of $\langle \cdot, \cdot \rangle$ are preserved under these transformations.
\end{proof}

\begin{proposition}
The bilinear form $\langle \cdot, \cdot \rangle$ is well-defined on $R_n$.
\end{proposition}

\begin{proof}
We must check that the bilinear form respects all Type I and Type II Russell relations. That is, given a relation $r=0$ in $R_n$, we must have
\[ \langle r, a \rangle = 0 \]
for all $a \in R_n$. We begin with the assumption that $r$ is a relation of Type I and consider the pairing of $r$ with an arbitrary diagram $a$. Recall that in the definition of $\langle r, a \rangle$, we first reflect $a$, change dots of $a$ to X's, and then for each diagram in $r$ match endpoints to form a linear combination of collections of closed circles. Let $i < j < k < l$ be the endpoints involved in the Type I relation of $r$, so that the diagrams on the lefthand side of Figure \ref{Russell_relns} have arcs between points $(i,j)$ and $(k,l)$, while the diagrams on the righthand side have arcs between points $(i,l)$ and $(j,k)$, with all other arcs among all four diagrams are identical.

When any diagram in $r$ gets matched up with $a$, either all four points $i,j,k,l$ will lie on a single closed circle, or they will lie on two distinct circles. We assume that for the diagrams in $r$ on the lefthand side of the Type I relation of Figure \ref{Russell_relns}, endpoints $i, j, k, l$ will lie on a single circle, whereas for the diagrams on the righthand side, $i, l$ will lie on a circle distinct from the one on which $j,k$ lie. The opposite scenario will be completely symmetric.

Now let $d$ be a diagram in $r$. The pairing $\langle d, a \rangle$ will be 0 unless $d$ and $a$ have the same number of dots. We may ignore any closed circles in $\langle d, a \rangle$ not including the endpoints $i, j, k, l$, since they will be identical in all four diagrams. Up to symmetry there are then four cases to check, shown in Figure \ref{bil_form_welldef}.

\begin{figure}[ht]
   \centering
   \includegraphics[scale=.7]{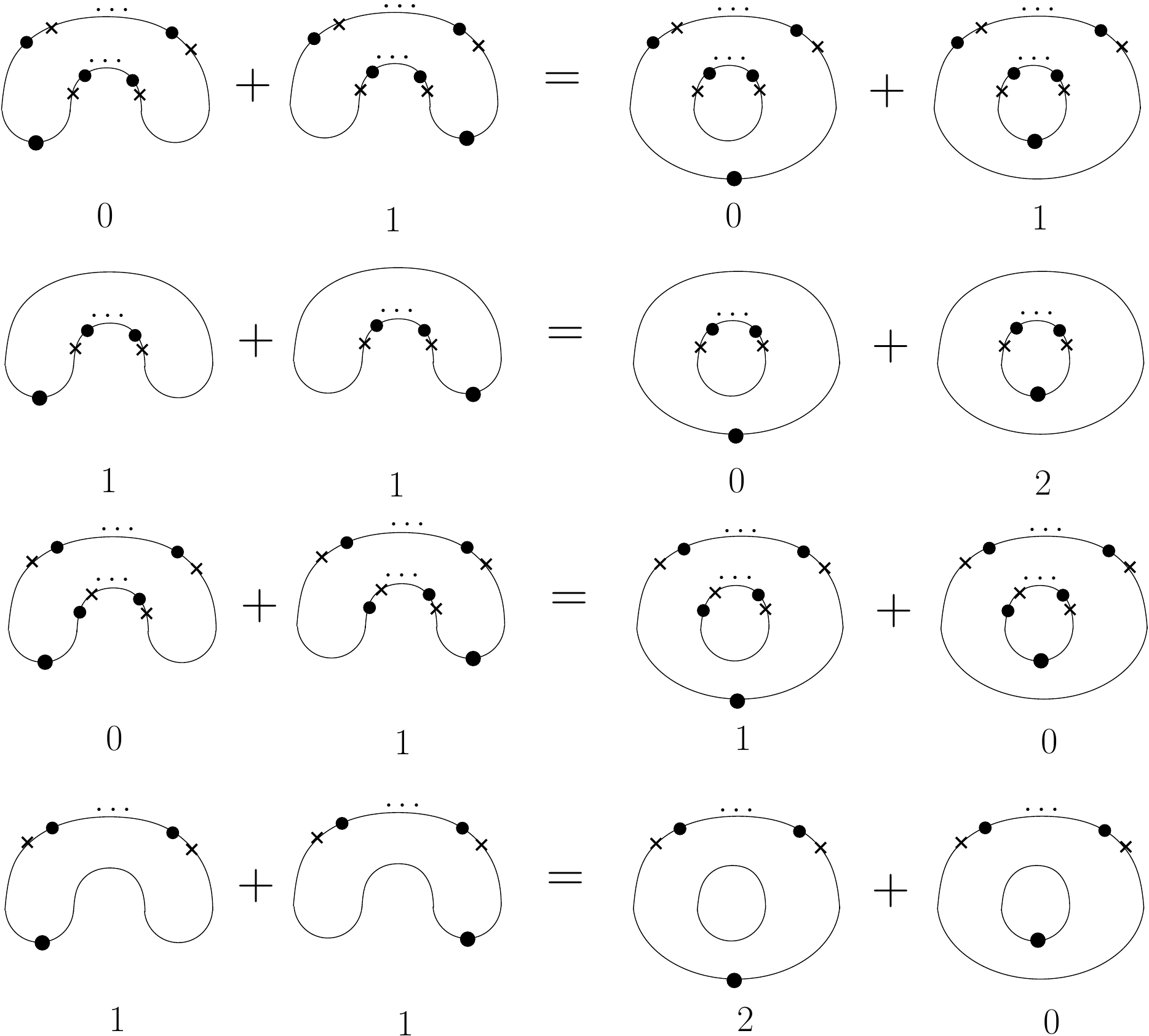} 
   \caption{The bilinear form respects Type I relations.}
   \label{bil_form_welldef}
\end{figure}

We may repeat a similar argument for $r$ of Type II, with only one case to check, shown in Figure \ref{bil_form_welldef2}.  $\blacksquare$

\begin{figure}[ht]
   \centering
   \includegraphics[scale=.8]{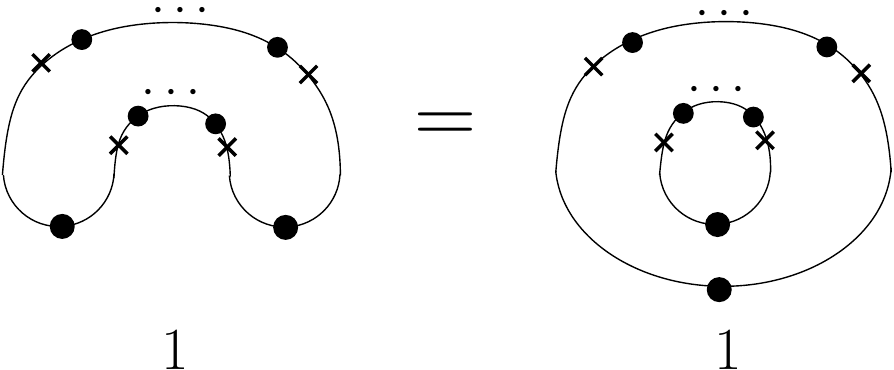} 
   \caption{The bilinear form respects Type II relations.}
   \label{bil_form_welldef2}
\end{figure}
\end{proof}
 
Note that the definition of the bilinear form $\langle \cdot, \cdot \rangle$ is equivalent to imposing the local relations between dots and X's shown in Figure \ref{linerelns}.

\begin{figure}[ht]
   \centering
   \includegraphics[scale=1]{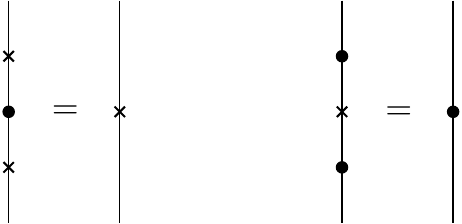}
   \caption{Local relations defining the bilinear form on $R_n$.}
   \label{linerelns}
\end{figure}

Consider a $\mathbb{Z}$-module $A$ consisting of linear combinations of diagrams, where a diagram is a vertical line decorated with dots and X's subject to the local relations of Figure \ref{linerelns}. Multiplication of diagrams in $A$ is given by vertical stacking. Then $A$ is 5-dimensional over $\mathbb{Z}$ and isomorphic to $\mathbb{Z} \langle a,b \rangle /(a^2 = b^2 = 0, aba=a, bab=b)$, where $a$ corresponds to a dot on the line and $b$ corresponds to an X. As a free abelian group, $A$ then has a basis given by $\{ 1, a, b, ab, ba\}$. This module has a complete set of three primitive orthogonal idempotents given by $\{ab, ba, 1-ab-ba\}$, and the associated quiver is shown in Figure \ref{quiverpres}, with the relation that the composition of any two compatible arrows is the identity map. Note it follows that
\[ A \cong \mathbb{Z} \times \mbox{Mat}(2, \mathbb{Z}).\]

\begin{figure}[h]
   \centering
   \includegraphics[width = 2in]{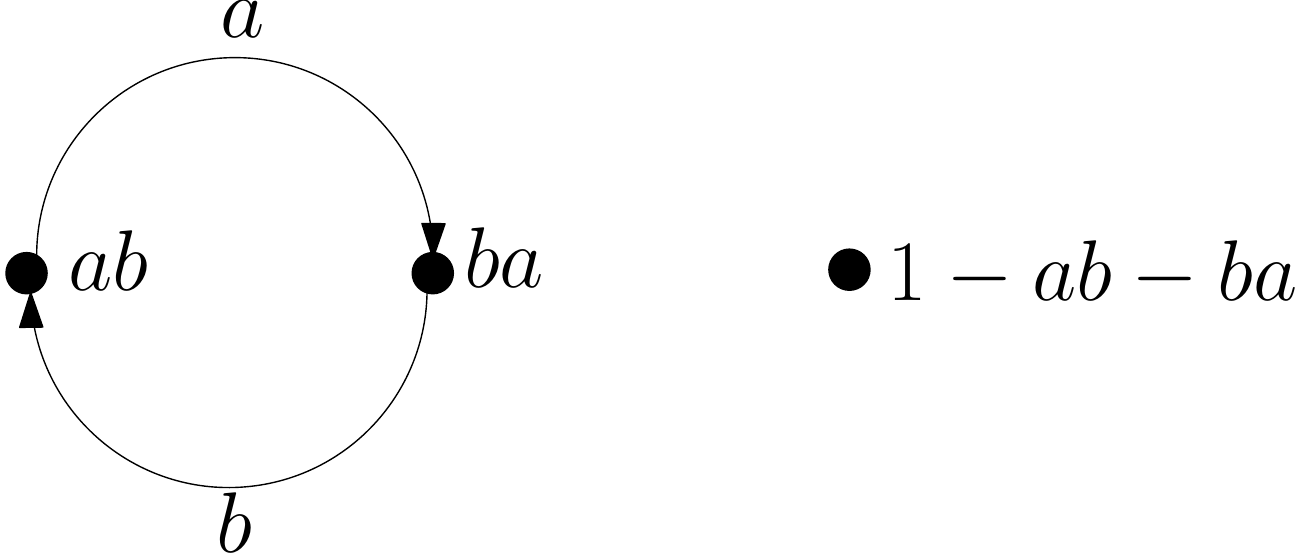}
   \caption{The quiver of an algebra defined by the local relations of Figure \ref{linerelns}.}
   \label{quiverpres}
\end{figure}

\begin{conjecture}
This bilinear form $\langle \cdot, \cdot \rangle$ is positive definite.
\end{conjecture}

\subsection{A graphical calculus for the dual space}
\label{sec:dual-calc}
We now turn our attention to the dual Russell space. Before doing so, we must first be working over a field, so we change our ground ring from $\mathbb{Z}$ to $\mathbb{Q}$ and define the dual space $R_n^{\ast}$ to be the vector space of linear functionals $R_n \to \mathbb{Q}$.

By Lemma \ref{nondeg-sym}, the bilinear form $\langle \cdot, \cdot \rangle: R_n \otimes R_n \to \mathbb{Q}$ is nondegenerate, so it defines an isomorphism from $R_n$ to $R_n^{\ast}$ given by tŒhe map
\[ x \mapsto \langle \cdot, x \rangle.\]

Diagrammatically, we draw the linear functional $\langle \cdot, x \rangle \in R_n^{\ast}$ as the 180 degree rotation of $x$ with dots changed to X's, consistent with our definition of the bilinear form. Therefore in general we consider $R_n^{\ast}$ to be a vector space spanned by crossingless matching caps of $2n$ endpoints decorated with $X$'s, subject to the relations that a diagram with two adjacent X's is zero and and the analogues of the local Type I and Type II Russell relations shown in Figure \ref{dual_space_relns}.

\begin{figure}[ht]
   \centering
   \includegraphics[scale=.8]{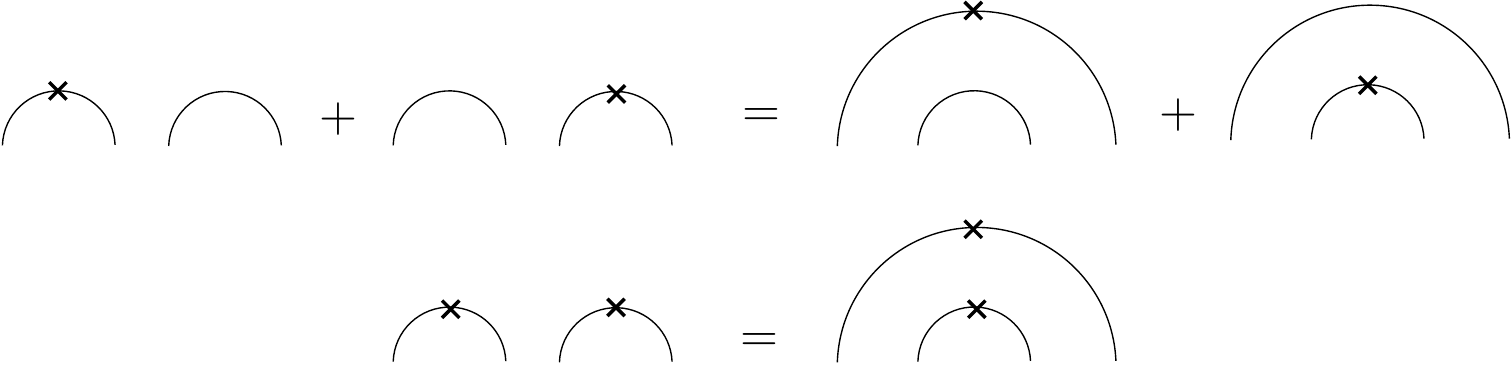}
   \caption{Local relations in $R_n^{\ast}$.}
   \label{dual_space_relns}
\end{figure}

For instance, $R_1^{\ast}$ is the vector space with basis given by the diagrams of Figure \ref{dualn1}.
\begin{figure}[ht] 
   \centering
   \includegraphics[scale=1]{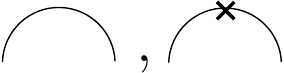} 
   \caption{Basis elements of $R_1^{\ast}$.}
   \label{dualn1}
\end{figure}

\section{Finding the dual basis}
In \cite{russ-2}, Russell and Tymoczko define ``standard'' dotted crossingless matchings to be those that have dots on outer arcs only and show that this set of diagrams forms a basis of $R_n$. Observe that by applying Type I and Type II relations to elements of Russell's graphical algebra, any element can be expressed in such a way that dots are only carried by outer arcs of a crossingless matching diagram. If a dotted arc is nested inside an undotted arc, then Type I relations can express the diagram as a linear combination of diagrams where dots only appear on outer arcs. If two dotted arcs are nested inside an undotted arc, a Type II relation can be applied.
\begin{definition}
The \emph{Russell basis} of $R_n$ is defined to be the graphical basis consisting of dotted crossingless matchings such that dots may only appear on outer arcs.
\end{definition}

Given the Russell basis of $R_n$ and the graphical calculus for the dual space $R_n^{\ast}$, we would like to construct the dual basis. We do this by first finding the duals of crossingless matchings without dots by specializing work of Frenkel and Khovanov and then by extending their result to dualize diagrams that do have dots present.

\subsection{Dual basis elements without dots}
\label{sec:no-dots}

In \cite{khfrenkel}, Frenkel and Khovanov give a graphical presentation for the dual of the Lusztig canonical basis of any tensor product $V_{a_1} \otimes \cdots \otimes V_{a_n}$ of irreducible finite-dimensional representations of $U_q(\mathfrak{sl}_2)$. They also give an explicit formula for the canonical Lusztig basis stemming from an inductive construction of duals to these graphical basis elements using Jones--Wenzl projectors. We will specialize their work to construct graphical duals of those Russell basis elements that do not carry dots.

\begin{comment}
Khovanov and Frenkel constructed dual elements to crossingless matching diagrams without dots in \cite{khfrenkel}. More generally, they constructed dual elements to a canonical basis of an arbitrary tensor product $V_{a_n} \otimes \cdots \otimes V_{a_1}$ of irreducible representations of $U_q(\mathfrak{sl}_2)$. We review this general construction and explain how it can be restricted to give graphical elements dual to crossingless matchings without dots. 

In the special case where each $V_{a_i} = V_1$, basis elements of the subspace $\mbox{Inv}(V_1^{\otimes n})$ are simply crossingless matching diagrams. 

Taking each $V_{a_i} = V_1$ and projecting onto the space of invariants, they obtain a graphical basis of $Inv(V_1^{\otimes n})$. 

We first construct basis elements dual to crossingless matching cups without dots. We borrow the standard graphical calculus of $U_q(\mathfrak{sl}_2)$ intertwiners, following \cite{KL}, and take $q = -1$ so that the circle evaluates to 2, consistent with our above definitions.
\end{comment}

We will use the standard graphical calculus of $U_q(\mathfrak{sl}_2)$ intertwiners, as reviewed in Section \ref{sec:graphical_calculus}. 
For now we work with generic $q$, though ultimately we will only need the specialization of these results to $q=-1$.
\begin{comment}
We take $q = -1$ so that the circle evaluates to 2, as it does in the definition of our bilinear form on $R_n$.
In particular, we adopt the conventions of Figure 
\end{comment} 

Recall that $V_1$ denotes the two-dimensional fundamental representation of $U_q(\mathfrak{sl}_2)$. Since this is the only representation we will be working with, we drop the subscript and just write $V$. The two standard basis vectors of the dual space will be denoted by $v^1$ and $v^{-1}$. Diagrammatically, we will depict $v^1$ as an up arrow and $v^{-1}$ as a down arrow. A vector $v^{\epsilon_1} \otimes \cdots \otimes v^{\epsilon_n}$ in the tensor product $V^{\otimes n}$, where $\epsilon_i = \pm 1$, will then be depicted as a horizontal sequence of up and down arrows numbered 1 through $n$ from left to right, where the direction of the $i$th arrow is determined by $\epsilon_i$.

\begin{definition}
We denote by $\{ v^{\epsilon_1} \heart \cdots \heart v^{\epsilon_n} \}$, where $\epsilon_i = \pm 1$ for all $i$, the elements of the dual of Lusztig's canonical basis of $V^{\otimes n}$.
\end{definition}

Following \cite{khfrenkel}, an element $v^{\epsilon_1} \heart \cdots \heart v^{\epsilon_n}$ of the dual canonical basis of $V^{\otimes n}$ admits a geometric interpretation. First, draw $v^{\epsilon_1} \otimes \cdots \otimes v^{\epsilon_n}$ as a sequence of up/down arrows as described above. Next, if the diagram contains a pair (up arrow, down arrow) such that
\begin{itemize}
\item the up arrow is to the left of the down arrow
\item no arrows lie between the two arrows,
\end{itemize}
then we connect the two arrows into a simple unoriented arc that does not intersect anything. Finally, we repeat this procedure until no such pairs of arrows remain.

Recall that in the usual graphical calculus of $U_q(\mathfrak{sl}_2)$, an arc represents the element of the standard basis given by applying the local relation of Figure \ref{cup_reln}.
\begin{figure}[htbp]
   \centering
   \includegraphics[scale=1.1]{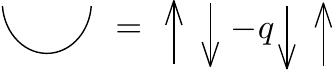}
   \caption{A relation in the graphical calculus of $U_q(\mathfrak{sl}_2)$ representations.}
   \label{cup_reln}
\end{figure}

\begin{example}
An example of a vector and its corresponding dual canonical basis vector in $V^{\otimes 6}$ is shown in Figure \ref{dualCanonBasisEx}.

\begin{figure}[ht]
   \centering
   \includegraphics[width = 4in]{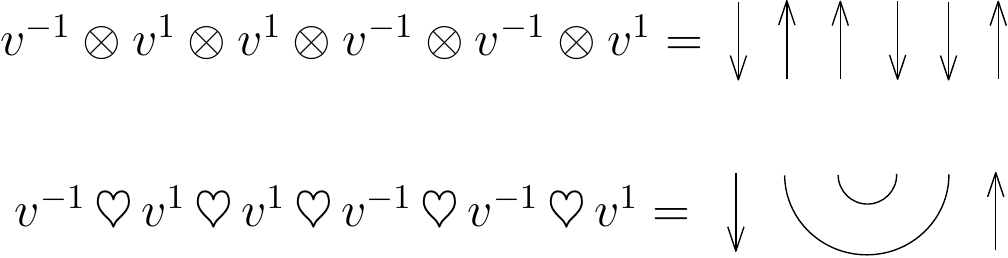}
   \caption{A sample element of the dual canonical basis of $V^{\otimes 6}$.}
   \label{dualCanonBasisEx}
\end{figure}
\end{example}

One can define the actions of $E, F,$ and $K^{\pm 1}$ on elements of the dual canonical basis graphically (see \cite{khfrenkel} Section 2.3). From this description, it is clear that the elements of the dual canonical basis of $V^{\otimes n}$ that are invariant under the action of $U_q(\mathfrak{sl}_2)$ are those which have only arcs and no arrows. On the algebraic side, this means that a dual canonical basis vector $v^{\epsilon_1} \heart \cdots \heart v^{\epsilon_n}$ in $V^{\otimes n}$ is invariant under the $U_q(\mathfrak{sl}_2)$ action if and only if $\sum_{i=1}^n \epsilon_i = 0$ and for any $i, 1 \leq i \leq n$, we have $\sum_{j=1}^i \epsilon_j \geq 0$. This establishes a bijection
\[
\mbox{crossingless matchings of } 2k \mbox{ points} \leftrightarrow \mbox{dual canonical basis of } \mbox{Inv}(V^{\otimes 2k}).
\]

\begin{comment}
Dual canonical vectors $v_{\epsilon_n} \heart \cdots \heart v_{\epsilon_1}$ in $V^{\otimes n}$, where $\epsilon_1, \ldots, \epsilon_n \in \{1,-1\}$, can be constructed recursively and depicted diagrammatically as a sequence of $n$ arrows, where the $i$th arrow points up if $\epsilon_i = 1$ and down if $\epsilon_i = -1$. Given any crossingless matching, we may interpret it as an element of this dual canonical basis by assigning orientations to its endpoints so that all cups are oriented as shown below.
\begin{figure}[H]
   \centering
   \includegraphics[width = 1in]{cupOrientation.pdf} 
   \label{cupOrientation}
\end{figure}

For example, we have

\begin{figure}[H]
   \centering
   \includegraphics[width = 6in]{arcsToDual.pdf} 
   \label{cupOrientation}
\end{figure}
\end{comment}

Before discussing how to dualize this graphical basis, we must first establish a graphical description of a bilinear form on $V^{\otimes n}$ that our bases will be dual with respect to, that is, we establish a bilinear pairing
\[ \langle \cdot, \cdot \rangle : V^{\otimes n} \times V^{\otimes n} \to \mathbb{C}[q, q^{-1}]. \]

Let $y$ be a diagram representing a standard basis element $v^{\eta_1} \otimes \cdots \otimes v^{\eta_n}$ of $V^{\otimes n}$ and let $x$ be a dual canonical basis vector,
\[ x = v^{\epsilon_1} \heart \cdots \heart v^{\epsilon_n}.\]
We define
\[ \langle y, x \rangle \]
as follows. First, rotate the diagram of $x$ by 180 degrees and reverse the orientations of all arrows. Then match the top of the diagram for $y$ with the bottom of the rotated diagram and evaluate the resulting diagram locally according to the following rules. Each arc in the resulting diagram evaluates to 0 if the orientations of the two ends are not compatible and 1 if they are compatible, except in the cases of Figure \ref{bil_form_Vn}.

\begin{figure}[ht]
   \centering
   \includegraphics[scale=.8]{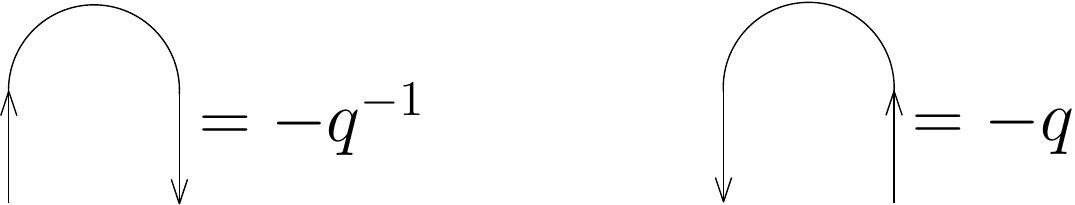}
   \caption{Local rules defining the graphical bilinear form on $V^{\otimes n}$.}
   \label{bil_form_Vn}
\end{figure}

\begin{definition}
The element $v_{\epsilon_n} \diam \cdots \diam v_{\epsilon_1}$ of $V^{\otimes n}$, where $\epsilon_i = \pm 1$ for all $i$, is defined to be the dual of $v^{\epsilon_1} \heart \cdots \heart v^{\epsilon_n}$ with respect to the bilinear form defined above.

The collection of all elements  $\{ v_{\epsilon_n} \diam \cdots \diam v_{\epsilon_1} \}$ form the \emph{canonical basis} of $V^{\otimes n}$.

\end{definition}

We denote the canonical basis element of $V^{\otimes n}$ with lexicographically highest term $v_1^{\otimes x_1} \otimes v_{-1}^{\otimes y_1} \otimes \cdots \otimes v_1^{\otimes x_k} \otimes v_{-1}^{\otimes y_k}$ by $v(x_1, y_1; \ldots; x_k, y_k)$, that is, 
\[v(x_1, y_1; \ldots; x_k, y_k) =  v_1^{\otimes x_1} \diam v_{-1}^{\otimes y_1} \diam \cdots \diam v_1^{\otimes x_k} \diam v_{-1}^{\otimes y_k}.\]

The following lemma is due to Khovanov in \cite{khthesis}.
\begin{lemma}[Khovanov]
\begin{enumerate}
\item $v(0,y_1; x_2, \ldots; x_k, y_k) = v_{-1}^{\otimes y_1} \otimes v(x_2, y_2; \ldots x_k, y_k)$
\item $v(x_1, y_1; \ldots, y_{k-1}; x_k, 0) = v(x_1, y_1; \ldots ; x_{k-1}, y_{k-1}) \otimes v_1^{x_k}$
\end{enumerate}
\end{lemma}

Also in \cite{khthesis}, Khovanov gives the following inductive description of these canonical basis elements.
\begin{theorem}[Khovanov]
\label{thm:khfrenkel}
\begin{enumerate}
\item If $y_{i-1} \geq x_i$ and $y_i \leq x_{i+1}$ then 
\begin{multline*} v(x_1, y_1; \ldots; x_k, y_k) = \left[ \begin{array}{c} x_i +y_i \\ x_i \end{array} \right] (1^{\otimes l} \otimes p_{x_i+y_i} \otimes 1^{\otimes j}) \\
(v(x_1, y_1; \ldots; x_{i-1}, y_{i-1}+y_i; x_i + x_{i+1}, y_{i+1}; \ldots; x_k, y_k))
\end{multline*}
where $l = \sum_{t < i}x_t+ y_t, j = \sum_{t>i}x_t+y_t$.
\item If $y_1 \leq x_2$ then
\begin{equation*}
v(x_1, y_1; \ldots; x_k, y_k) = \left[ \begin{array}{c} x_1 + y_1 \\ x_1 \end{array} \right] (p_{x_1+y_1} \otimes 1^{n-x_1-y_1})(v_{-1}^{\otimes y_1} \otimes v(x_1 + x_2, y_2; \ldots; x_k, y_k)).
\end{equation*}
\item If $y_{k-1} \geq x_k$ then
\begin{equation*}
v(x_1, y_1; \ldots; x_k, y_k) = \left[ \begin{array}{c} x_k+y_k \\ x_k \end{array} \right] (1^{\otimes n -x_k-y_k} \otimes p_{x_k +y_k})(v(x_1,y_1; \ldots; x_{k-1},y_{k-1}+y_k) \otimes v_1^{\otimes x_k}).
\end{equation*}
\end{enumerate}
\end{theorem}

Recall that $p_n$ denotes the Jones-Wenzl projector from $V^{\otimes n}$ to $V^{\otimes n}$, which in the usual graphical calculus is depicted by a box on $n$ strands.

At this point, we specialize to $q=-1$. In this case, the inductive relationship satisfied by Jones--Wenzl projectors (Figure \ref{JW-thm-fig}) becomes:

\begin{figure}[H]
   \centering
	\includegraphics[width = 3in]{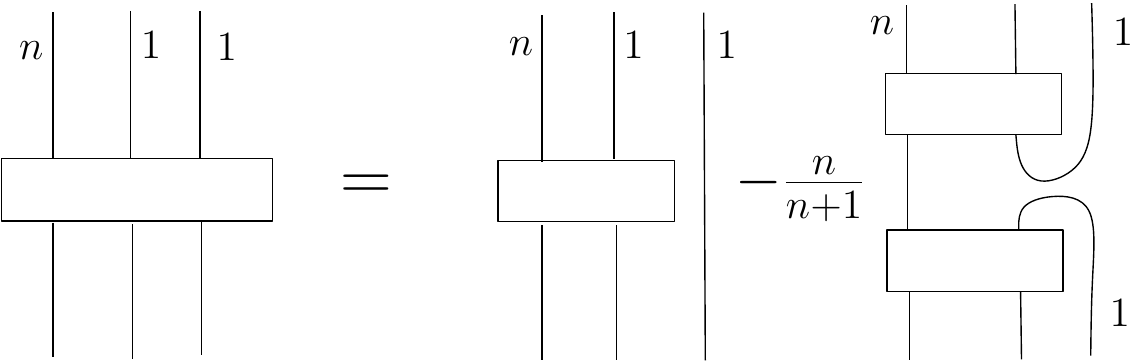} 
\end{figure}

\begin{example}
We apply Theorem \ref{thm:khfrenkel} to find a formula for the canonical basis vector \\ $ v_{-1} \diam v_{1} \diam v_{-1} \diam v_{-1} \diam v_{1} \diam v_{1}$:
 \begin{eqnarray*}
  v_{-1} \diam v_{1} \diam v_{-1} \diam v_{-1} \diam v_{1} \diam v_{1} &=&  v_{-1} \diam v_{1} \diam v_{-1}^{\otimes 2} \diam v_{1}^{\otimes 2} \\
 &= &v(0,1;1,2;2,0) \\
 &=& \left[ \begin{array}{c} 1+2 \\ 1 \end{array} \right]_{q=-1} (1^{\otimes 1} \otimes p_3 \otimes 1^{\otimes 2}) (v(0, 3; 3, 0) )\\
 &=& -3 (1^{\otimes 1} \otimes p_3 \otimes 1^{\otimes 2}) (v_{-1}^{\otimes 3} \otimes v_1^{\otimes 3})
 \end{eqnarray*}
 
 See Figure \ref{canonBasisEx} for a graphical description.
 
 \begin{figure}[ht]
   \centering
   \includegraphics[width = 1.5in]{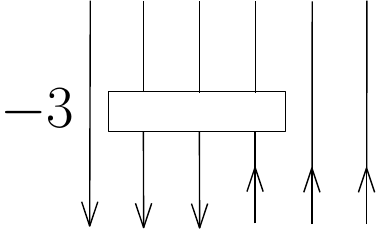}
   \caption{A graphical description of the canonical basis element $v(0,1;1,2;2,0).$}
   \label{canonBasisEx}
   \end{figure}
\end{example}

To get dual elements to our crossingless matchings of $2n$ endpoints, we would like to project our resulting vector onto the space of invariants $\mbox{Inv}(n) = \mbox{Inv}(V^{\otimes 2n})$. As outlined in \cite{khthesis} Section 3.3, this can be achieved by simply attaching a projector of size $n$ to the diagram so that all down arrows enter one side and all up arrows leave the other.

Note that when the projector is expanded in such an element of $\mbox{Inv}(n)$, we end up with a linear combination of arcs. The bilinear pairing of two collections of arcs evaluates each closed circle (up to isotopy) to $-q-q^{-1}$, according to the standard Kauffman-Lins relation. In particular, when $q=-1$, a circle evaluates to $2$, which is consistent with our bilinear form on the Russell space $R_n$.

From this point forward, when discussing duals of elements in $\mbox{Inv}(n)$, we will draw the duals as already having been rotated by 180 degrees (as is done in the definition of the bilinear form on $V^{\otimes n}$. This convention will be consistent with our graphical description of the dual Russell space, whose diagrams have crossingless matching caps.

\begin{comment}
our dual element to be closed at the bottom, so we attach a projector between the down arrows and up arrows. This does not affect the duality of our element, since a projector $p_n$ with $n$ up arrows above and $n$ down arrows below or vice versa evaluates to 1.
\end{comment}

\begin{example}
Building off of our previous example, we graphically compute the dual to the crossingless matching shown in Figure \ref{crossMatchEx} with the specialization $q=-1$. We view this matching as an element of $\mbox{Inv}(V^{\otimes 6})$ corresponding to the dual canonical basis element $v^1 \heart v^{1} \heart v^{-1} \heart v^{-1} \heart v^{1} \heart v^{-1}.$
\begin{figure}[ht]
   \centering
   \includegraphics[width = 1.5in]{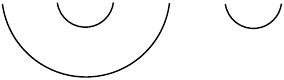}
   \caption{The element  $v^1 \heart v^{1} \heart v^{-1} \heart v^{-1} \heart v^{1} \heart v^{-1}$ of $\mbox{Inv}(V^{\otimes 6})$.}
   \label{crossMatchEx}
\end{figure}

The dual of this element is $ v_{-1} \diam v_{1} \diam v_{-1} \diam v_{-1} \diam v_{1} \diam v_{1}$, which was computed in the previous example. So we need only project to the invariant space by attached a projector, rotate the diagram by 180 degrees, and simplify the projectors. This is shown in Figure \ref{dualMatchingEx}.

\begin{figure}[H]
   \centering
   \includegraphics[scale=1]{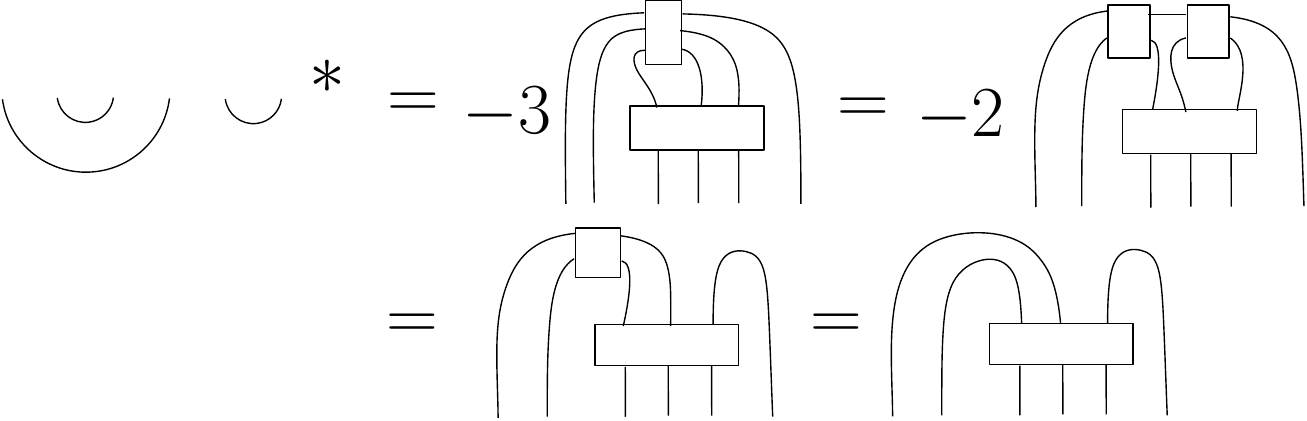}
   \caption{An example of the application of Theorem \ref{thm:khfrenkel} to find the graphical dual of a dottingless crossingless matching.}
   \label{dualMatchingEx}
\end{figure}
\end{example}

Figure \ref{dualBasis23} shows the dual basis elements for all crossingless matchings when $n=2$ and $n=3$.
\begin{figure}[ht]
   \centering
   \includegraphics[scale=.8]{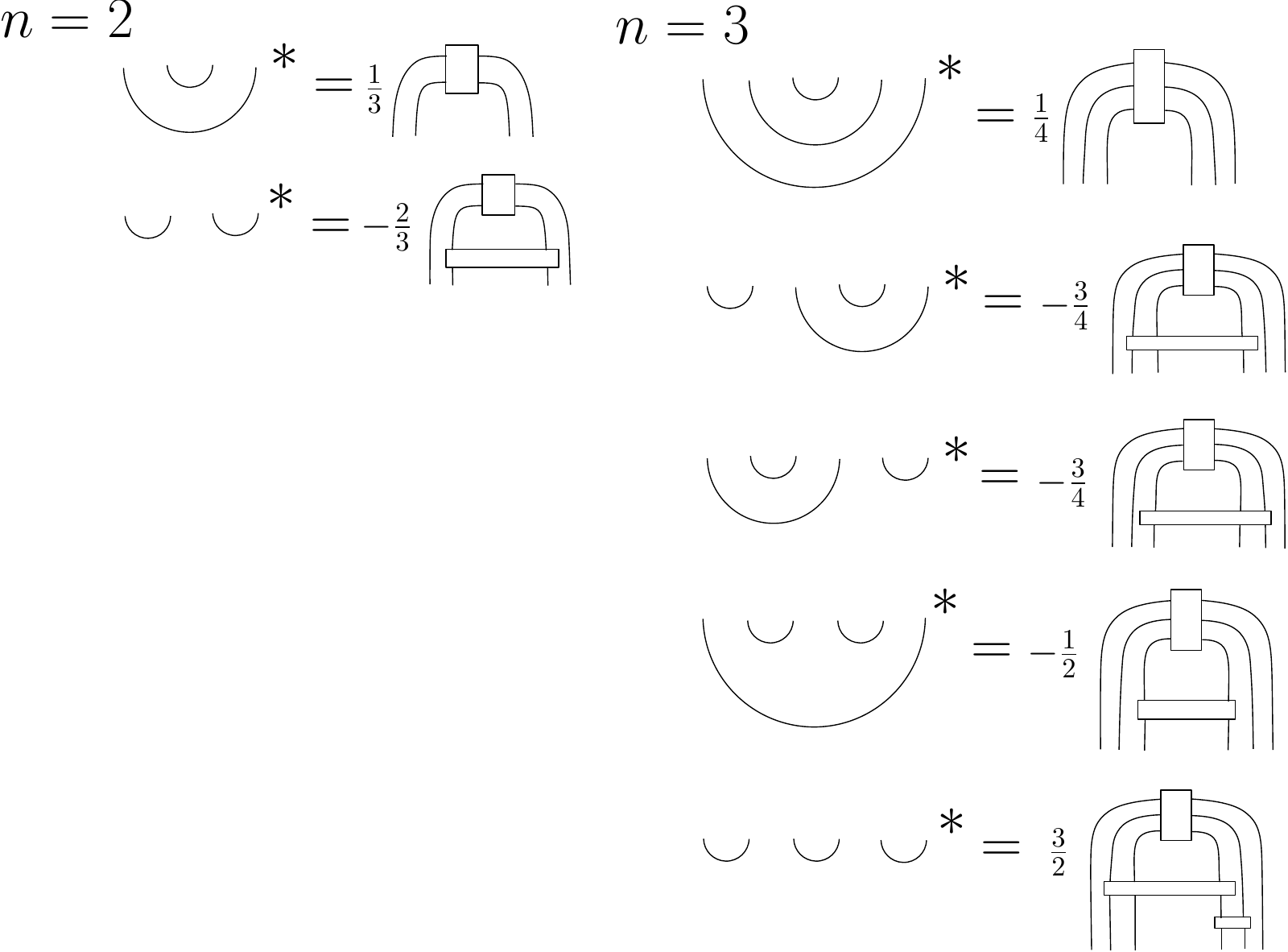}
   \caption{Crossingless matchings and their duals for $n = 2$ and $3$, $q=-1$.}
   \label{dualBasis23} 
\end{figure}

\subsection{Dual basis elements with dots}
\label{sec:dual-dots}
We now find a graphical description of elements dual to those diagrams in the Russell basis that carry dots. Ideally, we would be able to express these dual elements as a single picture, rather than a linear combination of pictures. 

\begin{comment}
When $n=2$, the following is a dual basis:
\begin{figure}[H]
   \centering
   \includegraphics[width =2.5in]{dualWDotsn=2.pdf}
   \caption{Dual graphical basis with dots for $n=2$}
\end{figure}
In general, we can say the following.
\begin{lemma}
The element 
\begin{figure}[H]
   \centering
   \includegraphics[width =2.5in]{singleDotNestedDual.pdf}
\end{figure}
is dual to
\begin{figure}[H]
   \centering
   \includegraphics[width =2in]{nestedArcsSingleDot.pdf}
\end{figure}
\end{lemma}

We will now describe how to write down the dual to any element of Russell's basis.
\end{comment}

\begin{comment}
\begin{definition}
An \emph{outer arc} of a Russell basis element is one that is not contained in any other arc.  A \emph{component} of a Russell basis element consists of all arcs contained within an outer arc.
\end{definition}

\begin{definition}
A \emph{nested} Russell basis element is one such that each component has all of its arcs completely nested in one another.
\end{definition}
\end{comment}

First we need to establish some notation. Any crossingless matching can be uniquely represented by a sequence of up and down arrows by replacing the left endpoint of any arc by a down arrow and the right endpoint by an up arrow. A Russell basis element may carry dots on outer arcs of a crossingless matching. For each dot on an outer arc, place a dot on the corresponding up and down arrows. Then, a Russell basis element may be denoted by
\[ (x_1^{\epsilon_{x_1}}, y_1^{\epsilon_{y_1}}; \ldots; x_k^{\epsilon_{x_k}}, y_k^{\epsilon_{y_k}}),\]
where each $(x_i, y_i)$ represents the number of consecutive down arrows and up arrows, numbered from left to right, and the exponents $\epsilon_{x_i}$ or $\epsilon_{y_i}$ are 0 if the corresponding outermost arc does not carry a dot and is 1 if it does. In recovering a Russell basis element from such a tuple, there is no ambiguity over which arcs carry the dots since only outer arcs may carry them.

\begin{example}
For example, the Russell basis element in Figure \ref{RussellBasisNotation} is denoted by \\
$(2^0, 1^0; 1^0, 2^0; 1^1, 1^1; 2^1, 2^1)$.

\begin{figure}[ht]
   \centering
   \includegraphics[scale=1]{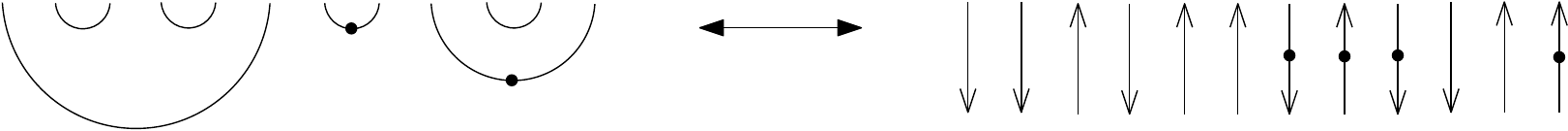}
   \caption{Example of notation for a Russell basis element.}
   \label{RussellBasisNotation}
\end{figure}
\end{example}

\begin{definition}
For a Russell basis element $m= (x_1^{\epsilon_{x_1}}, y_1^{\epsilon_{y_1}}; \ldots; x_k^{\epsilon_{x_k}}, y_k^{\epsilon{y_k}})$, we define an integer $r(m)$ as follows:
\[r(m) = k +1 - \gamma - \delta,\]
where
$\gamma = \#\{(x_i^{\epsilon_{x_i}}, y_i^{\epsilon_{y_i}}) = (1^1, 1^1)\} $ and
\begin{equation*}
\delta = \left\{
\begin{array}{ll}
0 & \mbox{if } \epsilon_{x_1} = \epsilon_{y_k} = 1 \\
1 & \mbox{if one of the following holds: } \\
& \bullet \mbox{ } k > 1, \epsilon_{x_1} = \epsilon_{y_k} = 0, \mbox{ and } \sum_{i=1}^j x_i \neq \sum_{i=1}^j y_i \mbox{ for any } j < k \\
 & \bullet \mbox{ } k > 1\mbox{ and one but not both of } \epsilon_{x_1}, \epsilon_{y_k} \mbox{ is } 1 \\
 & \bullet \mbox{ } k = 1 \mbox{ and } \epsilon_{x_1} = 0 \\
 2 &  \mbox{ if } k >1, \epsilon_{x_1} = \epsilon_{y_k} = 0, \mbox{ and } \sum_{i=1}^j x_i = \sum_{i=1}^j y_i \mbox{ for some } j < k.
\end{array}
\right.
\end{equation*}

\end{definition}

Continuing the above example, we see that
\[r((2^0, 1^0; 1^0, 2^0; 1^1, 1^1; 2^1, 2^1)) =  4+1-1-1 = 3.\]

Let $m^*$ be the dual basis element dual to $m$ with respect to the bilinear form of Section \ref{sec:bilin_form}.

\begin{lemma}
\label{basecase}
If $m$ is a Russell basis element such that $r(m)=1$, then $m$ must be of the form
\[m = (x_1^0, y_1^0; 1^1, 1^1; \ldots, 1^1, 1^1; x_{k}^0, y_k^0)\]
where $x_1 = y_1$ and $x_k = y_k$. In this case, $m^*$ is as shown (up to scale) in Figure \ref{basecasefig}.
\end{lemma}

\begin{figure}[H]
   \centering
   \includegraphics[scale=.7]{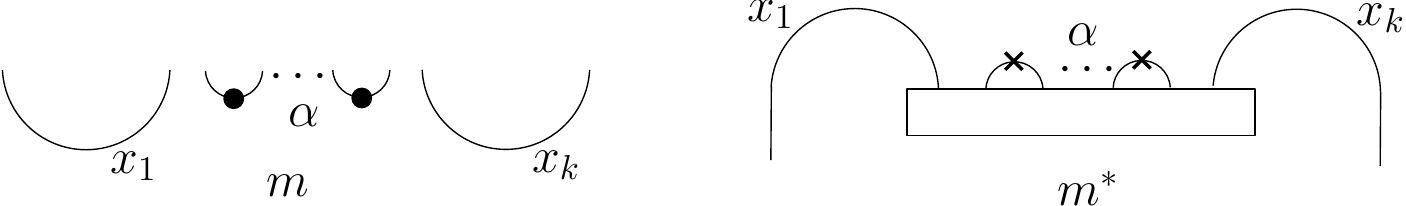}
   \caption{$m$ and $m^*$ when $r(m) = 1$. The labels $x_1$ and $x_k$ denote the number of parallel arcs represented by the single arc shown.}
      \label{basecasefig}
\end{figure}

\begin{proof}
First we show that any $m = (x_1^{\epsilon_{x_1}}, y_1^{\epsilon_{y_1}}; \ldots; x_k^{\epsilon_{x_k}}, y_k^{\epsilon_{y_k}})$ must be of the stated form. In order for $r(m) = 1$, we must have
\[ k =  \gamma + \delta.\]
Note that  $\delta$ can only be 0, 1, or 2. First suppose that it is 0. Then $k = \gamma$, so it must be that $m = (1^1, 1^1; \ldots, 1^1, 1^1)$. If $\delta= 1$, then $\gamma= k -1$, so we must have $m = (1^1, 1^1; \ldots 1^1, 1^1; x_k^0, y_k^0)$ with $x_k = y_k$ or $m = (x_1^0, y_1^0; 1^1, 1^1; \ldots ;1^1, 1^1)$ with $x_1 = y_1$. Finally, if $\delta = 2$, then $\gamma = k -2$ and $m = (x_1^0, y_1^0; 1^1, 1^1; \ldots; 1^1, 1^1; x_k^0, y_k^0)$ with $x_1 = y_1$ and $x_k = y_k$. Therefore in any of these three cases, $m$ is of the stated form.

Denote the element shown on the right in Figure \ref{basecasefig} by $b$. By abuse of notation, we will use $\langle m, b \rangle$ to denote the evaluation of the diagram obtained by matching the endpoints of $m$ with the endpoints of $b$ and evaluating according to the rules of Section \ref{sec:bilin_form}. To see that $m$ is the unique element of the Russell basis such that $\langle m, b \rangle \neq 0$, first note that any Russell basis element must have exactly $\alpha$ dots in order to pair with $b$ to something nonzero. Any Russell basis element with undotted cups within the intervals $(1, x_1), (x_1+1, 2x_1+2\alpha+x_k), (2x_1+2\alpha+x_k+1, 2x_1 +2\alpha+2x_k)$ will pair to 0 with $b$ because of the property of Jones-Wenzl projectors that

\begin{figure}[H]
   \centering
   \includegraphics[scale=1]{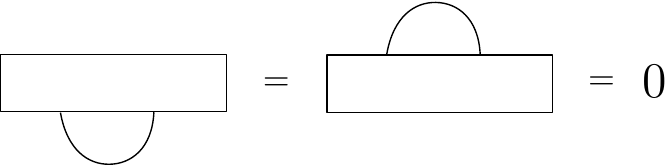}
   %\caption{$m$ and $m^*$ when $r(m) = 1$. The labels $x_1$ and $x_k$ denote the number of parallel arcs represented by the single arc shown.}
      %\label{basecasefig}
\end{figure}

It is clear that, up to scale, $m$ is the only diagram that has $\alpha$ dots and no undotted cups within those intervals.
$\blacksquare$
\end{proof}

The following theorem gives an algorithm for constructing any $m^*$ and is proved by induction on $r(m)$.
\begin{theorem} Any dual basis element with $r > 1$ can be constructed inductively by applying one of the following.
\label{dualbasisthm}
\begin{enumerate}
\item
\begin{multline*}
(x_1^{\epsilon_{x_1}}, y_1^{\epsilon_{y_1}}; \ldots; x_i^{\epsilon_{x_i}}, y_i^0; x_{i+1}^0, y_{i+1}^{\epsilon_{y_{i+1}}}; \ldots; x_k^{\epsilon_{x_k}}, y_k^{\epsilon_{y_k}})^* =  \\ \left( \begin{array}{c}
y_i + x_{i+1} \\
y_i
\end{array}
\right)
1^{\otimes l} \otimes p_{y_i + x_{i+1}} \otimes 1^{\otimes j} (x_1^{\epsilon_{x_1}}, y_1^{\epsilon_{y_1}}; \ldots; (x_i+x_{i+1})^{\epsilon_{x_i}}, (y_i+y_{i+1})^{\epsilon_{y_{i+1}}}; \ldots; x_k^{\epsilon_{x_k}}, y_k^{\epsilon_{y_k}})^*
\end{multline*}
if $x_i \geq y_i$ and $x_{i+1} \leq y_{i+1}$.

\item
\begin{enumerate}
\item
\begin{multline*}
(x_1^{\epsilon_{x_1}}, y_1^{\epsilon_{y_1}}; \ldots; x_i^0, y_i^0; 1^1, 1^1; \ldots; 1^1, 1^1; x_{i+\alpha+1}^0, y_{i+\alpha+1}^0; \ldots, x_k^{\epsilon_{x_k}}, y_k^{\epsilon_{y_k}})^* \\ =
\left(
\begin{array}{c}
x_i + 2\alpha + y_{i+\alpha+1} \\
x_i + 2\alpha
\end{array}
\right)
1^{\otimes l} \otimes p_{y_i + 2\alpha + x_{i+\alpha+1}} \otimes 1^{\otimes j} (x_1^{\epsilon_{x_1}}, y_1^{\epsilon_{y_1}}; \ldots; x_{i-1}^{\epsilon_{x_{i-1}}}, y_{i-1}^{\epsilon_{y_{i-1}}}; \\ (x_i+x_{i + \alpha +1})^0, (y_i+y_{i + \alpha +1})^0; 1^1, 1^1; \ldots, 1^1, 1^1; x_{i+\alpha+2}^{\epsilon_{x_{i+\alpha+2}}}, y_{i+\alpha+2}^{\epsilon_{y_{i+\alpha+2}}}; \ldots; x_k^{\epsilon{x_k}}, y_k^{\epsilon_{y_k}})^*
\end{multline*}
if $x_i = y_i$, $x_{i+\alpha+1} = y_{i+\alpha+1}$, and $i + \alpha + 1 < k$.
\item
\begin{multline*}
(x_1^{\epsilon_{x_1}}, y_1^{\epsilon_{y_1}}; \ldots; x_i^0, y_i^0; 1^1, 1^1; \ldots; 1^1, 1^1; x_{i+\alpha+1}^0, y_{i+\alpha_1}^0; \ldots, x_k^{\epsilon_{x_k}}, y_k^{\epsilon_{y_k}})^* \\ =
\left(
\begin{array}{c}
x_i + 2\alpha + y_{i+\alpha+1} \\
x_i
\end{array}
\right)
1^{\otimes l} \otimes p_{y_i + 2\alpha + x_{i+\alpha+1}} \otimes 1^{\otimes j} (x_1^{\epsilon_{x_1}}, y_1^{\epsilon_{y_1}}; \ldots; x_{i-1}^{\epsilon_{x_{i-1}}}, y_{i-1}^{\epsilon_{y_{i-1}}}; \\ 1^1, 1^1; \ldots, 1^1, 1^1; (x_i+x_{i + \alpha +1})^0, (y_i+y_{i + \alpha +1})^0; x_{i+\alpha+2}^{\epsilon_{x_{i+\alpha+2}}}, y_{i+\alpha+2}^{\epsilon_{y_{i+\alpha+2}}}; \ldots; x_k^{\epsilon{x_k}}, y_k^{\epsilon_{y_k}})^*
 \end{multline*}
 if $x_i = y_i$, $x_{i+\alpha+1} = y_{i+\alpha+1}$, and $i>1$.
 \end{enumerate}

\item 
\begin{enumerate}
\item
\begin{multline*}
(x_1^{\epsilon_{x_1}}, y_1^{\epsilon_{y_1}}; \ldots, x_i^0, y_i^0; 1^1, 1^1; \ldots; 1^1, 1^1; x_{i+\alpha+1}^1, y_{i+\alpha+1}^1; \ldots, x_k^{\epsilon_k}, y_k^{\epsilon_k})^* = \\
\left(
\begin{array}{c}
x_i + 2\alpha +y_{i+\alpha+1}\\
y_{i+\alpha+1} -1
\end{array}
\right)
1^{\otimes l} \otimes p_{y_i + 2\alpha + x_{i+\alpha+1}} \otimes 1^{\otimes j} (x_1^{\epsilon_{x_1}}, y_1^{\epsilon_{y_1}}; \ldots; (x_{i}+x_{i+\alpha+1}-1)^0, \\ (y_{i}+y_{i+\alpha+1}-1)^0; 1^1, 1^1; \ldots; 1^1, 1^1; x_{i+\alpha+2}^{\epsilon_{x_{i+\alpha+2}}}, y_{i+\alpha+2}^{\epsilon_{y_{i+\alpha+2}}}; \ldots; x_k^{\epsilon_{x_k}}, y_k^{\epsilon_{y_k}})^* \\
\end{multline*}
if $x_i = y_i$ and $x_{i+\alpha+1} = y_{i+\alpha+1} \neq 1$, where $x_i$ and $y_i$ are taken to be 0 if $i=0$.

\item
\begin{multline*}
(x_1^{\epsilon_{x_1}}, y_1^{\epsilon_{y_1}}; \ldots, x_i^1, y_i^1; 1^1, 1^1; \ldots; 1^1, 1^1; x_{i+\alpha+1}^0, y_{i+\alpha+1}^0; \ldots, x_k^{\epsilon_k}, y_k^{\epsilon_k})^* = \\
\left(
\begin{array}{c}
x_i+2\alpha+y_{i+\alpha+1}\\
x_i -1
\end{array}
\right)
1^{\otimes l} \otimes p_{y_i + 2\alpha + x_{i+\alpha+1}} \otimes 1^{\otimes j} (x_1^{\epsilon_{x_1}}, y_1^{\epsilon_{y_1}}; \ldots; x_{i-1}^{\epsilon_{x_{i-1}}}, y_{i-1}^{\epsilon_{y_{i-1}}}; \\ 1^1, 1^1; \ldots; 1^1, 1^1; (x_i+x_{i+\alpha+1}-1)^0, (y_i+y_{i+\alpha+1}-1)^0; \ldots; x_k^{\epsilon_{x_k}}, y_k^{\epsilon_{y_k}})^* \\
\end{multline*}
if $x_i = y_i \neq 1$ and $x_{i+\alpha+1} = y_{i+\alpha+1}$, where $x_{i+\alpha+1}$ and $y_{i+\alpha+1}$ are taken to be 0 if $i+\alpha = k$.
 \end{enumerate}
\end{enumerate}
In each statement, $l = x_1 + y_1 + \cdots + x_{i-1}+y_{i-1}+x_i$ and $j = y_{i+\alpha+1}+x_{i+\alpha+2} +y_{i+\alpha+2} + \cdots + x_k + y_k$, with $\alpha$ taken to be 0 in the first case.
\end{theorem}

Graphically, Theorem \ref{dualbasisthm} says that the dual of $m$ can be inductively constructed from the dual of some $m'$ by attaching a projector of size $y_i+2\alpha+x_{i+\alpha+1}$ between the first $l$ and last $j$ strands of the diagram for $(m')^{\ast}$.

On the way to proving Theorem \ref{dualbasisthm}, we will need the following lemma.
\begin{lemma}
Each of the statements of Theorem \ref{dualbasisthm} reduces $r$ by 1, and if $r(m) > 1$ then at least one of them can be applied.
\end{lemma}

\begin{proof}
First we see that each statement reduces $r$ by 1.
\begin{itemize}
\item Statement 1 reduces $k$ by 1 and leaves $\gamma$ and $\delta$ unchanged.
\item Statement 2a reduces $k$ by 1 and leaves $\gamma$ and $\delta$ unchanged if $i>0$, and if $i=0$ it increases $\delta$ by 1 and leaves $k$ and $\gamma$ unchanged. Statement 2b is analogous.
\item Statement 3a increases $\gamma$ by 1 and leaves $k$ and $\delta$ unchanged if $i=0$, and if $i =0$ it increases each of $k, \delta,$ and $\gamma$ by 1. Statement 3b is analogous.
\end{itemize}

Next we show that one of the statements of Theorem \ref{dualbasisthm} can be applied. First, if $m$ contains some $(x_i, y_i)$ such that $x_i \neq y_i$, then the first statement can be applied. If $x_i$ and $y_i$ are equal for all $i$ and $r(m) > 1$ then either $\epsilon_{x_i} = 1$ for some $i$ such that $x_i \neq 1$ or not all $(1^1,1^1)$'s are adjacent. If $\epsilon_{x_i} = 1$ for some $i$ such that $x_i \neq 1$, apply the operation of Statement 3a or 3b. If not all $(1^1, 1^1)$'s are adjacent, the operation of Statement 2a or 2b can be applied.
$\blacksquare$
\end{proof}

The following lemma will be key to the proof of Theorem \ref{dualbasisthm}.
\begin{lemma}
\label{keylemma}
The equalities shown in Figure \ref{keyLemma2Fig} hold up to sign.
%\afterpage{
\begin{figure}[p!]
   \centering
   \includegraphics[width=6in]{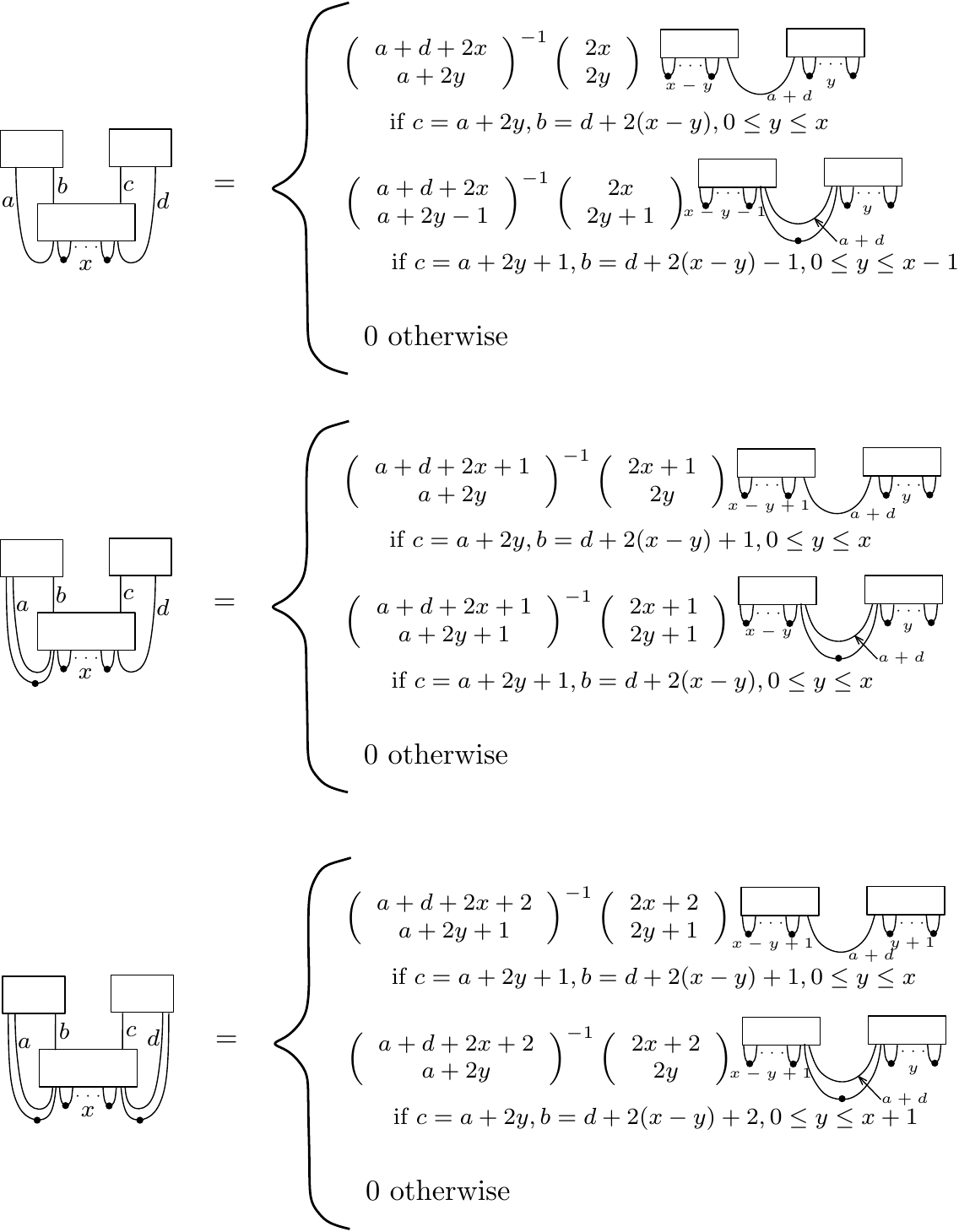}
   \caption{Lemma \ref{keylemma}}
      \label{keyLemma2Fig}
\end{figure}
%}
\end{lemma}
Note: the first equality in Figure \ref{keyLemma2Fig} with $x=0$ is Lemma 3.1 from \cite{khthesis}.

\begin{proof}
We prove the first case: the others are proven by an analogous argument. First observe the following manipulations of the left-hand side, obtained by applying the inductive relationship for Jones-Wenzl projectors.
\begin{figure}[H]
   \centering
   \includegraphics[scale=.8]{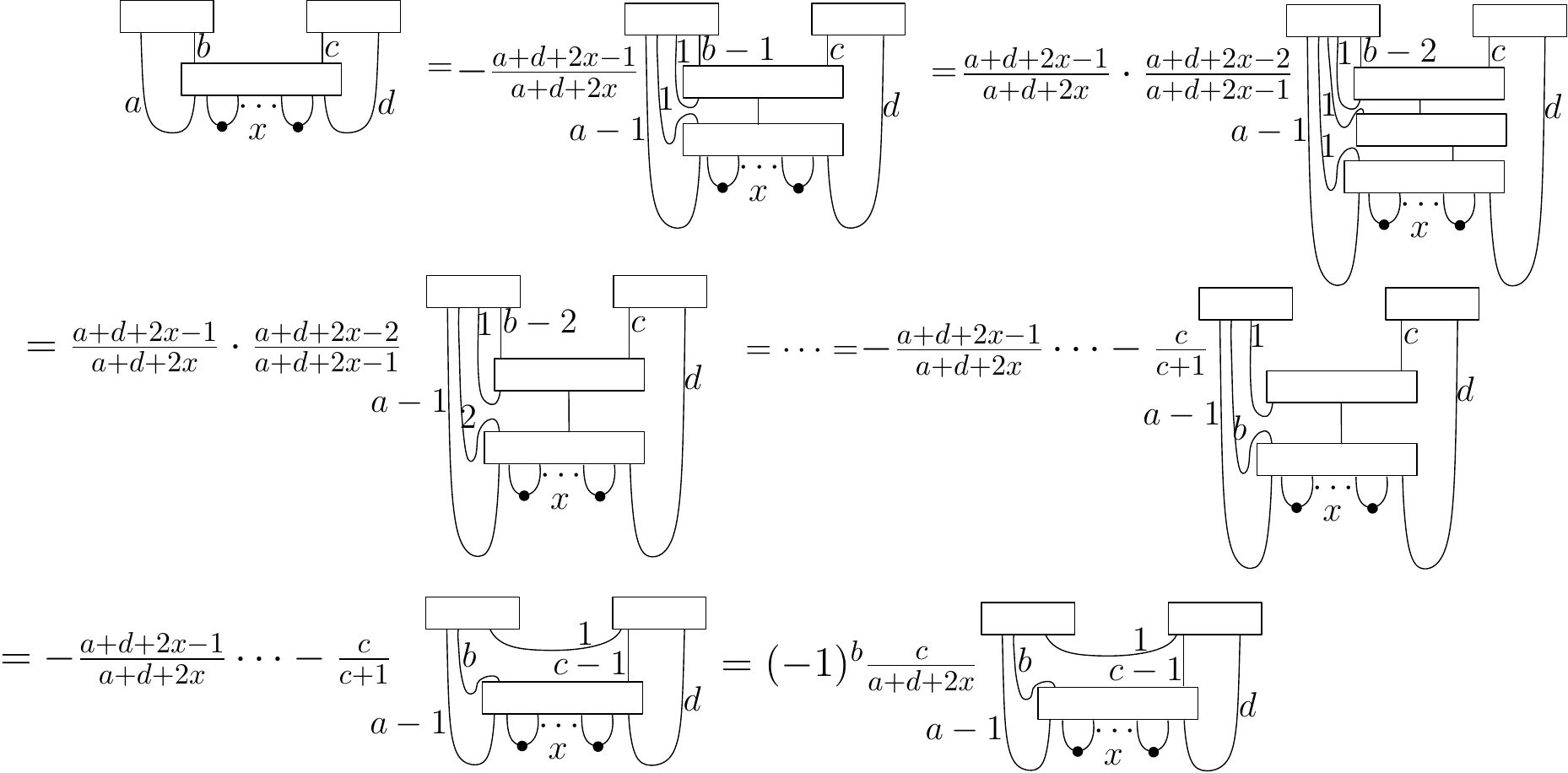}
   %\caption{Lemma \ref{keylemma}}
      %\label{keyLemma2Fig}
\end{figure}
Now if $a > c$, by induction we can simplify the diagram, up to scale, to
\begin{figure}[H]
   \centering
   \includegraphics[scale=1]{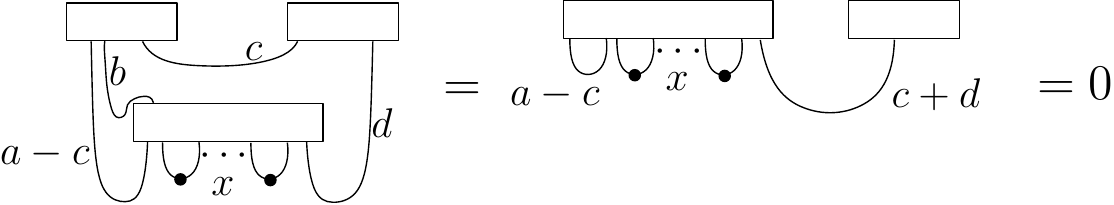}
   %\caption{Lemma \ref{keylemma}}
      %\label{keyLemma2Fig}
\end{figure}
Therefore for the left-hand side to be nonzero, we must have $a \geq c$ and $b \leq d+2x$, since $a+d +2x = b+c$. By a completely symmetric argument, we can conclude that for the left-hand side to be nonzero, we must have
\[
\begin{array}{ccccc}
 a &\leq& c &\leq& a+2x \\
d &\leq& b &\leq& d+2x.
\end{array} \]
Assuming these conditions, we now divide into two cases based on whether $c$ differs from $a$ by an even number or an odd number. First suppose that
\[c = a +2y, 0 \leq y \leq x\]
and therefore
\[b = d+2(x-y).\]
First, we specialize the result of the previous calculation to this situation:
\begin{figure}[H]
   \centering
   \includegraphics[scale=.6]{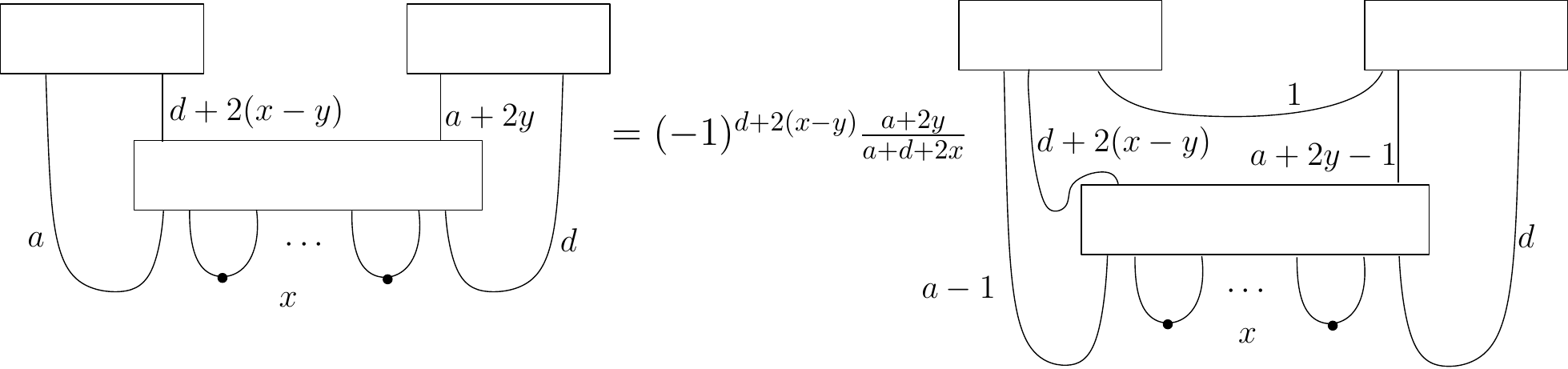}
   %\caption{Lemma \ref{keylemma}}
      %\label{keyLemma2Fig}
\end{figure}
By induction on $a$, we find that this is equal to
\begin{figure}[H]
   \centering
   \includegraphics[scale=.6]{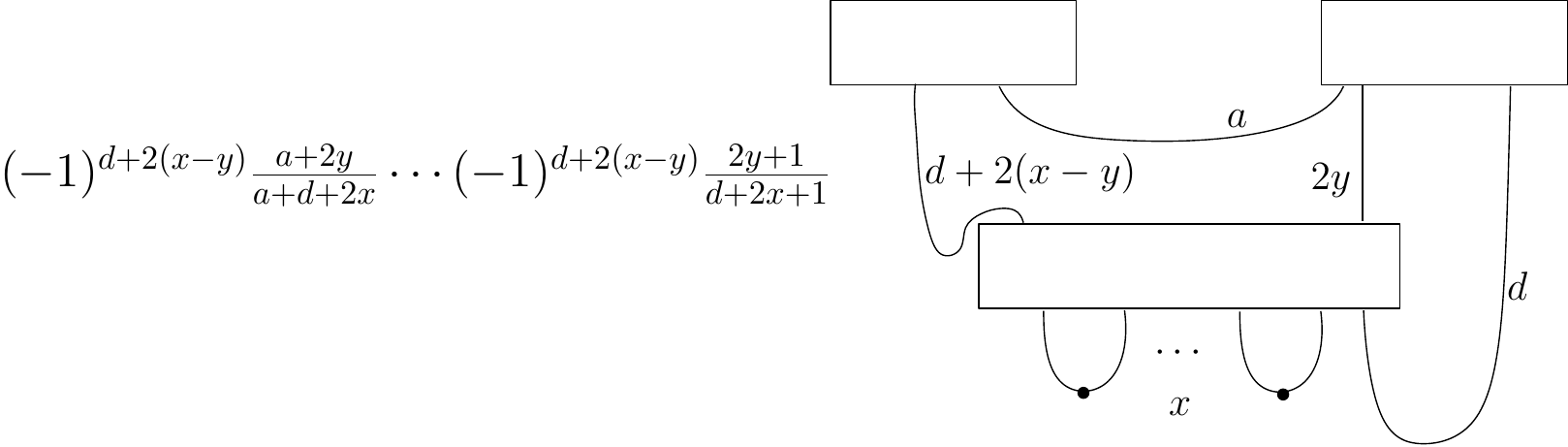}
   %\caption{Lemma \ref{keylemma}}
      %\label{keyLemma2Fig}
\end{figure}
Performing a symmetric argument, we may also inductively move the $d$ strands from the middle projector to between the two top projectors:
\begin{figure}[H]
   \centering
   \includegraphics[scale=.6]{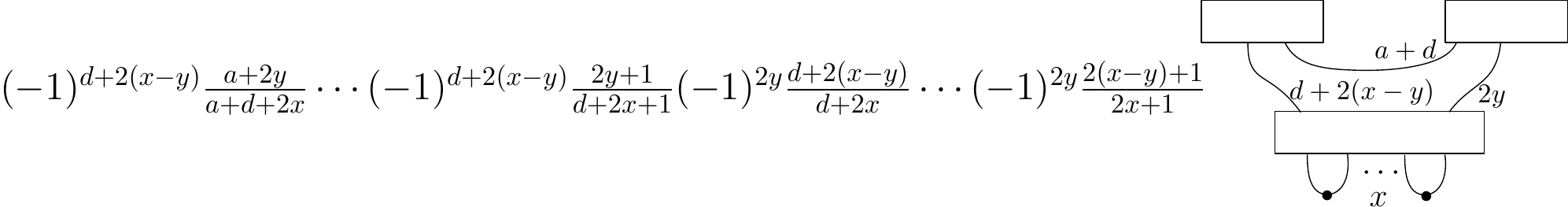}
   %\caption{Lemma \ref{keylemma}}
      %\label{keyLemma2Fig}
\end{figure}
Next observe that in general we have
\begin{figure}[H]
   \centering
   \includegraphics[scale=1]{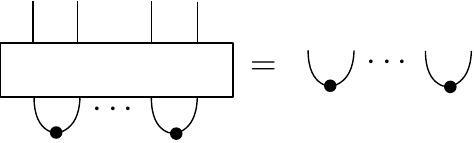}
   %\caption{Lemma \ref{keylemma}}
      %\label{keyLemma2Fig}
\end{figure}
\noindent because when the Jones--Wenzl projector is expanded into a linear combination of planar tangles, only the term containing the tangle of all vertical lines will be nonzero, since any others will contain two dots on a single strand.

Therefore, returning to our calculation and simplifying the coefficient, we obtain
\begin{figure}[H]
   \centering
   \includegraphics[scale=1]{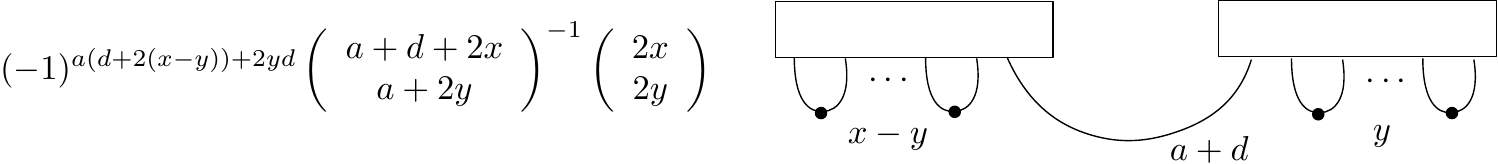}
   %\caption{Lemma \ref{keylemma}}
      %\label{keyLemma2Fig}
\end{figure}
\noindent as desired. The case where $a = 2y+1$ is completely analogous until the final step, where it is clear that the diagram on the right-hand side in Figure \ref{keylemma} is the one obtained.
$\blacksquare$
\end{proof}

\begin{comment}
\begin{lemma}
\label{keylemma}
The equalities shown in Figure \ref{keyLemma3}  hold up to sign.
\begin{figure}[ht]
   \centering
   \includegraphics[width=6in]{keyLemma3.pdf}
   \caption{Lemma \ref{keylemma}}
      \label{keyLemma3}
\end{figure}
\end{lemma}
\end{comment}

Simultaneously with Theorem \ref{dualbasisthm}, we will prove the following Lemma.
\begin{lemma}
\label{simlemma}

\begin{enumerate}
\item Let $m = (x_1^{\epsilon_{x_1}}, y_1^{\epsilon_{y_1}}; \ldots; x_k^{\epsilon_{x_k}}, y_k^{\epsilon_{y_k}})$. Then
\[ m^* = (p_{x_1} \otimes p_{y_1} \otimes \cdots \otimes p_{x_k} \otimes p_{y_k}) m^*.\]

\item Let $m = (x_1^{\epsilon_{x_1}}, y_1^{\epsilon_{y_1}}; \ldots; x_i^{\epsilon_{x_i}}, y_i^{\epsilon_{y_i}}; 1^1, 1^1; \ldots, 1^1, 1^1; x_{i+\alpha+1}^{\epsilon_{x_{i+\alpha+1}}}, y_{i+\alpha+1}^{\epsilon_{y_{i+\alpha+1}}}; \ldots; x_k^{\epsilon_{x_k}}, y_k^{\epsilon_{y_k}})$ where $x_i^{\epsilon_{x_i}} = y_i^{\epsilon_{y_i}}$ and $x_{i+\alpha+1}^{\epsilon_{x_{i+\alpha+1}}} = y_{i+\alpha+1}^{\epsilon_{y_{i+\alpha+1}}}$. Then
\begin{eqnarray*}
m^* &=& 1^{\otimes (x_1+y_1+ \cdots + x_i)} \otimes p_{y_i+2\beta} \otimes 1^{\otimes (2(\alpha-\beta) + x_{i+\alpha+1}+y_{i+\alpha+1} + \cdots+x_k+y_k )} m^* \\
&=& 1^{\otimes (x_1 +y_1 + \cdots + x_{i}+y_i + 2(\alpha-\beta))} \otimes p_{x_{i+\beta+1}+2\alpha} \otimes  1^{\otimes (y_{i+\alpha+1} +x_{i + \alpha +2} + y_{i+\alpha+2} + \cdots + x_k + y_k)} m^*.
\end{eqnarray*}
for any $\beta$ such that $ 0 \leq \beta \leq \alpha$.
\end{enumerate}
\end{lemma}

\begin{proof}
Theorem \ref{dualbasisthm} and Lemma \ref{simlemma} are proved simultaneously by induction on $r$. It is clear that the statement of Lemma \ref{simlemma} holds for the base case $r=1$ of Lemma \ref{basecase}. 

First we show Statement 1. Let $m = (x_1^{\epsilon_{x_1}}, y_1^{\epsilon_{y_1}}; \ldots; x_i^{\epsilon_{x_i}}, y_i^0; x_{i+1}^0, y_{i+1}^{\epsilon_{y_{i+1}}}; \ldots; x_k^{\epsilon_{x_k}}, y_k^{\epsilon_{y_k}})$, with $x_i \geq y_i$ and $x_{i+1} \leq y_{i+1}$, and let $m' = (x_1^{\epsilon_{x_1}}, y_1^{\epsilon_{y_1}}; \ldots; (x_i+x_{i+1})^{\epsilon_{x_i}}, (y_i+y_{i+1})^{\epsilon_{y_{i+1}}}; \ldots; x_k^{\epsilon_{x_k}}, y_k^{\epsilon_{y_k}})$. We want to show that $g := \left( \begin{array}{c}
y_i + x_{i+1} \\
y_i
\end{array}
\right)
1^{\otimes l} \otimes p_{y_i + x_{i+1}} \otimes 1^{\otimes j} (m')^{\ast}$ is dual to $m$. This amounts to showing that $g$ paired with $m$ is 1, and $g$ paired with any Russell basis element $b \neq m$ is zero. By the inductive hypothesis of Lemma \ref{simlemma} and using the fact that $x_i \geq y_i$ and $x_{i+1} \leq y_{i+1}$, we can pull projectors $A$ and $B$ off of $(m')^*$ as shown in Figure \ref{dualbasisproof6}, where $|A| = |B|= x_{i+1}+y_i$.

\begin{figure}[H]
   \centering
   \includegraphics[scale=1]{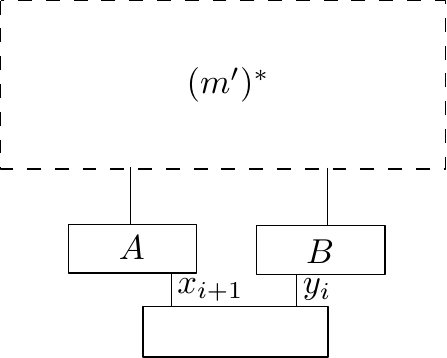}
   \caption{The element $g$ in the proof of Statement 1, where $|A| = |B|= x_{i+1}+y_i$.}
      \label{dualbasisproof6}
\end{figure}

Note that there are $x_{i+1} + y_i$ undotted arcs joining $A$ and $B$ in $m'$. For an arbitrary Russell basis element $b$, we again use the notation $\langle b, g \rangle$ to denote the evaluation of the diagram obtained by matching endpoints of $b$ and $g$ according to the rules of Section \ref{sec:bilin_form}. In order for $\langle b, g \rangle$ to be nonzero, by the inductive hypothesis there must be at least one term in the linear combination resulting from the expansion of the new projector that has $x_{i+1}+y_i$ arcs joining $A$ and $B$. This could only possibly happen if $\langle b, g \rangle$ is as shown in Figure \ref{dualbasisproof7}, where all arcs of $b$ that are not shown are identical to those of $m'$. Lemma \ref{keylemma} shows that $\langle b, g \rangle$ is indeed nonzero, with the coefficient appearing in the Lemma balancing that which appears in Theorem \ref{dualbasisthm}. 

\begin{figure}[H]
   \centering
   \includegraphics[scale=1]{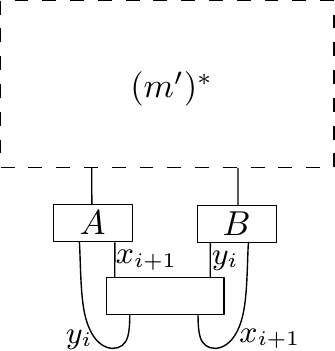}
   \caption{The only possible nonzero pairing $\langle b, g \rangle$.}
      \label{dualbasisproof7}
\end{figure}

Therefore $g = m^*$. It is clear from Figure \ref{dualbasisproof6} and the inductive hypothesis that $g$ satisfies the conclusion of Lemma \ref{simlemma}.

Next, we show Statement 2a of Theorem \ref{dualbasisthm}.

Suppose that $m = (x_1^{\epsilon_{x_1}}, y_1^{\epsilon_{y_1}}; \ldots; x_i^0, y_i^0; 1^1, 1^1; \ldots; 1^1, 1^1; x_{i+\alpha+1}^0, y_{i+\alpha+1}^0; \ldots; x_k^{\epsilon_{x_k}}, y_k^{\epsilon_{y_k}})$ is such that $x_i = y_i$, $x_{i+\alpha+1} = y_{i+\alpha+1}$, and $i + \alpha + 1 < k$. Let $m' = (x_1^{\epsilon_{x_1}},y_1^{\epsilon_{y_1}}; \ldots; x_{i-1}^{\epsilon_{i-1}}, y_{i-1}^{\epsilon_{y_{i-1}}}; (x_i + x_{i+\alpha+1})^0, (y_i+y_{i+\alpha+1})^0; 1^1, 1^1; \ldots, 1^1, 1^1; x_{i+\alpha+2}^{\epsilon_{x_{i+\alpha+2}}}, y_{i+\alpha+2}^{\epsilon_{y_{i+\alpha+2}}}; \ldots, x_k^{\epsilon_{x_k}}, y_k^{\epsilon_{y_k}})$. We want to show that $g := \left(
\begin{array}{c}
x_i + 2\alpha + y_{i+\alpha+1} \\
x_i + 2\alpha
\end{array}
\right)
1^{\otimes l} \otimes p_{y_i +2\alpha+x_{i+a+1}} \otimes 1^{\otimes j} (m')^*$ is dual to $m$. Again, this amounts to showing that $g$ paired with $m$ is 1, and $g$ paired with any Russell basis element $b \neq m$ is zero.

By the induction hypothesis of Lemma \ref{simlemma}, $(m')^* = 1^{(x_1+ y_1 + \cdots+x_{i-1}+y_{i-1})} \otimes p_{x_i+x_{i+\alpha+1}} \otimes p_{x_i+x_{i+\alpha+1} + 2\alpha} \otimes 1^{(x_{i+\alpha+2} + y_{i+\alpha+2} + \cdots + x_k + y_k)}.$ Label the first of those two projectors $A$ and the second $B$. Therefore $g$ is as shown in Figure \ref{dualbasisproof1} (with the scalar omitted).

\begin{figure}[H]
   \centering
   \includegraphics[scale=1]{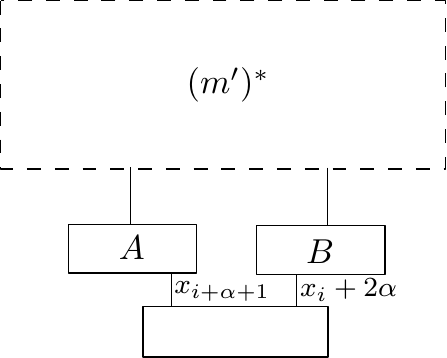}
   \caption{The element $g$ in the proof of Statement 2a, where $|A| = x_i+x_{i+\alpha+1}$ and $|B| = x_i + x_{i+\alpha+1} + 2\alpha$.}
      \label{dualbasisproof1}
\end{figure}

Consider the evaluation of the pairing of $g$ with an arbitrary Russell basis element $b$. Note that there are $x_i + x_{i+\alpha+1}$ arcs connecting $m'[l-x_i+1, l+x_{i+\alpha+1}]$ with $m'[l+x_{i+\alpha+1}+1, l+x_i + 2x_{i+\alpha+1}]$, where in general we use $m[a,b]$ to denote the subset of endpoints numbered $i$ in the diagram $m$ with $a \leq i \leq b$, and $\alpha$ dotted cups in $m'[l+x_i+2x_{i+\alpha+1}+1, l+x_i+2x_{i+\alpha+1}+2\alpha]$. Consider the pairing of $g$ with the arbitrary Russell basis element $b$. First suppose that $b$ has some nonzero number $p$ of cups in $b[l+x_i+2x_{i+\alpha+1}+1, l+x_i+2x_{i+\alpha+1}+2\alpha]$. Note that we must have $1 \leq p \leq \frac{x_{i+\alpha+1}}{2}$. Then $\langle b, g \rangle$ is as shown in one of the pictures of Figure \ref{dualbasisproof2}. In this picture, $a$ and $d$ represent the number of parallel strands in the indicated arc, none of which carry a dot.

\begin{figure}[H]
   \centering
   \includegraphics[scale=.85]{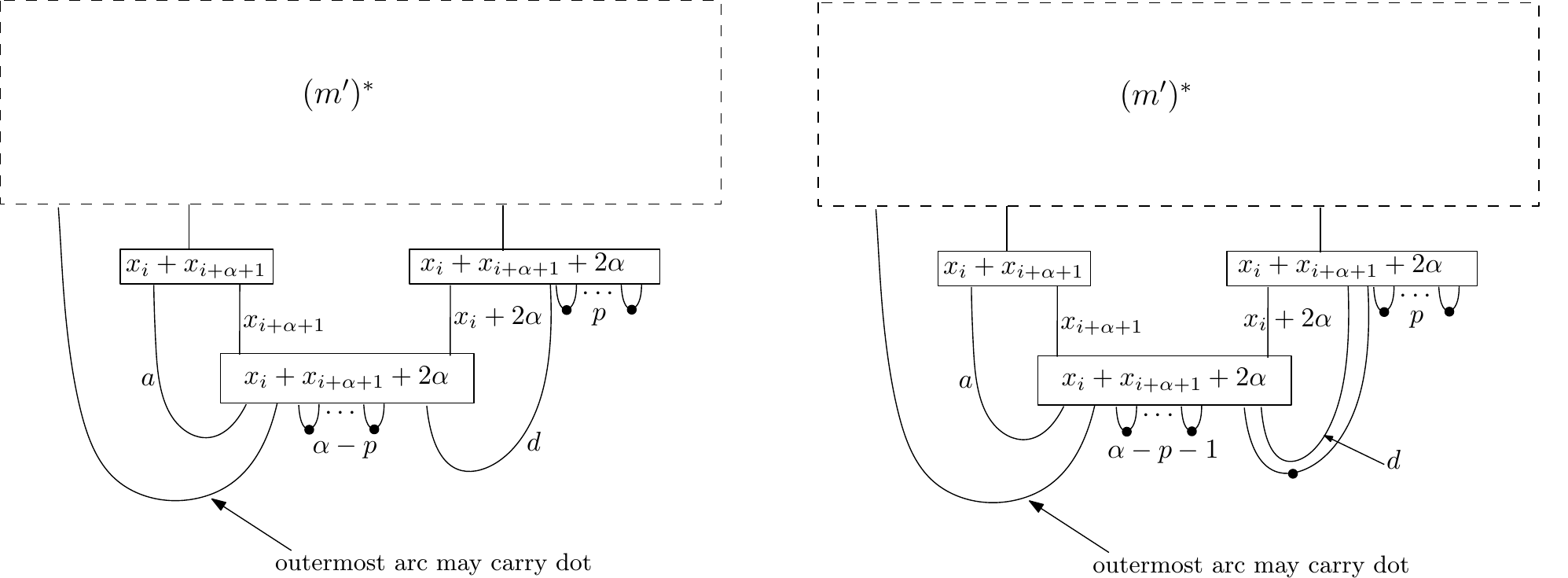}
   \caption{Pairing of $g$ with a Russell basis element with $p > 0$}
      \label{dualbasisproof2}
\end{figure}

Observe that $d$ must be exactly $x_{i+\alpha+1} - 2p$ in the picture on the left or $x_{i+\alpha+1}-2p-1$ in the picture on the right, and $a = x_i$. In either case, $a + d < x_i+x_{i+\alpha+1}$. Therefore in any expansion of the newly added projector, there will be fewer than $x_i + x_{i+\alpha}+1$ arcs connecting $[l-x_i+1, l+x_{i+\alpha+1}]$ with $[l+x_{i+\alpha+1}+1, l+x_i + 2x_{i+\alpha+1}]$, so $\langle b, g \rangle = 0$. We have thus established that if $\langle b, g \rangle \neq 0$, then we must have $p = 0$.

If $p=0$, then $\langle b, g \rangle$ is as shown in one of the pictures of Figure \ref{dualbasisproof3}. In the picture on the right, $\langle b, g \rangle = 0$ for any possible values of $a$ and $d$ since $a = x_i$ and $d = x_{i+\alpha+1}-1$, so there cannot be $x_i + x_{i+\alpha}+1$ arcs connecting $m'[l-x_i+1, l+x_{i+\alpha+1}]$ with $m'[l+x_{i+\alpha+1}+1, l+x_i + 2x_{i+\alpha+1}]$ in any expansion of the new projector. In the picture on the left we must have $a = x_i$ and $d = x_{i+\alpha+1}$. Here, $a+d = x_i + x_{i+\alpha+1}$,  so the diagram could possibly be nonzero. Therefore $\langle b, g \rangle=0$ except possibly when $a = x_i$ and $d=x_{i+\alpha+1}$.

\begin{figure}[H]
   \centering
   \includegraphics[scale=.8]{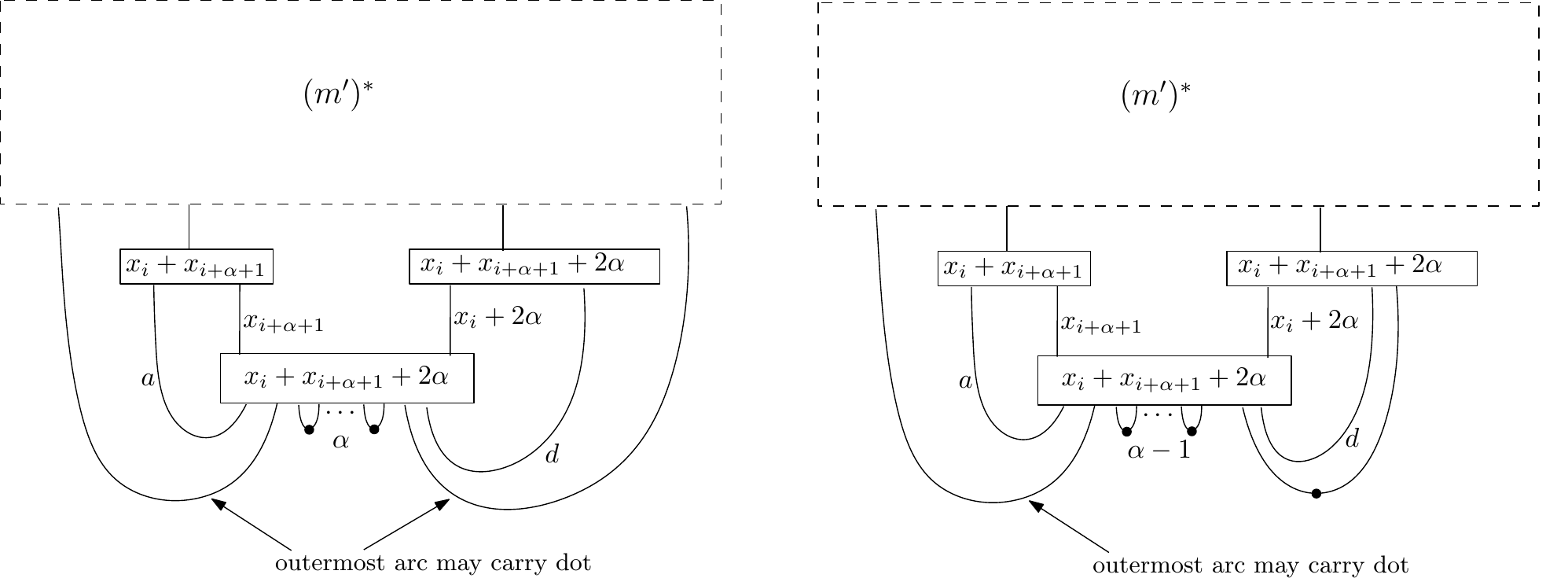}
   \caption{Pairing of $g$ with a Russell basis element with $p = 0$}
      \label{dualbasisproof3}
\end{figure}

In this case, $\langle b, g \rangle$ is shown in Figure \ref{dualbasisproof4}. Lemma \ref{keylemma} shows that $\langle b, g \rangle$ is indeed nonzero, with the coefficient appearing in the Lemma balancing that which appears in Theorem \ref{dualbasisthm}.

\begin{figure}[H]
   \centering
   \includegraphics[scale=.85]{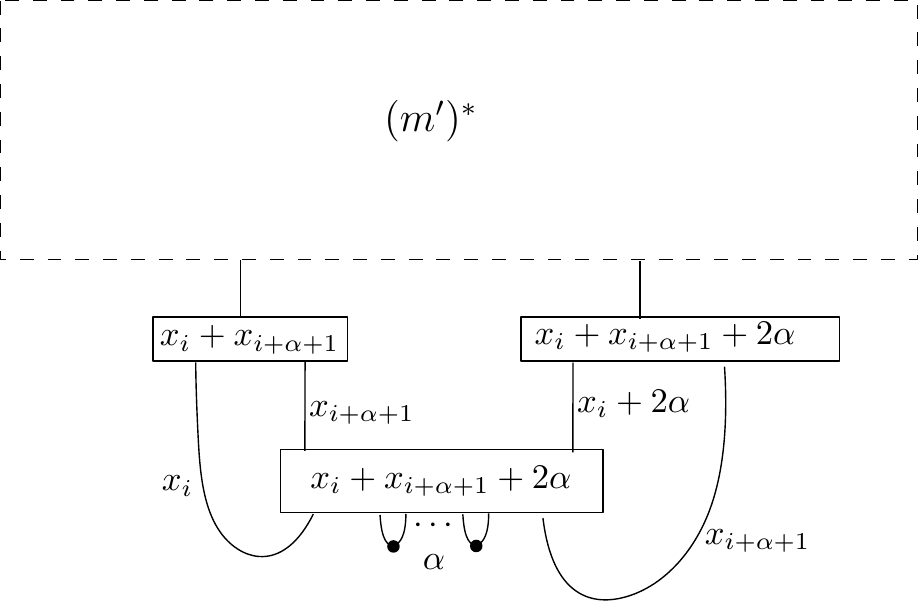}
   \caption{Only possible nonzero pairing $\langle b, g \rangle$.}
      \label{dualbasisproof4}
\end{figure}

Therefore $g = m^*$. It is clear from Figure \ref{dualbasisproof4} and the inductive hypothesis that $g$ satisfies the conclusion of Lemma \ref{simlemma}.

The proof of Statement 2b follows immediately from this one by left-right symmetry. The proof of Statement 3a follows by a completely analogous argument, where $g$ becomes as in Figure \ref{dualbasisproof5}, with $|A| = x_i + x_{i+\alpha+1}$ and $|B| = x_i + x_{i+\alpha+1}+2\alpha+1$, instead of as in Figure \ref{dualbasisproof1}. Again, the proof of Statement 3b follows from that of Statement 3a by left-right symmetry.

\begin{figure}[H]
   \centering
   \includegraphics[scale=.85]{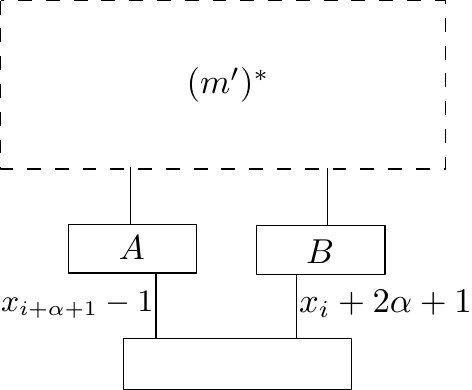}
   \caption{The element $g$ appearing in the right hand side of Statement 3a.}
      \label{dualbasisproof5}
\end{figure}

$\blacksquare$
\end{proof}

\begin{corollary}
$r(m)$ is the number of projectors in $m^*$.
\end{corollary}
\begin{proof}
In the base case $r = 1$, Lemma \ref{basecase} shows that one projector is used. It is clear that every statement in Theorem \ref{dualbasisthm} adds a single projector and reduces $r$ by one.
\end{proof}

Figure \ref{dualbasisEx} shows an example of how to build a dual basis element by applying Theorem \ref{dualbasisthm} twice and the base case of Lemma \ref{basecase}. The coefficients have been omitted.

\begin{figure}[H]
   \centering
   \includegraphics[width=6.5in]{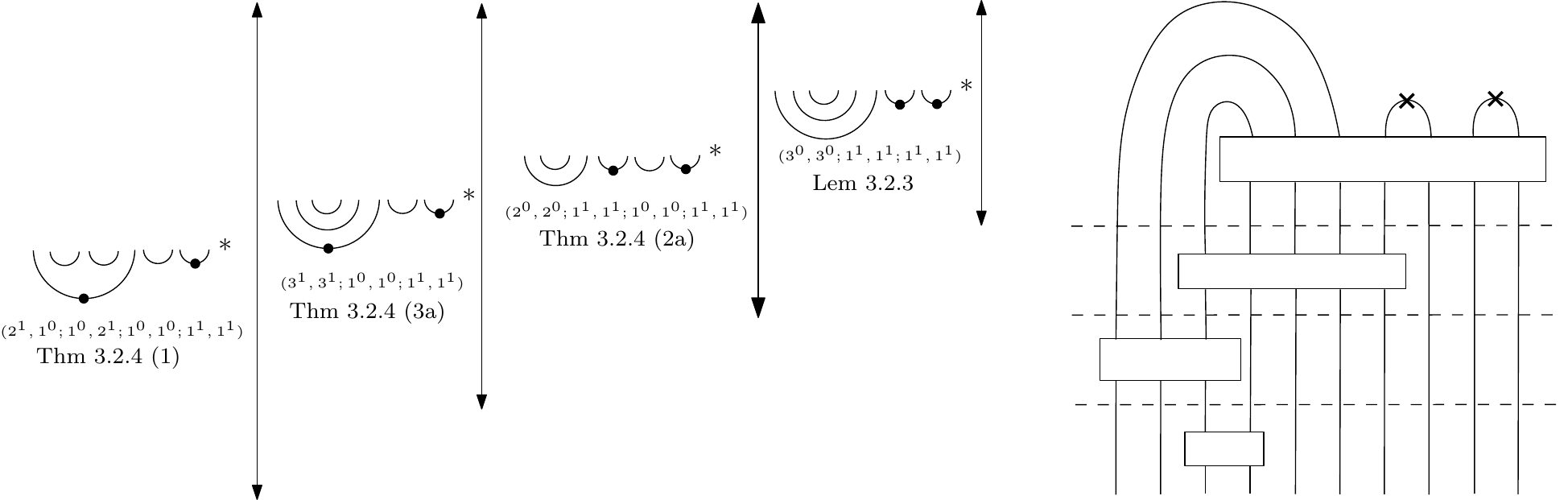}
   \caption{Sample construction of a dual basis element using Theorem \ref{dualbasisthm}. }
      \label{dualbasisEx}
\end{figure}

See Figure \ref{n=3Duals} for all graphical dual basis elements in the case $n=3$.

\begin{figure}[H]
   \centering
   \includegraphics[width=6in]{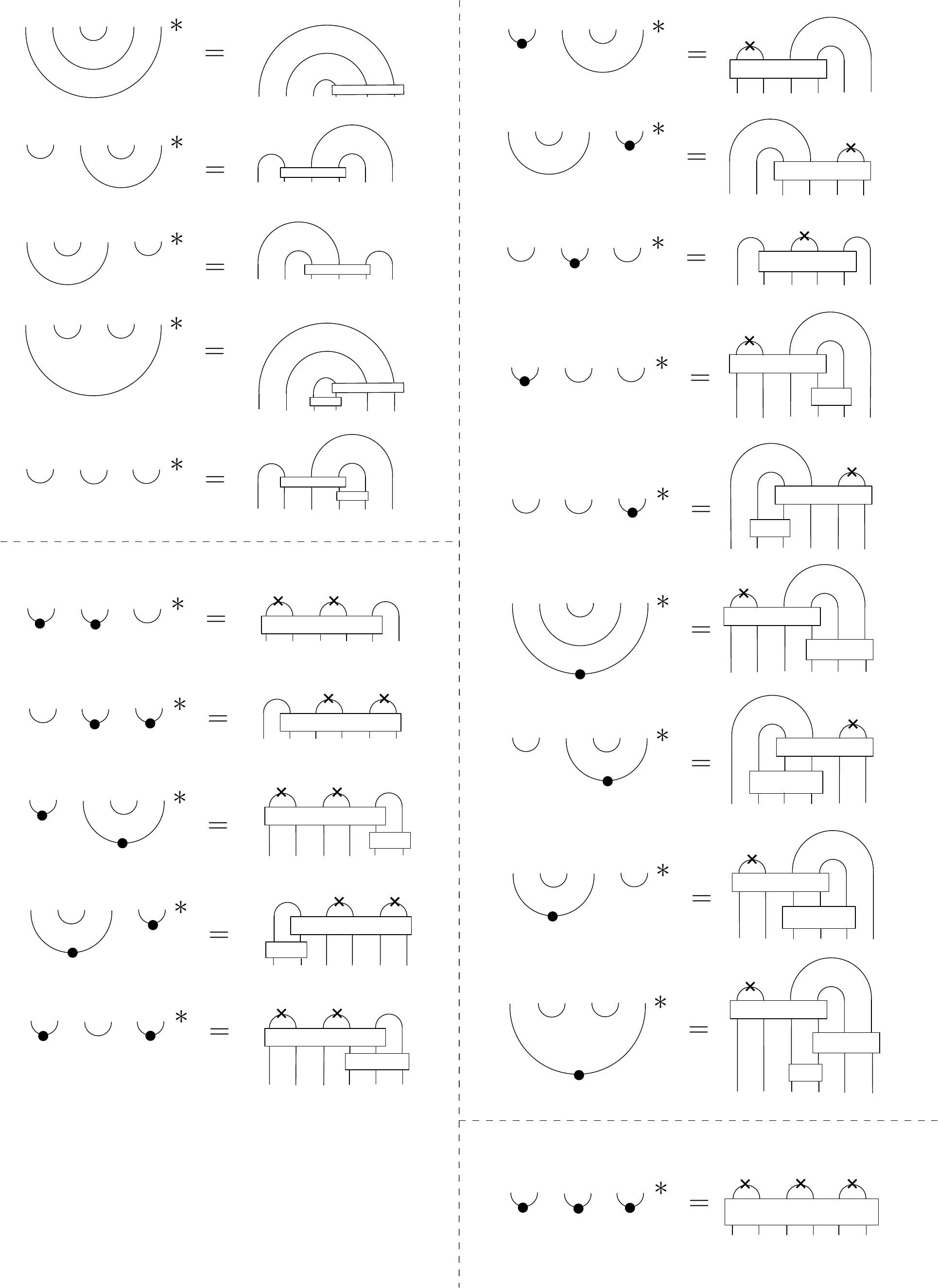}
   \caption{Graphical dual basis elements (up to scale) for $n=3$.}
      \label{n=3Duals}
\end{figure}

\chapter{A quantization of the Russell skein module}
\label{ch:Quantum-Russell}

\section{The spaces $\widetilde{R}^q_{n,k}$ and $R^q_{n,k}$}
\label{sec:qRussell-def}

%Moved traditional Russell space to introduction

\subsection {The quantum Russell space}
\label{sec:qRussell}
In this section we define a new quantum deformation of the Russell space.

\begin{definition}
Let $\widetilde{R}^q_{n,k} = \widetilde{R}_{n,k} \otimes_{\mathbb{Z}} \mathbb{Z}[q,q^{-1}]$. In other words, $\widetilde{R}^q_{n,k}$ is the space of linear combinations of the same diagrams as those that give a basis of $\widetilde{R}_{n,k}$ but now with coefficients in $\mathbb{Z}[q, q^{-1}]$ instead of $\mathbb{Z}$.
\end{definition}

When we talk about specializing $q$ to a particular integer value $a$, we first regard $\mathbb{Z}$ as a $\mathbb{Z}[q, q^{-1}]$-module $\mathbb{Z}_a$ whereby $q$ acts as $a$. Then we define
\[ \widetilde{R}_{n,k}^{q=a} = \widetilde{R}_{n,k}^q \otimes_{\mathbb{Z}[q,q^{-1}]} \mathbb{Z}_a.\]

More generally, if $a$ in a commutative ring $\mathbf{k}$, we first regard $\mathbf{k}$ as a $\mathbb{Z}[q, q^{-1}]$-module $\mathbf{k}_a$ where $q$ acts by $a$. Then we define
\[ \widetilde{R}_{n,k}^{q=a} = \widetilde{R}_{n,k}^q \otimes_{\mathbb{Z}[q,q^{-1}]}  \mathbf{k}_a.\] 

\begin{definition}
We introduce the \emph{quantum Russell relations} as shown in Figure \ref{quantum_Russell}. As before, the strand labels $a,b,c,d$ are such that $a < b < c < d$. The arcs labeled $x_1, \ldots, x_{\alpha}$ are the complete set of undotted arcs with endpoints $(e_1, e_2)$ such that $e_1 < a$ and $e_2  > d$. For each $i$, $1 \leq i \leq \alpha$, if the arc labeled $x_i$ has endpoints $(x_{i_1}, x_{i_2})$, $n_i$ is defined to be the number of dotted arcs with endpoints $(y_1, y_2)$ such that $x_{i_1} < y_1 < a$ and $ d < y_2 < x_{i_2}$. The relations are local in the sense that each of the diagrams appearing in each type of relation are identical apart from the arcs shown.
\end{definition}

\begin{figure}[H]
\centering
\includegraphics[scale=.8]{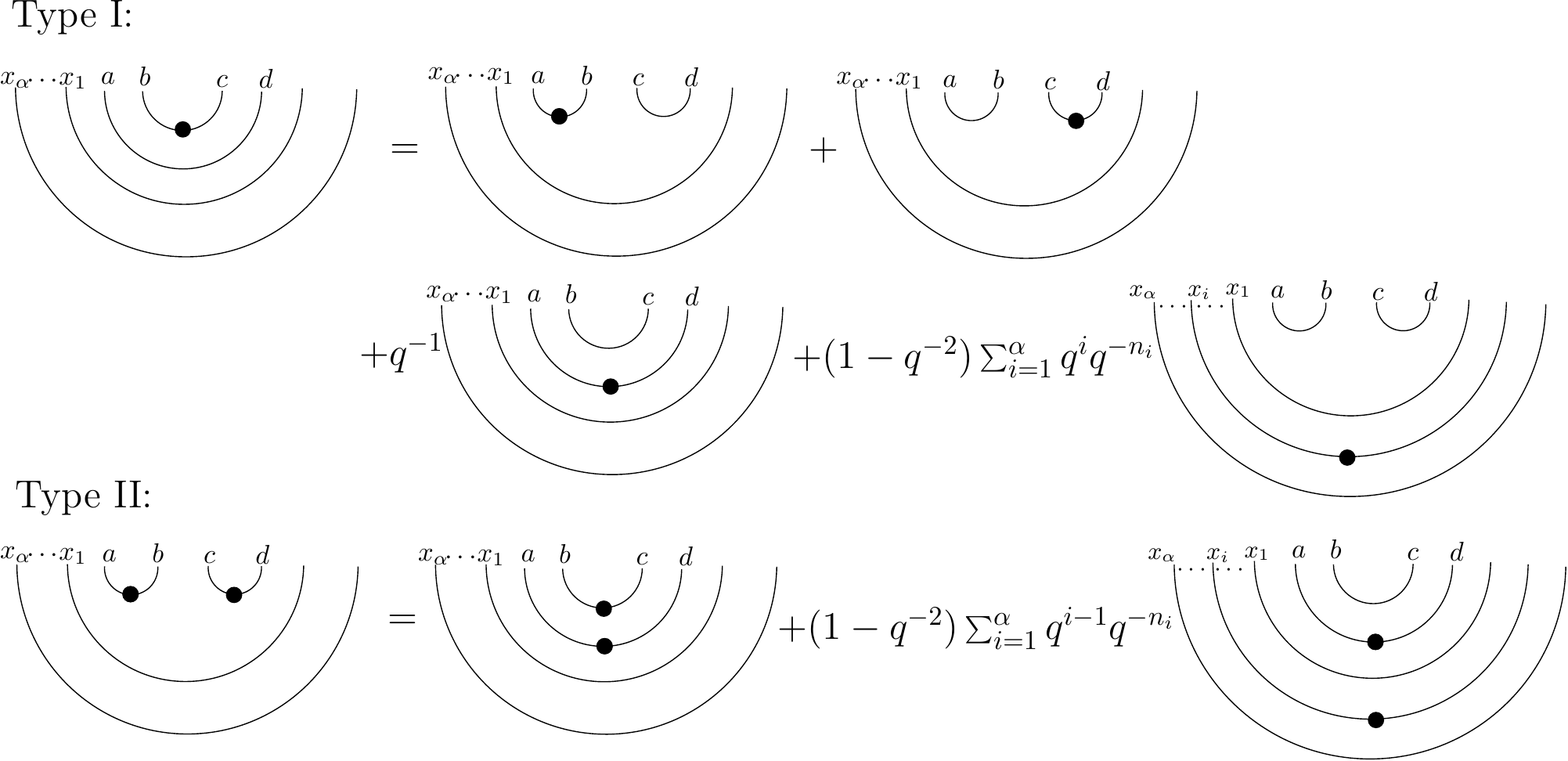}
\caption{The Type I and Type II quantum Russell relations}
\label{quantum_Russell}
\end{figure}

\begin{example}
Figure \ref{qRussell_ex} shows an application of a quantum Type I relation to the two bolded arcs in an element of $\widetilde{R}^q_{5,2}$.
\begin{figure}[H]
\centering
\includegraphics[scale=.8]{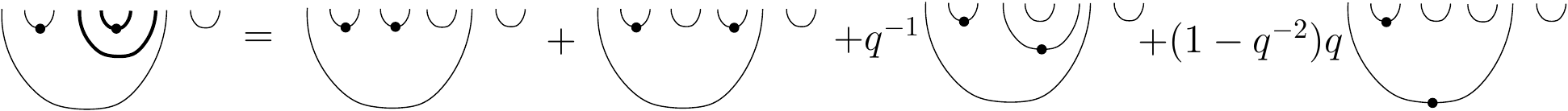}
\caption{An example application of a quantum Type I relation.}
\label{qRussell_ex}
\end{figure}
\end{example}

\begin{definition}
Define the \emph{quantum Russell space} $R^q_{n,k}$ to be the quotient of $\widetilde{R}^q_{n,k}$ by the Type I and Type II quantum Russell relations.
\end{definition}

\begin{comment}
\begin{example}
In $R^q_{n,k}$, the following Type I relation holds.
\end{example}
\end{comment}

\begin{remark}
When $q = -1$, the quantum Russell relations are the same as the original Russell relations. From this point forward we use the notation $R_{n,k}^{q=-1}$ and $R_{n,k}$ interchangeably.
\end{remark}

\begin{remark}
When $q = 1$, the space $R^{q=1}_{n,k}$ is isomorphic to the original Russell space $R_{n,k}$. The only difference in the relations is that the $q=1$ Type I relation has an extra minus sign, as shown in Figure \ref{modifiedRussell}. 

\begin{figure}[H]
\centering
\includegraphics[scale=1]{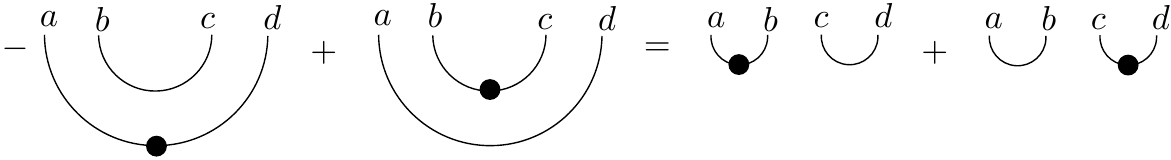}
\caption{The Type I quantum Russell relation with $q=1$ differs from the original Type I Russell relation by only a minus sign.}
\label{modifiedRussell}
\end{figure}

We note that there are two possible isomorphisms from $R^{q=1}_{n,k}$ to $R_{n,k}$ given by the rescaling of diagrams. One rescaling is given by multiplying a diagram by a factor of $-1$ for each pair of nested arcs for which the outer arc is dotted and the inner arc is undotted. Another rescaling is to multiply by $-1$ for each dotted arc nested in another (dotted or undotted) arc as well as an additional factor of $-1$ if the diagram is ``odd''. An diagram is defined to be odd if its underlying undotted arc structure can be transformed to n adjacent unnested arcs in an odd number of moves, where a move involves transforming a pair of arcs $(a, b), (c, d)$ into $(a, d), (b, c)$ or vice versa, with $a < b < c < d$. Note that under either rescaling, the new Type I relation becomes the original and the Type II relation is unchanged.
\end{remark}

\begin{comment}
More generally, the rescalings just described give the following proposition:
\begin{proposition}
The spaces $R_{n,k}^{q=a}$ and $R_{n,k}^{q=-a}$ are isomorphic as $\mathbb{Z}[q, q^{-1}]$-modules for any $a$ in a commutative ring $\mathbf{k}$.
\end{proposition}
\end{comment}

\subsection{A note on locality}
\label{sec:no-local}
Note that the quantum Type I and Type II Russell relations are not as local as the original ones, in the sense that the original relations only involved two arcs and the endpoints $a,b,c,d$, whereas the quantum relations involve each undotted arc with one endpoint to the left of $a$ and the other to the right of $d$. Only when $q= \pm 1$ do the terms involving such arcs disappear. However, our relations are still semi-local in the sense that arcs completely outside of the arc labeled $x_{\alpha}$ as well as those in between the labeled undotted arcs are unaffected by either Type I or Type II relations.

It turns out that this semi-locality is the best that we could hope for in a quantum version of the Russell relations.
\begin{proposition}
There are no fully local Type I and Type II relations depending on a parameter $q$ producing a space of the same dimension as the original Russell space $R_{n,k}$.
\end{proposition}

\begin{proof}
In order to have such local relations, they must be consistent, in the sense that for a given diagram in which two or more different relations may be performed, all choices of which to perform first must yield equivalent linear combinations of diagrams. Consider the most generic local Type I and Type II relations, where $a(q), b(q), c(q), d(q)$ are any functions of $q$:
\begin{figure}[H]
\centering
\includegraphics[scale=1]{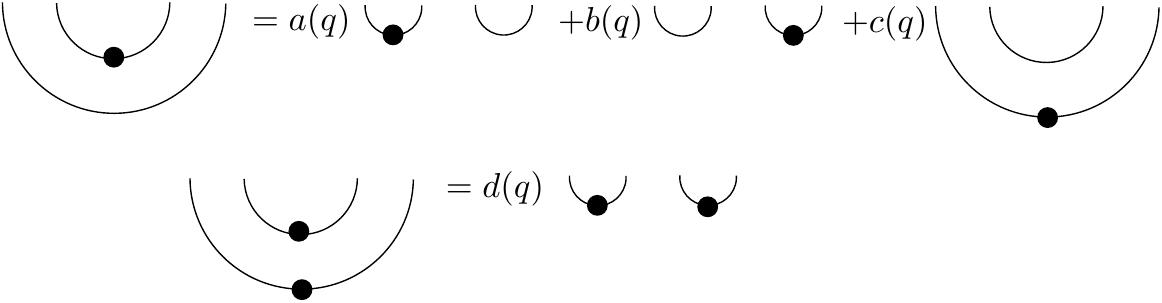}
\end{figure}
We will investigate what conditions are necessary on $a(q), b(q), c(q), d(q)$ for a deformation of the Russell space with these Type I and Type II relations to have the same dimension as the original. First consider the following equality, obtained by first applying a Type I relation to the outer undotted arc and the left of the two inner dotted arcs and then by applying the only possible Type I or Type II relation to the resulting terms.
\begin{figure}[H]
\centering
\includegraphics[scale=1]{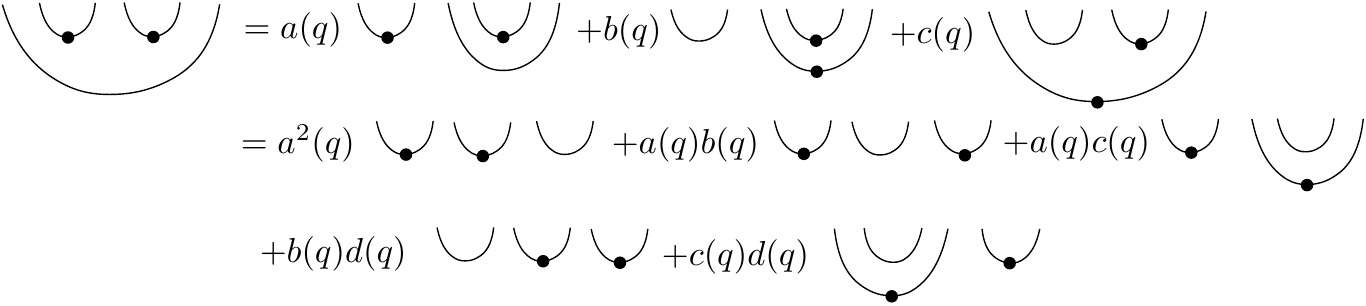}
\end{figure}
Next consider an expansion of the same element by first applying the other of the two possible Type I relations:
\begin{figure}[H]
\centering
\includegraphics[scale=1]{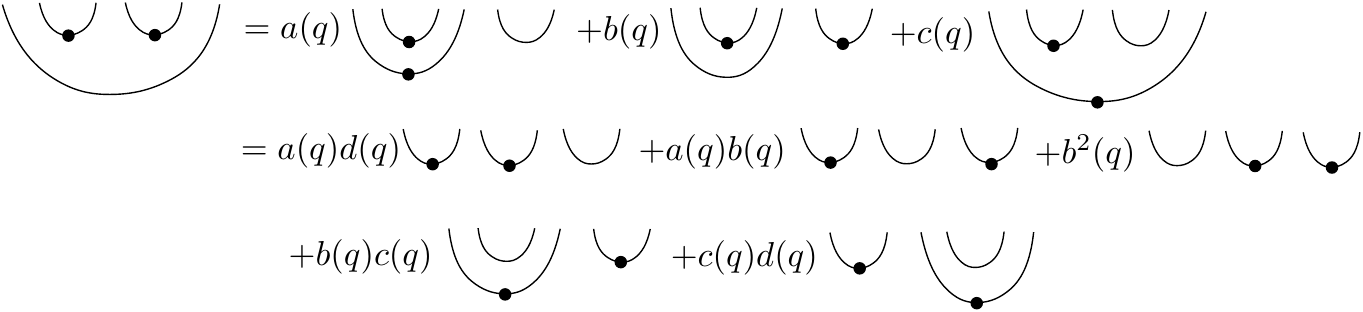}
\end{figure}
From these two computations, we see that in order to have consistency we must have
\[a(q) = b(q) = d(q).\]
Given these constraints, next consider the following computation, given by first applying the only possible Type II relation and then using the results of the previous computation.
\begin{figure}[H]
\centering
\includegraphics[scale=1]{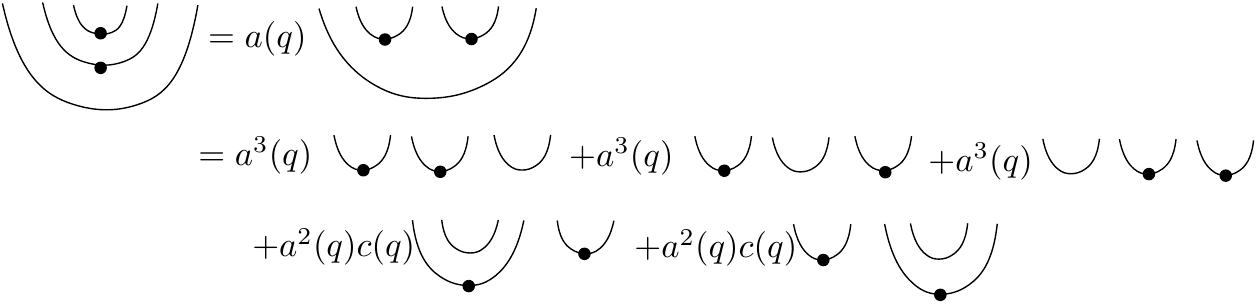}
\end{figure}
Alternatively, by first applying the only possible Type I relation, we see that
\begin{figure}[H]
\centering
\includegraphics[scale=1]{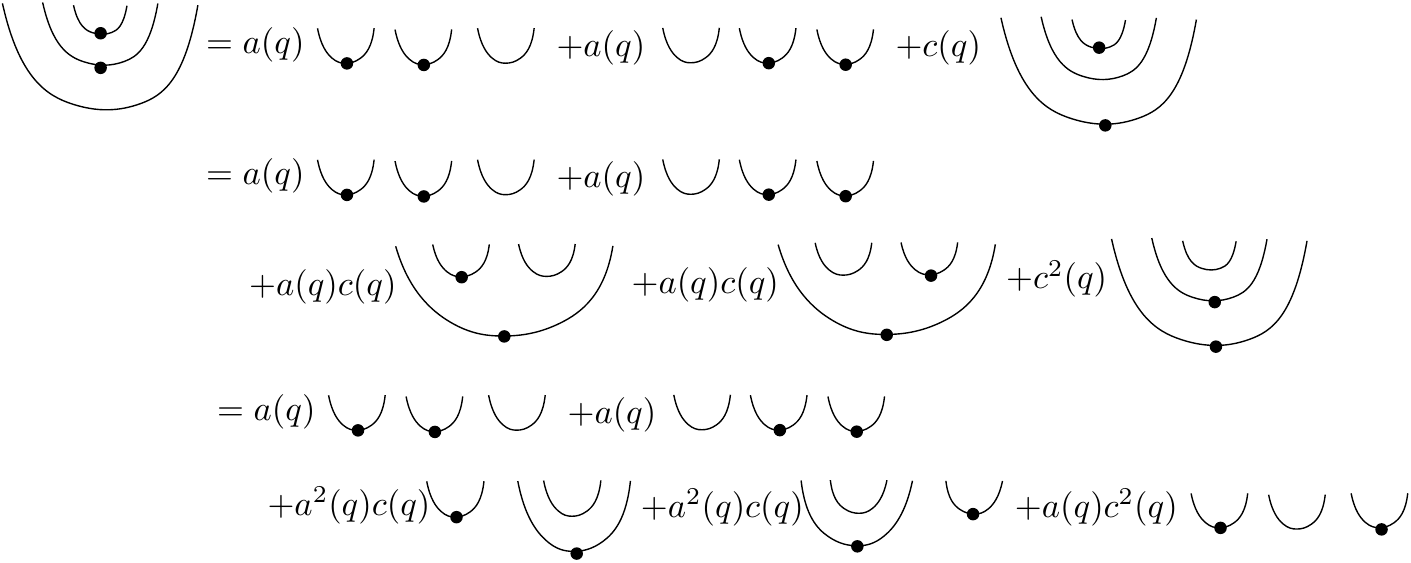}
\end{figure}
To have consistency, we conclude that
\[ a(q) = \pm 1\]
and
\[a(q) = \pm c(q), \]
so that $a, b, c, $ and $d$ may not depend on $q$. $\blacksquare$
\end{proof}

\begin{remark}
The system where $a=1, c=-1$ is the original Russell space, and that where $a=1, c=1$ is our $q=1$ quantum Russell space.
\end{remark}

\section{The spaces $\widetilde{S}^q_{n,k}$ and $S^q_{n,k}$}
\label{sec:KL-def}

First we define spaces $\widetilde{S}^q_{n,k}$.

\begin{definition}
Fix $2(n+k)$ points on a horizontal line and place a box or ``projector'' around each of the leftmost and rightmost $k$ points. A diagram in $\widetilde{S}^q_{n,k}$ is defined to be an embedding of $n+k$ arcs and a finite number of circles such that the boundary points of the arcs bijectively correspond to the $2(n+k)$ fixed points. Additionally, we require that no arc has both of its endpoints inside a single projector (such diagrams are defined to be 0). We define $\widetilde{S}^q_{n,k}$ to be the space of formal linear combinations of such diagrams with coefficients in $\mathbb{Z}[q^{1/2}, q^{-1/2}]$ subject to local relations for the expansion of a crossing and the evaluation of a circle given by traditional Kauffman--Lins diagrammatics and shown in Figure \ref{KL}.
\end{definition}

\begin{figure}[H]
\centering
\includegraphics[scale=1]{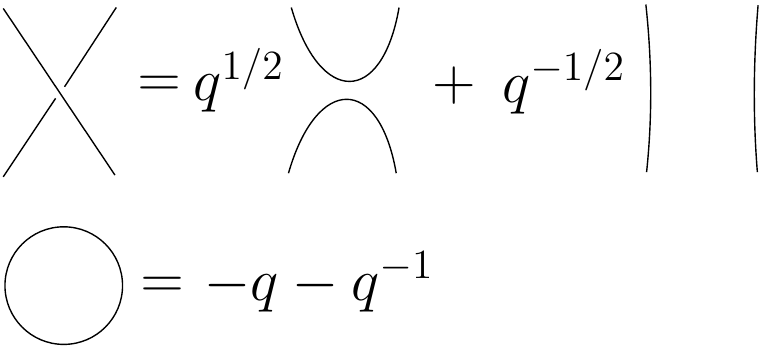}
\caption{The traditional Kauffman--Lins diagrammatics. The other type of crossing is given by the 90 degree rotation of the one pictured.}
\label{KL}
\end{figure}

\begin{remark}
For notational convenience we often suppress the dependence on $q$ and simply write $\widetilde{S}_{n,k}$.
\end{remark}

\begin{remark}
The notational use of `$\sim$' in $\widetilde{R}_{n,k}^q$ and $\widetilde{S}_{n,k}$ is unrelated.
\end{remark} 

\begin{proposition}
The space $\widetilde{S}_{n,k}$ has a basis consisting of those diagrams that do not have any circles or crossings between arcs.
\end{proposition}

\begin{proof}
It is clear from the Kauffman--Lins relations that such diagrams are linearly independent, since there are no crossings to expand or circles to remove. To see that they form a spanning set, note that any diagram in $\widetilde{S}_{n,k}$ that does have crossings between arcs may have those crossings resolved one at a time by applying a Kauffman--Lins relation. Similarly, any diagram containing circles may have them replaced by a coefficient of $-q - q^{-1}$. Therefore any diagram in $\widetilde{S}_{n,k}$ may be expressed as a $\mathbb{Z}[q^{1/2}, q^{-1/2}]$-linear combination of diagrams without crossings or circles.
\end{proof}

See Figure \ref{S_tilde-ex} for two equivalent sample elements of $\widetilde{S}_{3,2}$.

\begin{figure}[H]
\centering
\includegraphics[scale=.8]{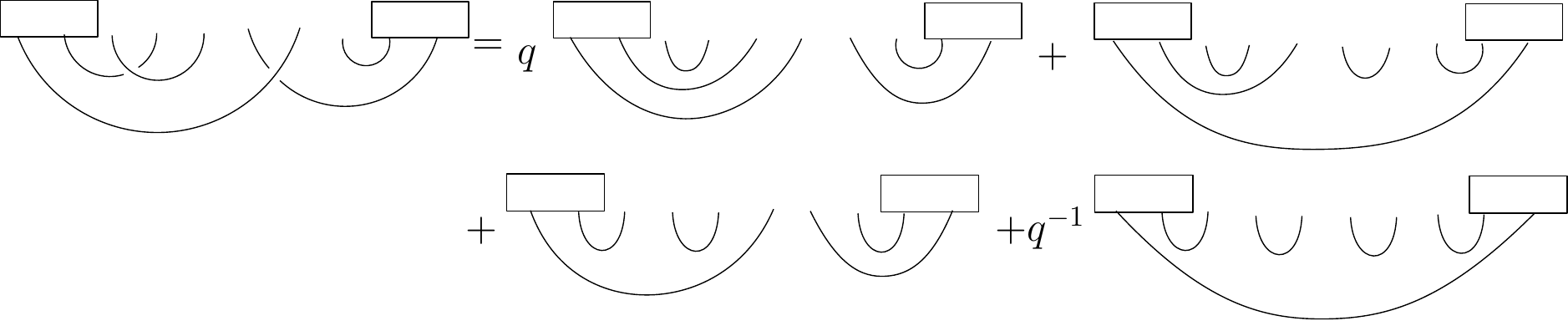}
\caption{A sample element of $\widetilde{S}_{3,2}$.}
\label{S_tilde-ex}
\end{figure}

We define an embedding $\varphi_{n,k} : \widetilde{S}_{n,k} \hookrightarrow \widetilde{S}_{n,k+1}$ as follows. Given a diagram $x$ in $\widetilde{S}_{n,k}$, add an additional fixed point on both the left and right sides of the existing $2(n+k)$ points and expand each projector to include a new fixed point. Add a new arc connecting the two new points, leaving the original arcs of $x$ in place. We define $\varphi_{n,k}(x)$ to be this resulting element of $\widetilde{S}_{n,k+1}$ and extend $\phi_{n,k}$ by linearity on an arbitrary element of $\widetilde{S}_{n,k}$. See Figure \ref{S_tilde_emb} for an example.

\begin{figure}[H]
\centering
\includegraphics[scale=1]{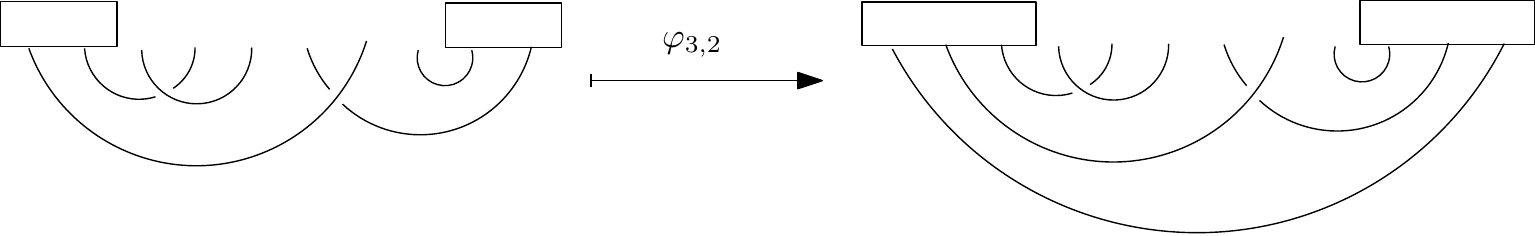}
\caption{The embedding of a sample element of $\widetilde{S}_{3,2}$ into $\widetilde{S}_{3,3}$.}
\label{S_tilde_emb}
\end{figure}

\begin{definition}
For $k \geq 1$,  define the space $S^q_{n,k}$ to be $\mbox{coker}(\varphi_{n,k-1}) = \widetilde{S}^q_{n,k}/\widetilde{S}^q_{n, k-1}$. Again, we frequently suppress $q$ and just write $S_{n,k}$.
\end{definition}

\begin{proposition}
\label{S-basis}
$S_{n,k}$ has a basis of diagrams without circles, crossings between arcs, or arcs connecting the two projectors.
\end{proposition}

\begin{proof}
We have already seen that $\widetilde{S}_{n,k}$ has a basis of diagrams without circles or crossings, and in particular these diagrams form a spanning set, so these diagrams remain a spanning set in the quotient space $S_{n,k}$. Now we see that we may eliminate diagrams that have arcs connecting the two projectors. Suppose that $d$ is a diagram in $S_{n,k}$ without circles or crossings that does have an arc between its two projectors. Label the fixed points of $d$ by $1, \ldots, 2(n+k)$ from left to right. Since $d$ has no crossings, it must have an arc with endpoints $1$ and $2(n + k)$. Let $d' \in  \widetilde{S}_{n,k-1}$ be the diagram obtained from $d$ by only looking at the arcs with endpoints $2, \ldots, 2(n + k) - 1$. Then it is clear that $\varphi_{n,k-1}(d') = d$, so that $d$ is trivial in $S_{n,k}$.
\end{proof}

\section{The map from $\widetilde{R}^q_{n,k}$ to $\widetilde{S}_{n,k}$}
\label{sec:psi-tilde}

In this section we define a map $\widetilde{\psi}_{n,k}: \widetilde{R}^q_{n,k} \to \widetilde{S}^q_{n,k}$. Ultimately, we would like this map to descend to one from $R^q_{n,k}$ to $S^q_{n,k}$:

\centerline{
\xymatrix{
\widetilde{R}^q_{n,k} \ar[r]^{\widetilde{\psi}_{n,k}} \ar@{->>}[d] & \widetilde{S}^q_{n,k} \ar@{->>}[d] \\
R^q_{n,k} \ar@{-->}[r]^{\psi_{n,k}} & S^q_{n,k}
}
}

The map from $R^q_{n,k}$ to $S^q_{n,k}$ is the subject of the following section.

\begin{definition}
For $x$ a diagram of $\widetilde{R}^q_{n,k}$, define $\widetilde{\psi}_{n,k}(x)$ as follows:
\begin{itemize}
\item Add $k$ fixed points on each side of the existing $2n$ fixed points in $x$ and surround each new group by a projector.
\item Do not modify the undotted arcs of $x$.
\item Expand the dotted arcs of $x$ one at a time from left to right. If a dotted arc connects endpoints labeled $i$ and $j$ in $x$ ($i < j$), replace it with two new arcs: one from $i$ to the rightmost free fixed point in the left projector and one from $j$ to the rightmost free fixed point in the right projector. The following conventions for crossings between expanded arcs are used:
\begin{itemize}
\item If a dotted arc of $x$ is nested inside any undotted arcs, the left arc in the expansion should pass underneath the undotted arcs and the right arc in the expansion should pass over the undotted arcs.
\item The expansion of a dotted arc of $x$ should not intersect any undotted arc that did not contain the dotted arc.
\item If a pair of dotted arcs of $x$ is nested, then the arcs in their expansions are disjoint.
\item If a pair of dotted arcs of $x$ is unnested, then the only point of intersection of the arcs in their expansion should be where the right expanded arc of the left dotted arc passes above the left expanded arc of the right dotted arc.
\end{itemize}
\item For each pair of nested dotted arcs in $x$, $\widetilde{\psi}_{n,k}(x)$ gains a coefficient of $q^{-1/2}$. (Note that when $q=1$ we may ignore this step.)
\end{itemize}
An example is shown in Figure \ref{expansion_ex}. The map $\widetilde{\psi}_{n,k}$ is then extended to all of $\widetilde{R}^q_{n,k}$ by linearity.
\end{definition}

\begin{figure}[H]
\centering
\includegraphics[scale=1]{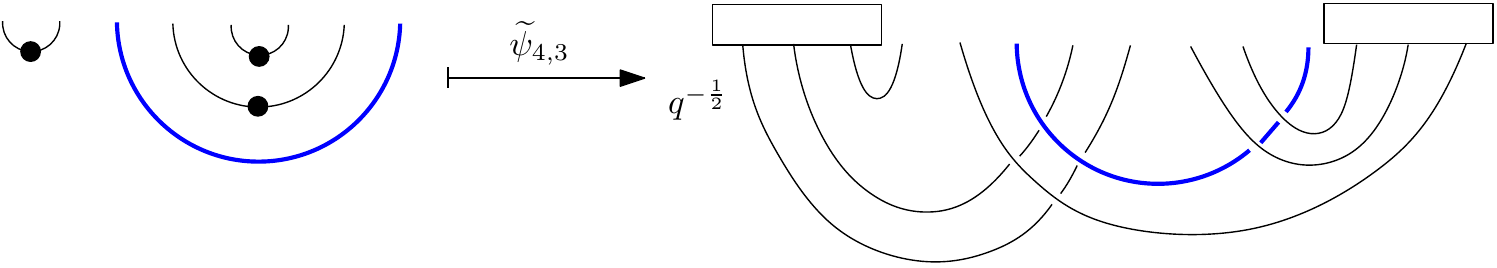}
\caption{The expansion of an element of $\widetilde{R}^q_{4,3}$. The undotted arc is bolded and highlighted blue for clarity.}
\label{expansion_ex}
\end{figure}

\section{The relationship between $R^q_{n,k}$ and $S^q_{n,k}$}
\label{sec:psi}

We will show that the map $\widetilde{\psi}_{n,k}$ is indeed well-defined on the quotient space $R_{n,k}^q$, giving an embedding of the quantum Russell skein module inside one which respects local traditional Kauffman--Lins diagrammatic relations.

Prior to proving this result, we need to establish a lemma concerning a particular type of diagram $d$ in the Kauffman--Lins space $S_{n,k}^q$, shown in Figure \ref{fix1}. 

\begin{figure}[H]
\centering
\includegraphics[scale=1]{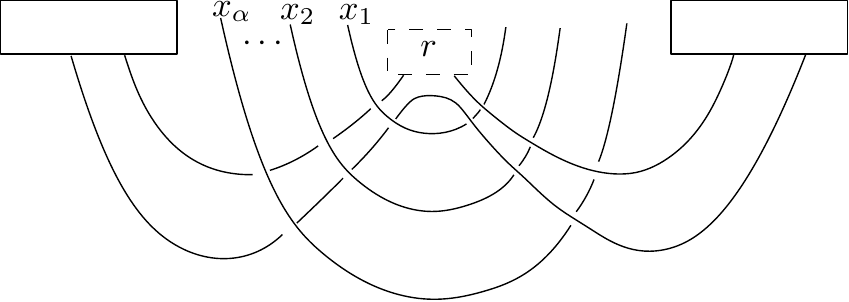}
\caption{The diagram $d \in S_{n,k}^q$.}
\label{fix1}
\end{figure}

For clarity, not all arcs of $d$ are pictured. All arcs that are interwoven with the arc between the two projectors are shown and are labeled $x_1, x_2, \ldots, x_\alpha$. The region labeled $r$ is meant to represent the sub-diagram present inside the arc labeled $x_1$. The arcs from the box $r$ to the left and right projectors represent some number $|r|$ of parallel arcs. Also not shown are sub-diagrams in between the arcs labeled $x_1, \ldots, x_\alpha$. A closeup of the region between arcs labeled $x_{i-1}$ and $x_i$ is shown in Figure \ref{fix2}.

\begin{figure}[H]
\centering
\includegraphics[scale=1]{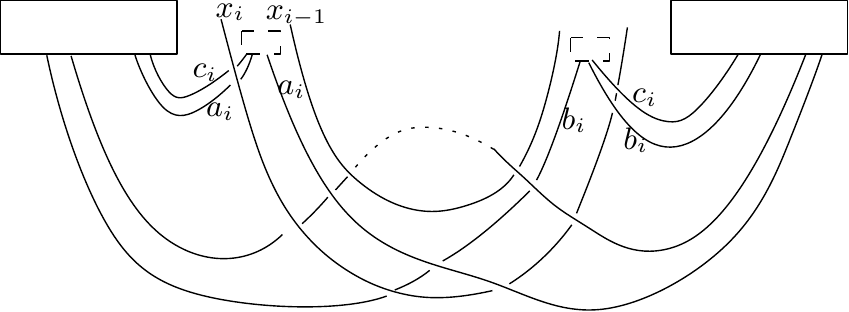}
\caption{A closeup of the region between the arcs $x_{i-1}$ and $x_i$ in $d$.}
\label{fix2}
\end{figure}
For $2 \leq i \leq \alpha$, the labels $a_i, b_i,$ and $c_i$ denote the number of parallel strands represented by the single strand adjacent to the label. There may be additional arcs and crossings inside the dotted boxes, but all arcs leaving those regions and entering one of the two projectors are shown. Again, this is a closeup of only a part of the diagram $d$ in Figure \ref{fix1}- the arcs present inside arc $x_{i-1}$ and outside arc $x_i$ remain.

\begin{lemma}
\label{lem:fix1}
In the space $S_{n,k}^q$, we have the equality
\[ d = (1-q^{-2}) \sum_{i=1}^{\alpha} q^{i-1}q^{\sum_{j=2}^i (a_j+b_j)/2} d_i,\]
where $d_i$ is the diagram shown in Figure \ref{fix3}.
\begin{figure}[H]
\centering
\includegraphics[scale=1]{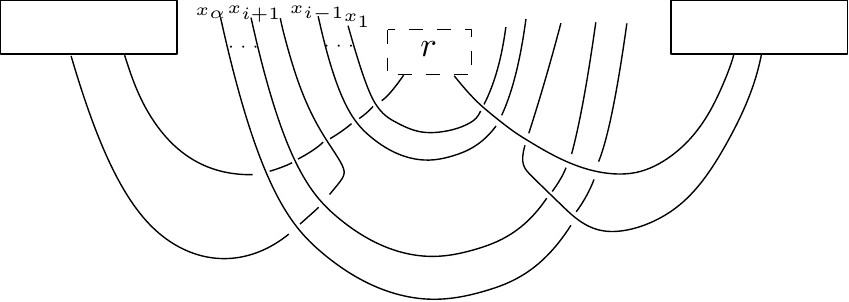}
\caption{The diagrams $d_i$ of Lemma \ref{lem:fix1}.}
\label{fix3}
\end{figure}
\end{lemma}
Note that in the diagram $d_i$, the region $r$ and the regions in between the arcs $x_1, \ldots, x_\alpha$ and outside $x_\alpha$ are identical to those in $d$ (not pictured).

\begin{proof}
Proof is by induction on $\alpha$. First suppose that $\alpha = 1$. Then we see
\begin{figure}[H]
\centering
\includegraphics[scale=.7]{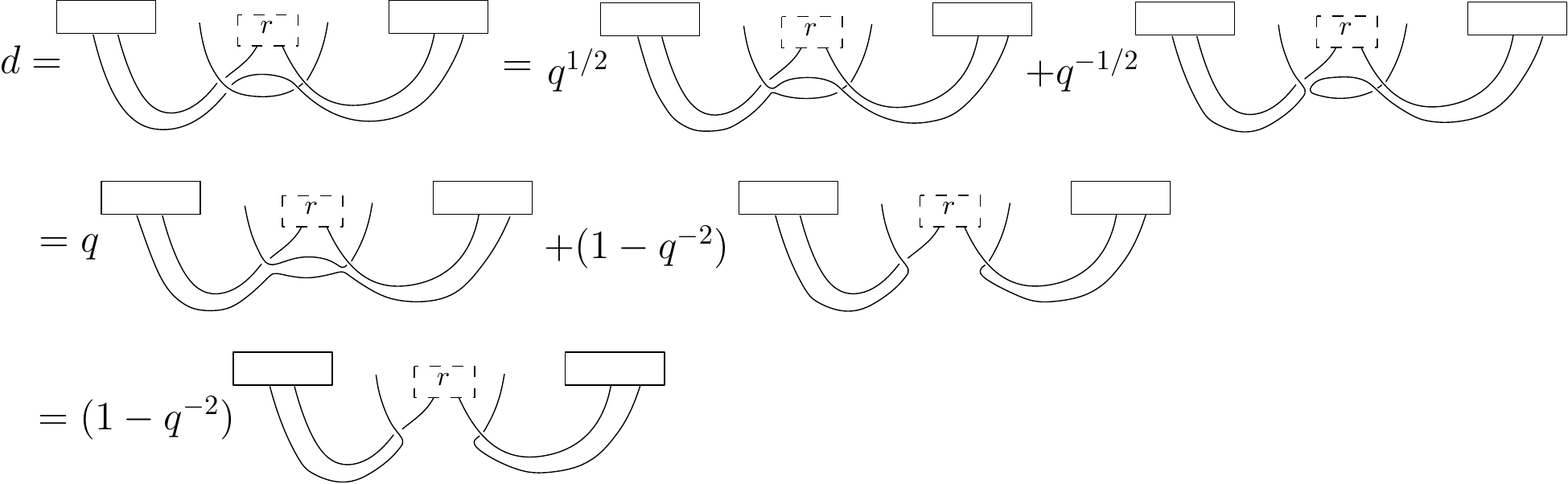}
\end{figure}

where the last equality follows from the fact that a diagram with an unwoven arc between the two projectors is trivial in $S_{n,k}^q$.

Next suppose that the statement has been proven for $\alpha-1$ arcs. Then by expanding the two crossings between the arc labeled $x_1$ and the interwoven arc between the two projectors, we see the following, where the highlighted region of a given diagram indicates the crossing(s) that are resolved to obtain the following step:
\begin{figure}[H]
\centering
\includegraphics[scale=.8]{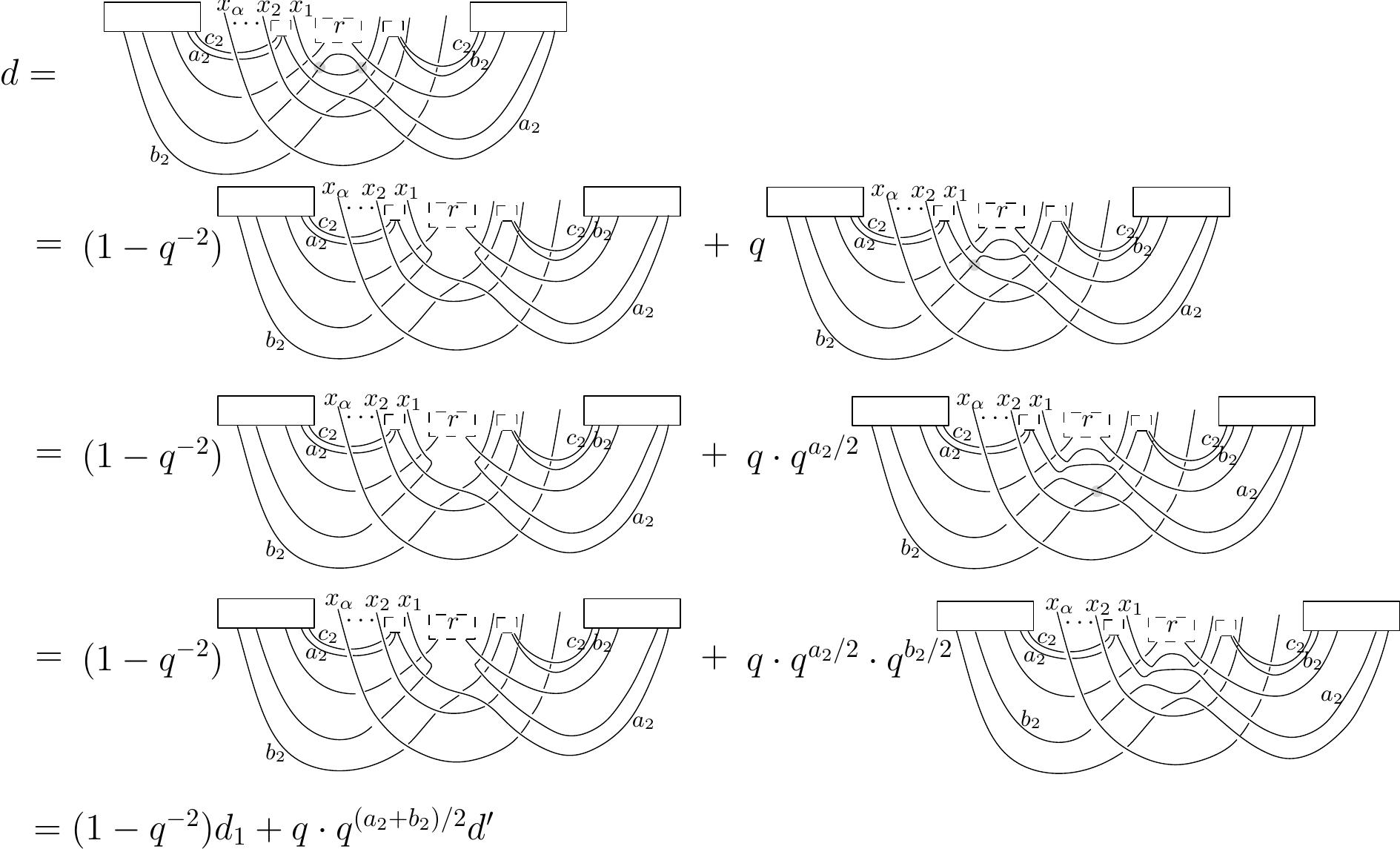}
\end{figure}
where $d_1$ and $d'$ are the two diagrams appearing in the previous line.

We observe that the diagram $d'$ has one fewer arc interwoven with the arc between the two projectors. Therefore by the inductive hypothesis, we see that
\begin{eqnarray*}
d &=& (1-q^{-2})d_1 + q \cdot q^{(a_2+b_2)/2} d' \\
&=& (1-q^{-2})d_1 + q \cdot q^{(a_2+b_2)/2} \cdot (1-q^{-2}) \cdot \sum_{i=2}^{\alpha} q^{i-2} q^{\sum_{j=3}^i(a_j+b_j)/2} d_i \\
&=& (1-q^{-2})d_1 + (1-q^{-2}) \cdot \sum_{i=2}^{\alpha} q^{i-1}q^{\sum_{j=2}^i (a_j+b_j)/2} d_i \\
&=& (1-q^{-2}) \sum_{i=1}^{\alpha} q^{i-1}q^{\sum_{j=2}^i (a_j+b_j)/2} d_i .
\end{eqnarray*}
$\blacksquare$
\end{proof}

\begin{theorem}
\label{key-thm}
The map $\widetilde{\psi}_{n,k}$ descends to a well-defined map from $R^q_{n,k}$ to $S^q_{n,k}$, which we call $\psi_{n,k}$.
\end{theorem}

\begin{proof}
We need to show that for any quantum Type I or Type II Russell relation, the elements in $S_{n,k}^q$ obtained by applying $\widetilde{\psi}$ to the left and right-hand sides agree. For the quantum Type I relation, let $m$ denote the diagram on the left-hand side and $m_i$ the $i$th diagram appearing in the fourth term on the right-hand side in Figure \ref{quantum_Russell} (i.e, that corresponding to the index $i$ in the summation). Figure \ref{Russell_trans_GEN} shows the application of $\widetilde{\psi}$ to $m$. In this picture, the boxes labeled $A, B, C$ are meant to be generic representations of dotted crossingless matchings. The thick lines coming out of each dotted box after $\widetilde{\psi}$ is applied represent several parallel strands that appear when the dotted arcs inside the boxes are expanded. For the box $A$, the number $|A|$ is defined to be the number of dotted arcs that appear in the dotted crossingless matching that $A$ represents, and similarly for $B, C$. Also, for any diagram $M$ in the Russell space,
\[p(M) := -(\# \mbox{ pairs of nested dotted arcs in } M)/2,\]
i.e., the coefficient of the diagram in $\psi(M)$. The unlabeled dotted region represents the image under $\widetilde{\psi}$ of any additional arcs that might appear outside of the region $(a,d)$. The crossings that are expanded from one step to the next are highlighted. 

\begin{figure}[H]
\centering
\includegraphics[height= 8 in]{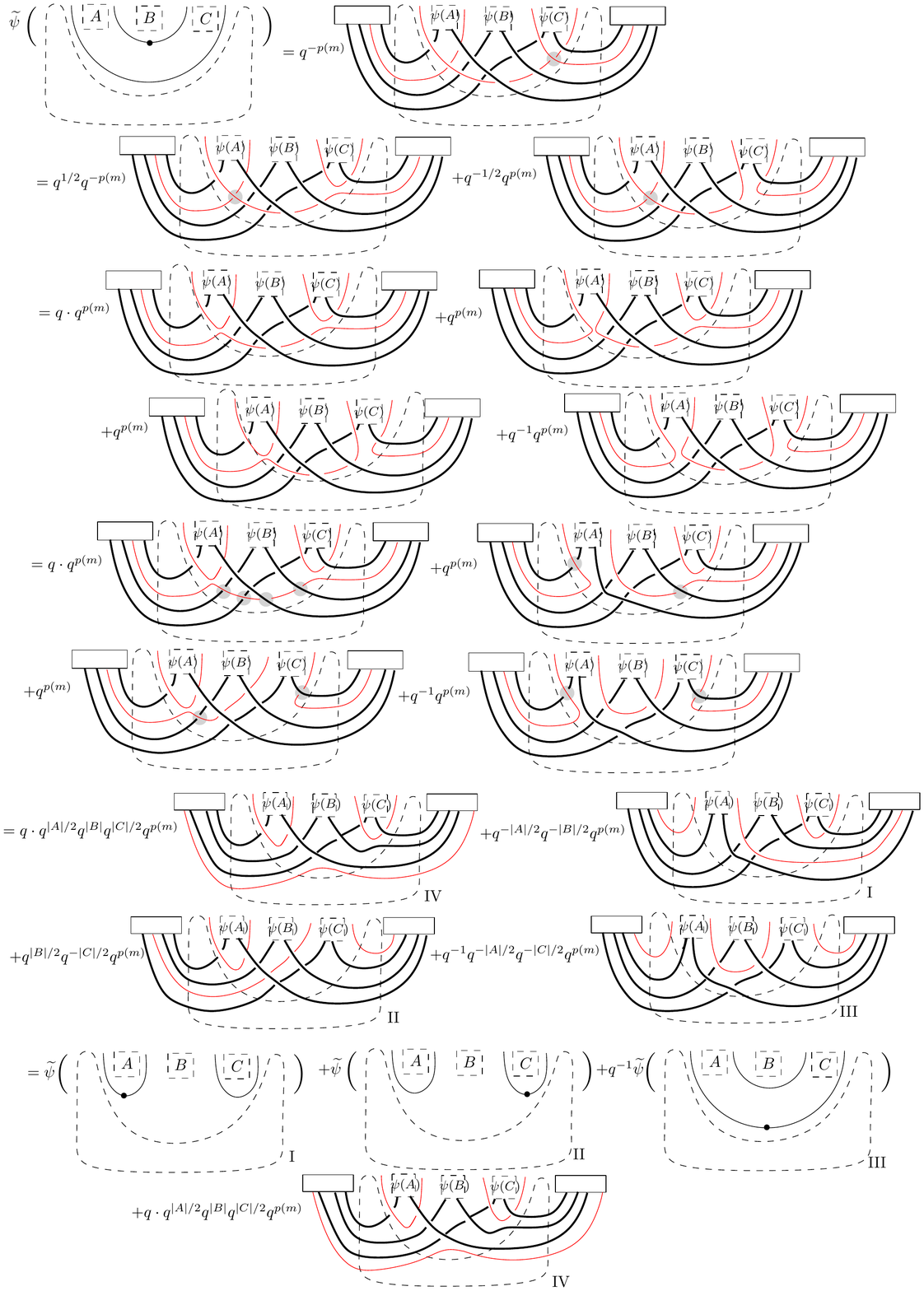}
\caption{The Type I quantum Russell relations are preserved by $\widetilde{\psi}$.}
\label{Russell_trans_GEN}
\end{figure}

From the result of the computations of Figure \ref{Russell_trans_GEN}, we recover the first three terms appearing in the quantum Type I relation. The fourth term appearing is almost of the form addressed in Lemma \ref{lem:fix1}, with the arcs $x_1, \ldots, x_\alpha$ not pictured but belonging to the outside dotted region. Denote the endpoints of each arc $x_i$ by $(x_{i_1}, x_{i_2})$. To get our fourth term into the form of Lemma \ref{lem:fix1}, we must untangle it from the expansion of any dotted arcs (not pictured) with both endpoints between $x_{1_1}$ and $a$, the number of which we denote $a_1$, as well as any with both endpoints between $d$ and $x_{1_2}$, the number of which we denote $b_1$. The result of applying Lemma \ref{lem:fix1} is analyzed separately in Figure \ref{q_Russell_pres_fix_2}. For $2 \leq i \leq \alpha$, we define $a_i$ to be the number of dotted arcs in $m$ with endpoints $(y_1, y_2)$ such that $x_{i_1} < y_1,y_2 < x_{(i-1)_1}$. Similarly, $b_i$ is the number of dotted arcs with $x_{(i-1)_2} < y_1, y_2 < x_{i_2}$. This is consistent with the $a_i$ and $b_i$ that appear in Lemma \ref{lem:fix1}.

\begin{figure}[H]
\centering
\includegraphics[width=6.5 in]{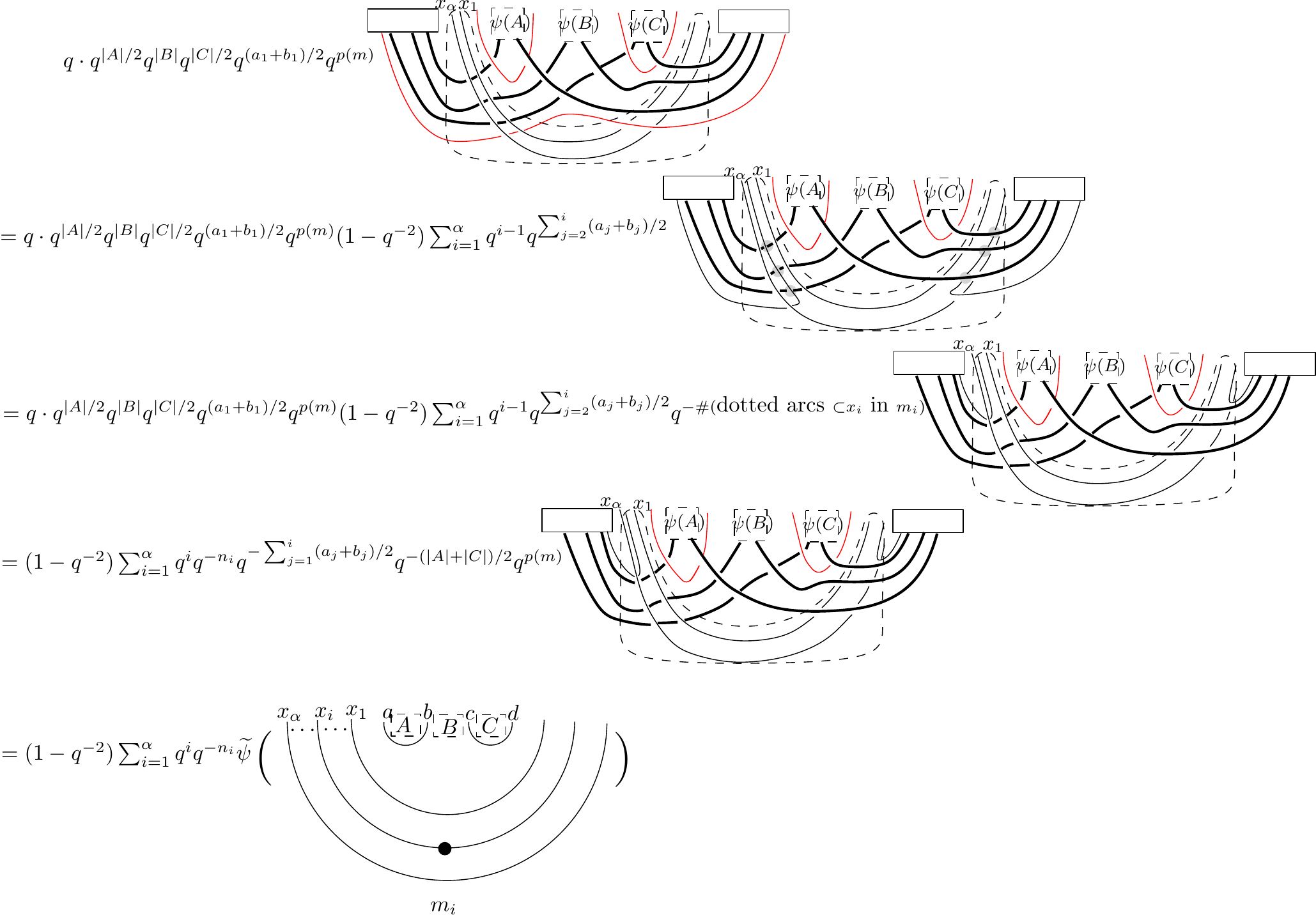}
\caption{Continuation of proof that Type I quantum Russell relations are preserved by $\widetilde{\psi}$.}
\label{q_Russell_pres_fix_2}
\end{figure}

To reach the penultimate equality of Figure \ref{q_Russell_pres_fix_2}, note that inside the sum we have a factor of $q^{-1}$ for each of the dotted arcs inside $i$ in $m_i$, which we classify into four types: those that are inside $B$, those that are inside $A$ or $C$, those that have both endpoints to the left of $a$ or to the right of $d$, and those that are nested in between $i$ and the arc $(a,d)$. The terms coming from the first class cancel with the $q^{|B|}$, and those in the second and third classes partly cancel with the $q^{|A|/2}, q^{|C|/2}, $ and $q^{\sum_{j=1}^{\alpha}(a_j+b_j)/2}$. There are then exactly $n_i$ terms which do not cancel with anything, where $n_i$ is the same as appears in the definition of the Type I and Type II quantum Russell relations of Section \ref{sec:qRussell}.

The final equality follows from the fact that, if $m_i$ is the $i$th diagram appearing in the summation in the Type I relation, then
\[p(m_i) = p(m)-(|A|+|C|)/2 - \sum_{j=1}^{i}(a_j+b_j)/2,\]
since $m_i$ has $|A|+|C|+\sum_{j=1}^i(a_j+b_j)$ more pairs of nested dotted arcs than $m$.

The proof for the quantum Type II relations is analogous.

\begin{comment}
the corresponding argument is illustrated in Figure \ref{Russell_trans_GEN2}. 
\begin{figure}[H]
\centering
\includegraphics[scale=.8]{qRussell_pres2.pdf}
\caption{The Type II quantum Russell relations are preserved by $\widetilde{\psi}$.}
\label{Russell_trans_GEN2}
\end{figure}

\end{comment}

$\blacksquare$
\end{proof}

\begin{definition}
Let $B_{n,k}$ be the subset of ${R}^q_{n,k}$ consisting of dotted crossingless matchings that only have dots on outer arcs.
\end{definition}

\begin{lemma}
\label{basis-lem}
The images of the elements of the set $B_{n,k}$ under $\psi_{n,k}$ are linearly independent in $S^q_{n,k}$.
\end{lemma}

\begin{proof}
Let $m$ be an element of $B_{n,k}$. Then the only crossings in $\psi_{n,k}(m)$ are between the right arc in the expansion of each dotted arc of $m$ with the left arc in the expansion of any dotted arcs to its right. 

Apply Kauffman--Lins diagrammatic relations to express $\psi_{n,k}(m)$ as a linear combination of diagrams without crossings with coefficients depending on $q$, that is, $\psi_{n,k}(m) =\sum_{i=1}^p c_i(q)D_i$ where the diagrams $D_i$ have no crossings. Note that all but one of the diagrams $D_i$ are trivial because they contain an arc between the two projectors. Suppose the nontrivial term has index $j$, so $\psi_{n,k}(m) = c_j(q)D_j$. 

We claim that the diagram $D_j$ is unique to $m$. Suppose that for some $m_1, m_2$ in $B_{n,k}$, $\psi_{n,k}(m_1)$ and $\psi_{n,k}(m_2)$ are nonzero multiples of the same crossingless diagram $D$. Then the $n-k$ undotted arcs of $m_1$ and $m_2$ must be in the same positions, since $\psi$ leaves these arcs in place. But since $m_1$ and $m_2$ can only have dots on outer arcs, their $k$ dotted arcs have to be in the same positions as well, so $m_1 = m_2$. $\blacksquare$
\end{proof}

\begin{proposition}
\label{basis}
The elements of $B_{n,k}$ form a basis of $R^q_{n,k}$ as a $\mathbb{Z}[q, q^{-1}]$-module.
\end{proposition}

\begin{proof}
First we show that the elements of $B_{n,k}$ form a spanning set of $R^q_{n,k}$. Suppose that a diagram $d$ in $R^q_{n,k}$ has one or more dots on inner arcs. If the nearest arc to a dotted inner arc is undotted, then a Type I quantum Russell relation can be used to rewrite the diagram in terms of those with strictly fewer arcs containing dotted arcs. If the nearest arc is dotted, then we may rearrange the Type II quantum Russell relation and again apply to obtain a linear combination of diagrams with strictly fewer arcs containing dotted arcs:
\begin{figure}[H]
\centering
\includegraphics[scale=.8]{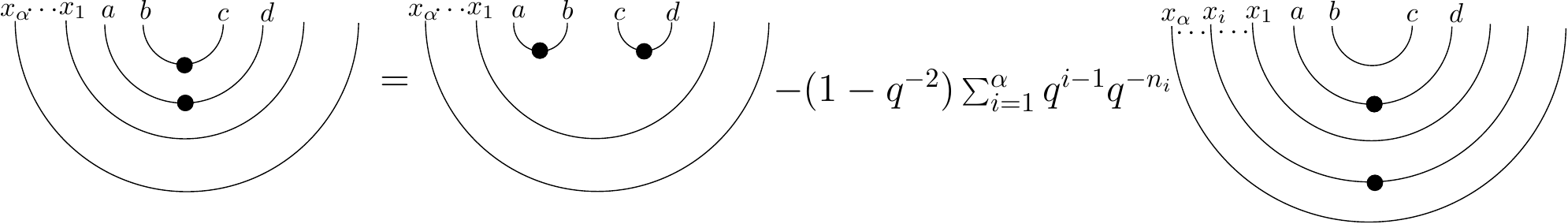}
%\caption{A rearrangement of the Type II quantum Russell relation.}
%\label{type2_rearrange}
\end{figure}

Then Type I and Type II relations can be repeatedly applied until all diagrams have dots on outer arcs only.

Next we show that the elements of $B_{n,k}$ are linearly independent. Suppose to the contrary that there is some dependence relation $f_1(q)b_1 + \cdots + f_r(q)b_r = 0$ in ${R}^q_{n,k}$. Applying $\psi_{n,k}$ to the relation, by the linearity of $\psi_{n,k}$ we would have a dependence relation among $\psi_{n,k}(b_1), \ldots, \psi_{n,k}(b_r)$. However, Lemma \ref{basis-lem} shows that this cannot be. $\blacksquare$
\end{proof}

\begin{proposition}
\label{prop:rank}
The rank of $R^q_{n,k}$ as a $\mathbb{Z}[q, q^{-1}]$-module is given by
\[ \mbox{rank}( R^q_{n,k}) = \left( \begin{array}{c} 2n \\ n+k \end{array} \right) - \left(\begin{array}{c} 2n \\ n+k+1 \end{array} \right) = \frac{2k+1}{n+k+1} \left( \begin{array}{c} 2n \\ n+k \end{array} \right) . \]
\end{proposition}

\begin{proof}
We employ a combinatorial argument, following that of Davis in \cite{catalan}. First we claim that the basis elements of $R^q_{n,k}$ (as given in Proposition \ref{basis}) are in one-to-one correspondence with paths on a grid from $(0,0)$ to $(n+k,n-k)$ that lie entirely on or above the diagonal. Call the space of all such paths $P_{n,k}$. To see the correspondence, we define inverse maps $\alpha: B_{n,k} \to P_{n,k}$ and $\beta: P_{n,k} \to B_{n,k}$. Given $m$, a dotted crossingless matching in $B_{n,k}$, look at its endpoints from left to right. Define $\alpha(m)$ to be the path that steps down for each right endpoint of an undotted arc and right for all other endpoints. It can be readily checked that such a path lies entirely on or above the diagonal. Given a path $p$ in $P_{n,k}$, translate it into a sequence of up and down arrows where, starting from $(0,0)$, each step right becomes an up arrow and each step down becomes a down arrow. Define $\beta(p)$ as follows: for each down arrow, draw an undotted arc connecting it with the first up arrow to its left. Connect the remaining up arrows in pairs from left to right with dotted arcs. It can again be readily checked that $\beta(p)$ is indeed an element of $B_{n,k}$ and that $\alpha$ and $\beta$ are inverses. An example is shown in Figure \ref{dim-ex}.

\begin{figure}[H]
\centering
\includegraphics[scale=1]{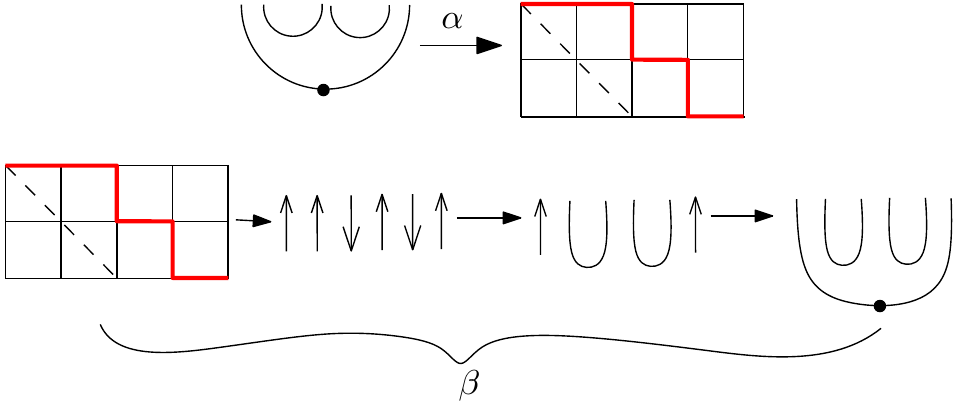}
\caption{An example of the maps $\alpha$ and $\beta$ with $n =3, k =1$.}
\label{dim-ex}
\end{figure}

Now that the bijection has been established, we can count the paths in $P_{n,k}$. To do this, we count all paths from $(0,0)$ to $(n+k,n-k)$ and then subtract the ``bad'' ones, that is, the ones that cross the diagonal. It is clear that the number of all such paths is $\left( \begin{array}{c} 2n \\ n+k \end{array} \right)$, since a path contains $2n$ total steps, $n+k$ of which are to the right. For any bad path, there must be a first point below the diagonal: label this point $p$. Reflecting the portion of the path after $p$ (i.e. switching right with down in the steps that follow) yields a new path from $(0,0)$ to $(n-k-1,n+k+1)$. This is a reversible process, so the number of bad paths is the same as the number of all paths from $(0,0)$ to $(n-k-1,n+k+1)$, which is $\left( \begin{array}{c} 2n \\ n+k+1 \end{array} \right)$. $\blacksquare$
\end{proof}

\begin{remark}
There is also a representation-theoretic way of seeing the above statement. In \cite{russ-2}, Russell and Tymoczko show that the subspace of the original Russell space $R_{n,k}$ spanned by diagrams with $2n$ endpoints with $k$ dots only on outer arcs has an action of the symmetric group $S_{2n}$ and corresponds to the same irreducible representation as the Young diagram with two rows of size $n+k$ and $n-k$. The standard hook-length formula on this representation gives the same answer.
\end{remark}

Combining several results from this section, in conclusion we have the following.

\begin{proposition}
The map $\psi_{n,k} : R_{n,k}^q \otimes_{\mathbb{Z}[q,q^{-1}]} \mathbb{Z}[q^{1/2}, q^{-1/2}] \to S^q_{n,k}$ is an isomorphism.
\end{proposition}

\begin{proof}
Theorem \ref{key-thm} shows that this map is well-defined. We have established that the elements of the set $B_{n,k}$ form a basis of $R_{n,k}^q$ and that the images of these elements under $\psi_{n,k}$ are linearly independent, so $\psi_{n,k}$ is injective. To see that $\psi_{n,k}$ is surjective, let $m$ be an element of the basis of $S^q_{n,k}$ given in Proposition \ref{S-basis}, that is, a crossingless matching containing no arcs that have both endpoints inside one of the projectors. From $m$ we will construct an element $b$ in $R_{n,k}^q$ such that $\psi_{n,k}(b) = m$. First, the undotted arcs of $b$ will exactly match the arcs of $m$ that have neither endpoint inside a projector. Next, if any arcs remain in $m$, find a pair of arcs $(a, a')$ such that $a$ has left endpoint inside the left projector, $a'$ has right endpoint inside the right projector, and the only arcs between $a$ and $a'$ have neither endpoint inside a projector. It is clear that such a pair must exist since $m$ has no crossings. If $a$ has right endpoint in position $i$ and $a'$ has left endpoint in position $j$ we replace $a$ and $a'$ with a dotted arc from $i$ to $j$ and a factor of $q^{1/2}$. We repeat this process until no more pairs $(a, a')$ exist, at which point it is clear from the definition of $\psi$ that $\psi_{n,k}(b) = m$.
\end{proof}

\section{A braid group action on $R^q_{n,k}$ for generic $q$}
\label{sec:Bn-act}

\subsection{A local braid group action on an isomorphic space}

We now have an embedding of a quantum deformation of the Russell skein module into a space that obeys Kauffman--Lins relations. Since Kauffman--Lins diagrammatics for generic $q$ satisfy the braid group relations, we look to define a braid group action on the quantum Russell space. 

Further, we would like the action to be local in the sense that applying the braid group generator $\sigma_i$ to a diagram $m$ produces a linear combination of diagrams that are the same as $m$ apart from arcs that have either $i$ or $i+1$ as endpoints. We find that there is a local action of the braid group on the subspace of $\widetilde{R}^{q}_{n,k}$ spanned by elements of $B_{n,k}$ but that the construction does not extend to the entire space $\widetilde{R}^q_{n,k}$.

\begin{comment}
After all, the Kauffman--Lins diagrammatics for generic $q$ satisfy the braid group relations, and when $q=1$ we get a diagrammmatic presentation of the symmetric group that was used in the previous section. First we observe that the $S_{2n}$ action on $\widetilde{S}^{q=1}_{n,k}$ and $S^{q=1}_{n,k}$ described above generalizes to an action of $B_{2n}$, the braid group on $2n$ strands, on $\widetilde{S}_{n,k}$.
\end{comment}

The following proposition describes the braid group action on the Kauffman--Lins space $S_{n,k}$.

\begin{proposition}
 The braid group $B_{2n}$ acts on $\widetilde{S}_{n,k}$ as follows: a generator $\sigma_i$ of $B_{2n}$ acts on a diagram $s$ in $\widetilde{S}_{n,k}$ by attaching a positive crossing above the $i$th and $i+1$st fixed points of $s$ (numbering left to right starting at the first point after the left projector) and then expanding the crossing using Kauffman--Lins diagrammatics. This action descends to the quotient space $S_{n,k}$.
\end{proposition}

\begin{proof}
That this defines a braid group action on $\widetilde{S}_{n,k}$ follows from basic properties of Kauffman--Lins diagrammatics. To see that the action descends to $S_{n,k}$,  observe that if a diagram in $\widetilde{S}^{q}_{n,k}$ has an arc between the two projectors, then it will still have such an arc when a crossing is applied, since crossings are never attached to projectors.
\end{proof}

\begin{comment}
\begin{remark}
When $q=1$, the diagrammatics corresponding to $\sigma_i$ and $\sigma_i^{-1}$ are the same, so the braid group action descends to one of the symmetric group, which is the same as the one described in the previous section.
\end{remark}
\end{comment}

\begin{definition}
Let $\overline{R}^q_{n,k}$ be the subspace of $\widetilde{R}_{n,k}^q$ that consists of linear combinations of elements of $B_{n,k}$ with coefficients in $\mathbb{Z}[q, q^{-1}]$.
\end{definition}

Since the basis of $\overline{R}_{n,k}^q$ is the same as that of $R_{n,k}^q$, we have the following relationship among all three spaces:

\centerline{
\xymatrix{
\overline{R}_{n,k}^q \ar@{^{(}->}[r]  \ar@/_1pc/[rr]_{\cong} & \widetilde{R}_{n,k}^q \ar@{->>}[r] & R_{n,k}^q \\
}}

We now look to define a braid group action on the space $\overline{R}^q_{n,k}$ such that the map $\pi \circ \widetilde{\psi}_{n,k}|_{\overline{R}^q_{n,k}}$, where $\pi$ is the quotient map from $\widetilde{S}_{n,k}$ to $S_{n,k}$, commutes with the braid group actions on $\overline{R}^q_{n,k}$ and $S_{n,k}$:

\centerline{
{\xymatrixcolsep{8pc}
 \xymatrix{
\widetilde{R}^q_{n,k} \ar[rd]^{\pi \circ \widetilde{\psi}_{n,k}} \\
\overline{R}^q_{n,k}  \ar@{^{(}->}[r]^{\pi \circ \widetilde{\psi}_{n,k}|_{\overline{R}^q_{n,k}}}  \ar@{_{(}->}[u] \ar@(dl,ul)@{.>}[]^{B_{2n}} & S_{n,k} \ar@(dr,ur)[]_{B_{2n}} \\
}} }

In the proof of Lemma \ref{basis-lem}, we showed that the images of the generators of ${\overline{R}^q_{n,k}}$, that is, the elements of $B_{n,k}$, are linearly independent in $S_{n,k}$. Therefore $\psi_{n,k}|_{\overline{R}^q_{n,k}}$ is injective. To define the braid group action on each basis element $b$ of $\overline{R}^q_{n,k}$, then, we map $b$ to $S_{n,k}$, apply the braid group action there, and then pull the result back to $\overline{R}^q_{n,k}$.

\begin{proposition}
\label{braid-group-act}
There is a local braid group action on $\overline{R}^q_{n,k}$ given in Figure \ref{braid_group_act}. The map $\widetilde{\psi}_{n,k}|_{\overline{R}^q_{n,k}} : \overline{R}^q_{n,k} \to S_{n,k}$ respects the action of $B_{2n}$. 
\end{proposition}

\afterpage{
\begin{figure}
\centering
\includegraphics[scale=1]{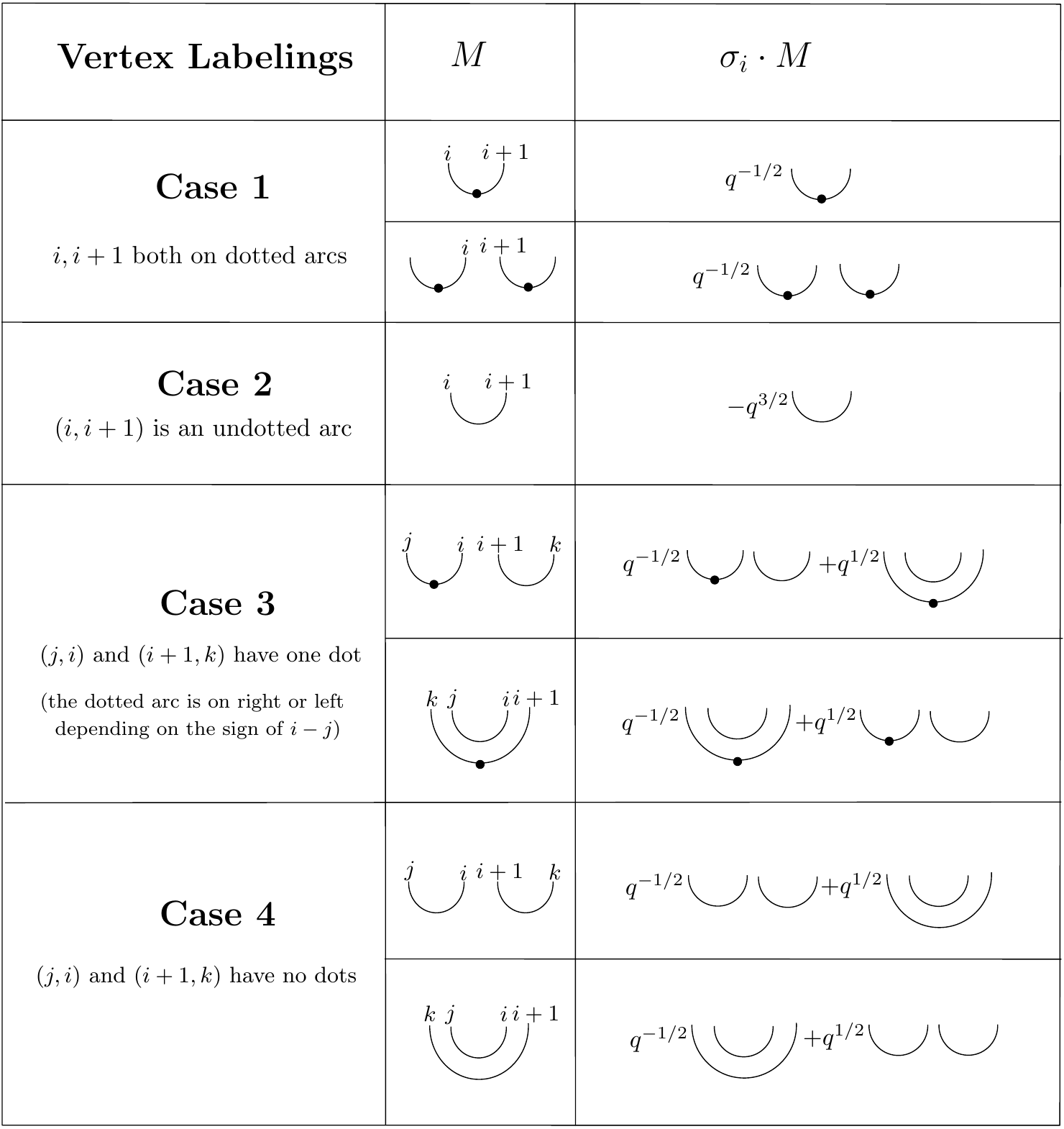}
\caption{The local quantum braid group action on $\overline{R}^q_{n,k}$.}
\label{braid_group_act}
\end{figure}}

\begin{proof}
We need only show that the construction of the action described above, which involved first passing to $S_{n,k}$, performing the braid group action there, and then pulling back to $\overline{R}_{n,k}$, gives a local action on $\overline{R}_{n,k}$ and that it is the one pictured in Figure \ref{braid_group_act}. If this is the case, it is obvious by construction that $\widetilde{\psi}_{n,k}|_{\overline{R}_{n,k}}$ respects the braid group action.

That this action is local and that it is the one pictured in Figure \ref{braid_group_act} can be proven on a case-by-case basis. As an example of the argument used, we establish this for the second picture in Case 3. The argument is illustrated in Figure \ref{braid_group_ex}. The picture inside $\psi$ on the left-hand side is meant to represent a generic element of the type in Case 3, which has an undotted arc with left endpoint in position $i+1$ nested inside a dotted arc with left endpoint in position $i$. The dashed boxes labeled $A$ and $D$ may contain any dotted crossingless matching with dots on outer arcs. The dashed circles labeled $B$ and $C$ may contain any crossingless matching without dots, since they are contained within a dotted arc. The thick lines on the right-hand side represent parallel strands coming from the expansion of dotted arcs inside boxes $A$ and $D$. We see that the diagrams coming from the expansion of the crossing in $S_{n,k}$ pull back to the desired ones in $\overline{R}_{n,k}$. The proofs for the other cases are completely analogous. $\blacksquare$
\end{proof}

\begin{figure}[H]
\centering
\includegraphics[scale=.8]{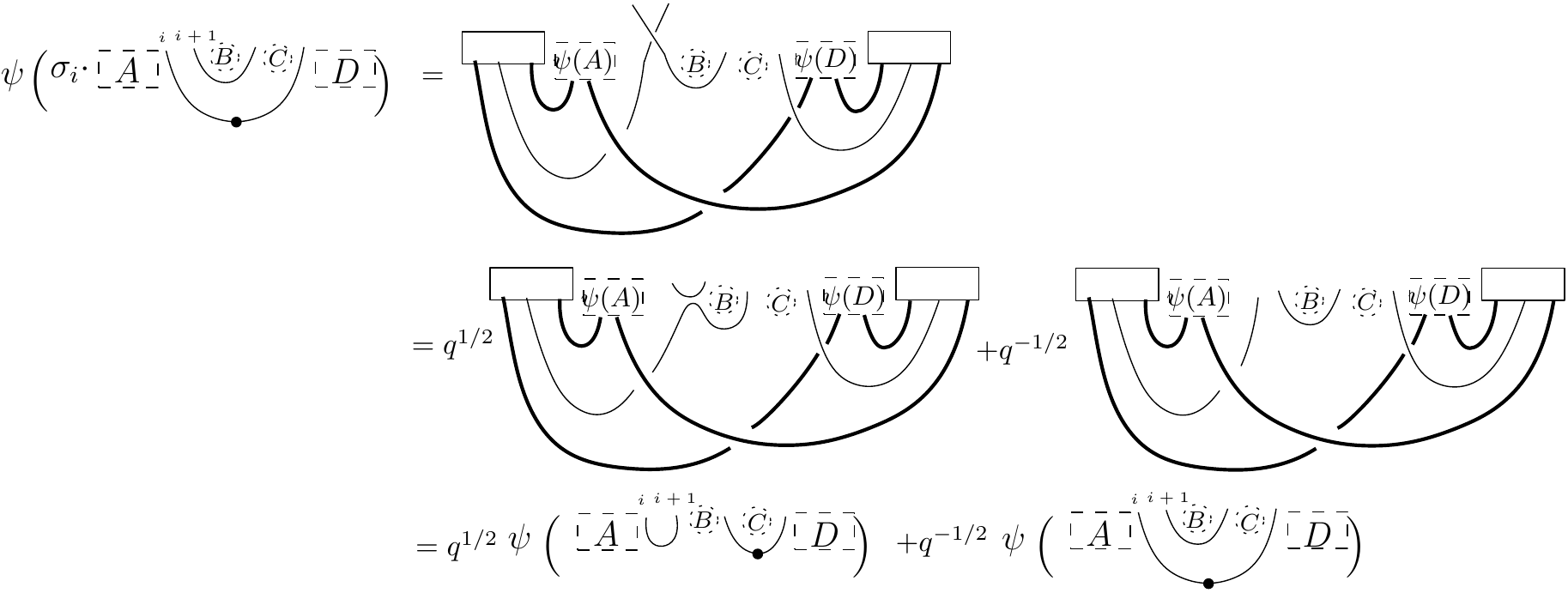}
\caption{A sample computation showing that the action of $B_{2n}$ on $\overline{R}^q_{n,k}$ constructed via $\psi$ is local and as pictured in Figure \ref{braid_group_act}.}
\label{braid_group_ex}
\end{figure}

\subsection{A semi-local braid group action on $R^q_{n,k}$}

In the previous section, we saw that by restricting to a certain class of diagrams in $R^q_{n,k}$, which happen to form a basis, we get a local braid group action on an isomorphic space by looking at the braid group action on the images of these diagrams in $S^q_{n,k}$ and then pulling back the result. In this section, we investigate what happens when we pull back the braid group action on the image of any diagram in $R^q_{n,k}$. We find that for 6 of 9 cases, the action remains completely local in the same sense described above. For the remaining three cases, the action is semi-local in the same sense that the quantum Russell relations are: an additional term appears for every undotted arc containing the arcs of interest.

\begin{theorem}
Figure \ref{braid_group_act_gen} describes a semi-local action of $B_{2n}$ on $R^q_{n,k}$.
\end{theorem}
In the three cases of Figure \ref{braid_group_act_gen} in which a term involving a summation of diagrams appears, the arcs labeled $x_{1}, \ldots, x_{\alpha}$ form the complete set of undotted arcs containing the arcs in $M$ with endpoints $i$ or $i+1$. We denote the endpoints of the arc $x_j$ by $(x_{j_1}, x_{j_2})$ and define $n_j$ to be the number of dotted arcs in $M$ with left endpoint between $x_{j_1}$ and arcs with endpoints $i$ or $i+1$ and right endpoint between arcs with endpoints $i$ or $i+1$ and $x_{j_2}$. This definition is analogous to that which appears in the quantum Type I and Type II Russell relations. In all cases, the diagrams appearing in $\sigma_i \cdot M$ are identical to $M$ apart from the arcs shown.

\afterpage{
\begin{figure}
\centering
\includegraphics[scale=.85]{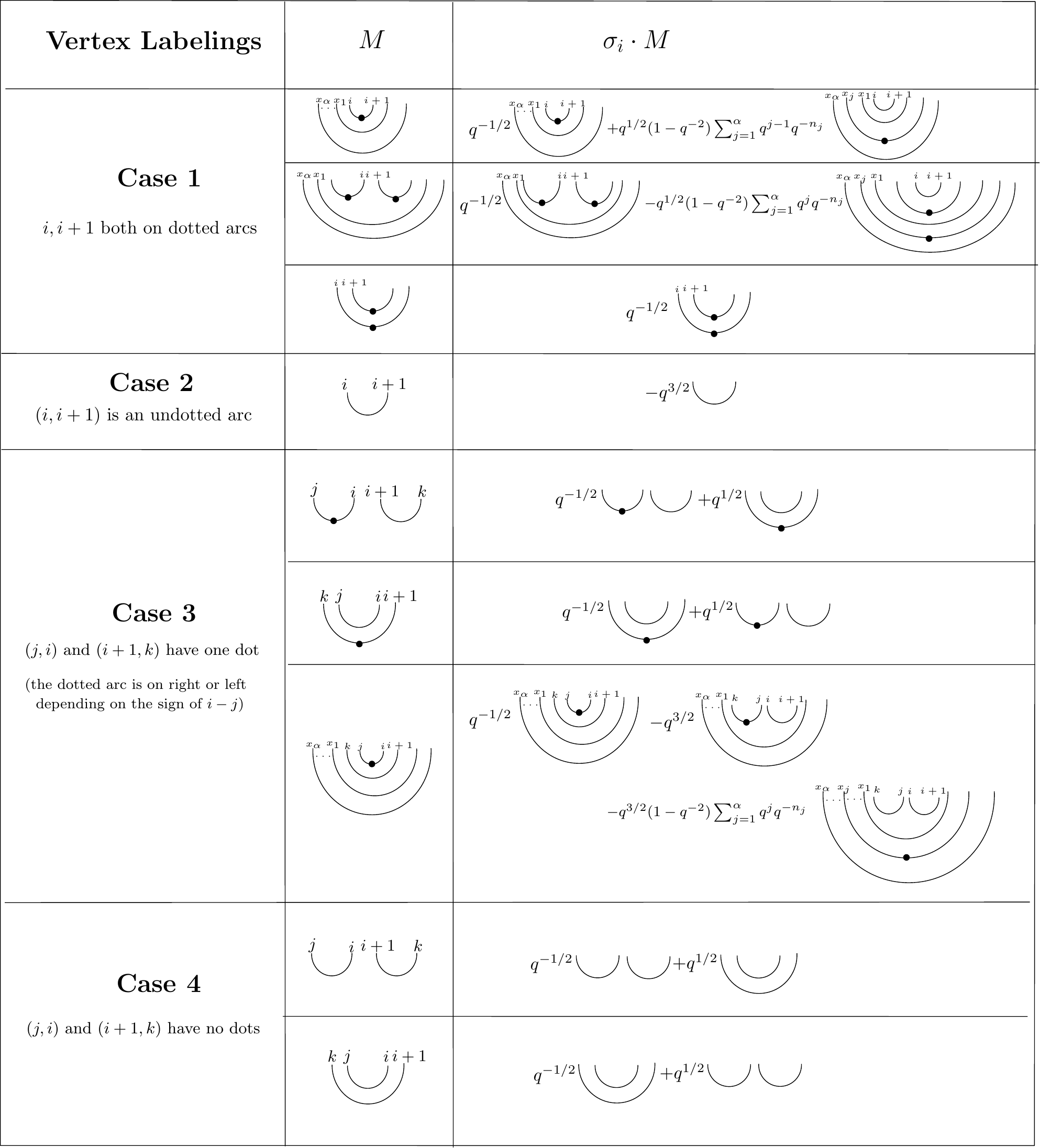}
\caption{The semi-local action of $B_{2n}$ on $R^q_{n,k}$.}
\label{braid_group_act_gen}
\end{figure}
}

\begin{proof}
This generalization of Proposition \ref{braid-group-act} follows by the same argument: in each case we map the diagram $M$ to $S_{n,k}^q$ via $\psi$, look at the braid group action there, and pull the result back. The nonlocal terms that appear in three of the cases arise from the appearance of the same type of diagram in $S_{n,k}^q$ that we saw in the proof of Theorem \ref{q_Russell_pres_fix_2}, that is, that of the form shown in Figure \ref{fix3}. We omit the details of this calculation.
\end{proof}

\begin{remark}
At this point, diagrams of the type in Figure \ref{fix3} have shown up several times and contributed to nonlocal terms both in the quantum Russell relations and in the braid group action. One might hope that the obstruction to locality could be removed if we alter our definition of $\psi$. A natural choice in a modified map $\psi'$ would be to make both arcs in the expansion of a dotted arc nested in an undotted arc both go either over or under the undotted arc. Then, when diagrams in $\widetilde{R}_{n,k}^q$ are mapped to $S_{n,k}^q$ and crossings are expanded, it would be impossible to have an arc between the two projectors that was interwoven with an undotted arc, and any such diagram would always be trivial.

Suppose we chose the expanded arcs to go under the undotted arc. Then, we would have, for example: 
\begin{figure}[H]
\centering
\includegraphics[scale=1]{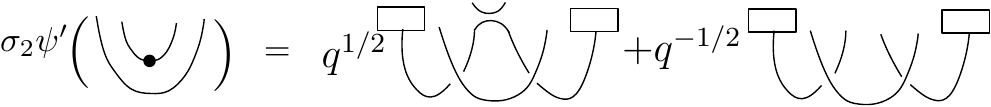}
%\caption{}
%label{}
\end{figure}
\end{remark}
The first term would disappear as desired, since a Reidemeister II move could be be performed to separate the bottom two strands, producing a diagram with an arc between the two projectors, which is zero in $S_{2,1}$.

Of course, if we alter our definition of $\psi$, the quantum Russell relations would also have to be modified so that $\psi': R^q_{n,k} \to S^q_{n,k}$ remains well-defined. Figure \ref{under-under-obs} shows that there is no clear additional modification of $\psi'$ that would contribute to local quantum Russell relations that would be preserved under $\psi'$.

\begin{figure}[H]
\centering
\includegraphics[scale=1.1]{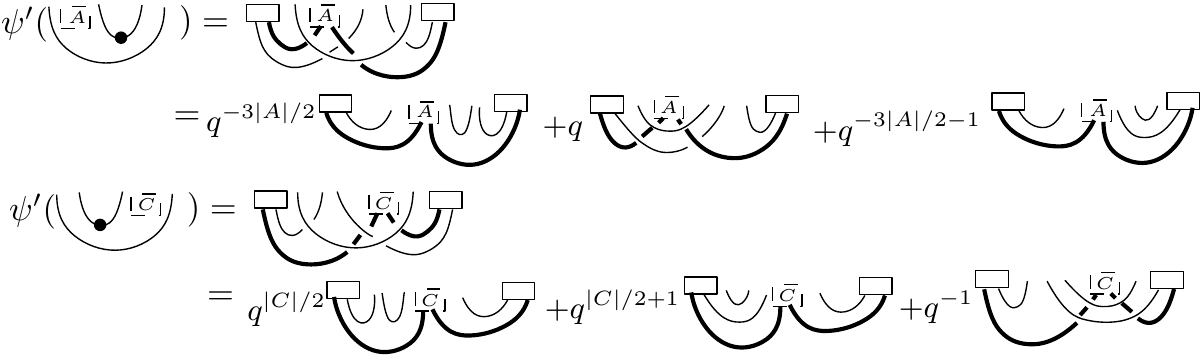}
\caption{There is no clear way to obtain quantum Russell relations that would be respected by a map $\psi'$ that expands a dotted arc nested inside an undotted arc by having both strands pass under.}
\label{under-under-obs}
\end{figure}

\section{Action of the symmetric group on $R_{n,k}^{q=1}$}
\label{sec:symm_gp_act}
We now specialize the above discussion of a braid group action on $\overline{R}^q_{n,k}$ and $R^q_{n,k}$ to the case $q=1$.

\subsection{The $q=1$ specialization of the braid group action on $\overline{R}^q_{n,k}$} 
\label{sec:symm_gp_act-1}

\begin{definition}
Denote by $\mathbb{Z}_{1}$ the one-dimensional $\mathbb{Z}[q, q^{-1}]$-module where $q$ acts as the identity. Let $\overline{R}_{n,k} \subseteq \widetilde{R}^{q=1}_{n,k}$ be the space $\overline{R}_{n,k}^q \otimes_{\mathbb{Z}[q,q^{-1}]} \mathbb{Z}_1$. That is, $\overline{R}_{n,k}$ is the free abelian group with basis $B_{n,k}$.
\end{definition}

\begin{remark}
The basis elements of $\overline{R}_{n,k}$ are given by the same diagrams as those of $R^{q=1}_{n,k}$ as well as the original Russell space $R_{n,k} = R^{q=-1}_{n,k}$.
\end{remark}

\begin{remark}
In \cite{russ-2}, Russell and Tymoczko call the elements of $B_{n,k}$ ``standard non crossing matchings.''
\end{remark}

The relationship between the spaces $\widetilde{R}^q_{n,k}, \overline{R}^q_{n,k}, \widetilde{R}^{q=1}_{n,k}$, and $\overline{R}_{n,k}$ can be visualized as follows:
\centerline{
\xymatrix{
\widetilde{R}^q_{n,k} \ar[r]^{q=1} & \widetilde{R}^{q=1}_{n,k} \\
\overline{R}^q_{n,k} \ar@{^{(}->}[u] \ar[r]^{q=1} & \overline{R}_{n,k} \ar@{^{(}->}[u]
}}

Now consider the braid group action of the previous section on both $\overline{R}^q_{n,k}$ and $S^q_{n,k}$ under the specialization $q=1$. When $q=1$, the Kauffman--Lins expansion of the positive crossing is the same as that of the negative crossing. Therefore we may neglect the orientation and depict both positive and negative crossings as a single crossing as in Figure \ref{KLq=1}. So when $q=1$, the actions of $\sigma_i$ and $\sigma_i^{-1}$ are the same, and the braid group action described in the previous section descends to an action of the symmetric group.

\begin{figure}[H]
\centering
\includegraphics[scale=1]{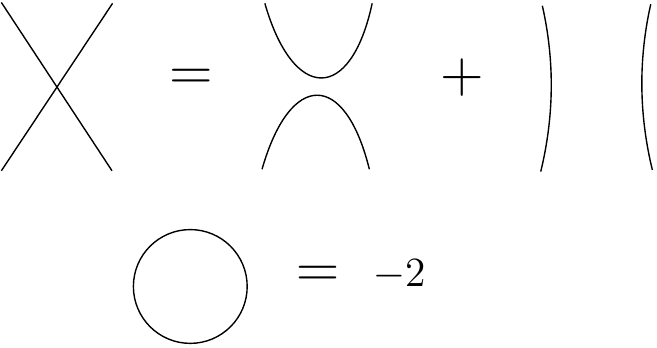}
\caption{The Kauffman--Lins diagrammatics when $q$ is specialized to 1.}
\label{KLq=1}
\end{figure}

In Figure \ref{Sn_action}, we show the specialization of the braid group action of Figure \ref{braid_group_act} to an $S_{2n}$ action on $\overline{R}_{n,k}$. We denote the generators of $S_{2n}$ by $s_i$, which as previously explained corresponds to both $\sigma_i$ and $\sigma_i^{-1}$ in the generic case.

\afterpage{
\begin{figure}
\centering
\includegraphics[scale=1]{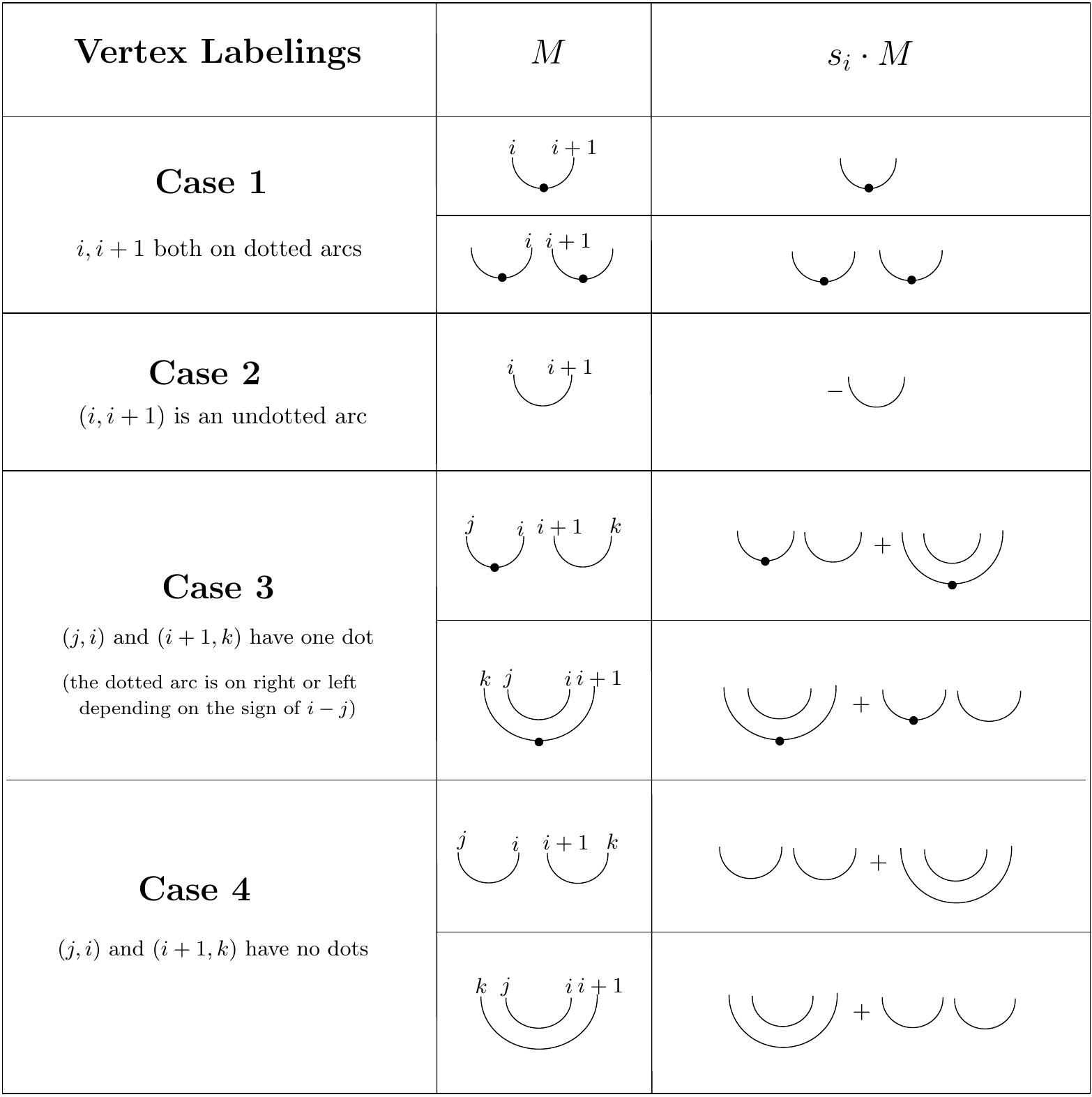}
\caption{The Russell-Tymoczko action of the symmetric group on $\overline{R}_{n,k}$}
\label{Sn_action}
\end{figure}
}

\begin{observation}
In \cite{russ}, Russell gives an explicit correspondence between the basis elements of $\overline{R}_{n,k}$ and generators of $H_{\ast}(X_n)$, where $X_n$ is the $(n,n)$-Springer variety. In \cite{russ-2}, Russell and Tymoczko define a symmetric group action on these basis elements and show that the action is the Springer representation. Our action shown in Figure \ref{Sn_action} is exactly the same as theirs. This is consistent with a result in a follow-up paper of Russell \cite{russ-3}, in which she describes a skein-theoretic formulation of the symmetric group action on $\overline{R}_{n,k}$ which parallels ours on the space $S_{n,k}$.
\end{observation}

\begin{corollary}
The $q=1$ specialization of the quantum braid group action on $\overline{R}^q_{n,k}$ yields the symmetric group representation corresponding to the Specht module of the partition $(n+k, n-k)$. Taking the union over all $k$ gives the well-known action of $S_{2n}$ on $H_{\ast}(X_n)$.
\end{corollary}

\begin{remark}
As previously mentioned, the traditional Russell skein module can be recovered from our quantum Russell space by specializing $q$ to $-1$. However, to get diagrammatics for the symmetric group from the Kauffman--Lins space, we must specialize $q$ to $1$. Russell and Tymoczko avoid this inconsistency by defining their action on the subspace of $\widetilde{R}_{n,k} = \widetilde{R}^{q=-1}_{n,k}$ spanned by standard crossingless matchings, which happens to be isomorphic to the quotient of $\widetilde{R}_{n,k}$ by the $q=-1$ Type I and Type II relations. Instead, we view the space spanned by standard crossingless matchings inside $\widetilde{R}_{n,k}^{q=1}$, which we extend to all of $\widetilde{R}_{n,k}^{q=1}$ in the following section in such a way that respects the $q=1$ Russell relations.
\end{remark}

\subsection{A local extension of the symmetric group action to $\widetilde{R}_{n,k}$ and $R_{n,k}$}
\label{sec:symm_gp_act-2}

Recall that in the case of generic $q$, it was not possible to extend our braid group action from $\overline{R}^q_{n,k}$ to $\widetilde{R}^q_{n,k}$ in such a way that the action remained completely local. The braid group action on the full quantum Russell space $R^q_{n,k}$ was shown in Figure \ref{braid_group_act_gen}.

Upon closer examination of Figure \ref{braid_group_act_gen}, we note that in the three cases in which nonlocal terms appear, the nonlocal terms all contain a factor of $(1-q^{-2})$. This means that when $q = \pm 1$, these terms disappear and the braid group action is in fact completely local. Note that when $q=\pm 1$, the types of diagrams $M$ that appear in both Figure \ref{braid_group_act} (where we assumed dots appeared on outer arcs only) and Figure \ref{braid_group_act_gen} have the same action of $\sigma_i$.

Therefore in particular when $q=1$, we may extend the Russell-Tymoczko symmetric group action on $\overline{R}_{n,k}$ to all of $\widetilde{R}_{n,k}$ in such a way that descends to an action on the $q=1$ Russell space $R^{q=1}_{n,k}$. That is, if we let $\pi: \widetilde{S}^{q=1}_{n,k} \to S^{q=1}_{n,k}$ be the quotient map, then we are able to extend the $S_{2n}$ action from $\overline{R}_{n,k}$ to $\widetilde{R}^{q=1}_{n,k}$ in such a way that it commutes with the map $\pi \circ \widetilde{\psi}$ and descends to $R^{q=1}_{n,k}$. See Figure \ref{new_symm_act} for a clarifying diagram.

\begin{figure}[H]
\centering
\includegraphics[scale=1]{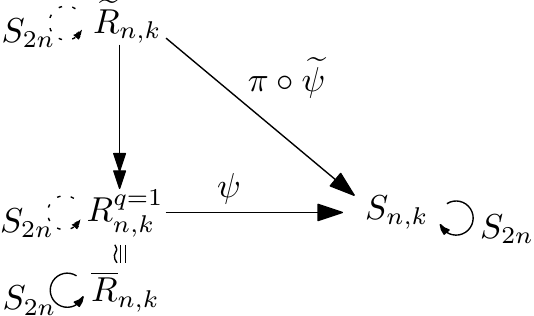}
\caption{The local $S_{2n}$ action on $\overline{R}_{n,k}$ can be extended to $\widetilde{R}_{n,k}$ and $R_{n,k}^{q=1}$ such that the diagram commutes.}
\label{new_symm_act}
\end{figure}

\begin{proposition}
Figure \ref{Sn_action_ext} defines a local action of the generators of $S_{2n}$ on diagrams of $\widetilde{R}_{n,k}$ that have dots on inner arcs. Together with the local definitions of Figure \ref{Sn_action}, this gives a well-defined $S_{2n}$ action on $\widetilde{R}_{n,k}$ that descends to $R_{n,k}^{q=1}$ such that Figure \ref{new_symm_act} commutes.
\end{proposition}

\begin{figure}[H]
\centering
\includegraphics[scale=1]{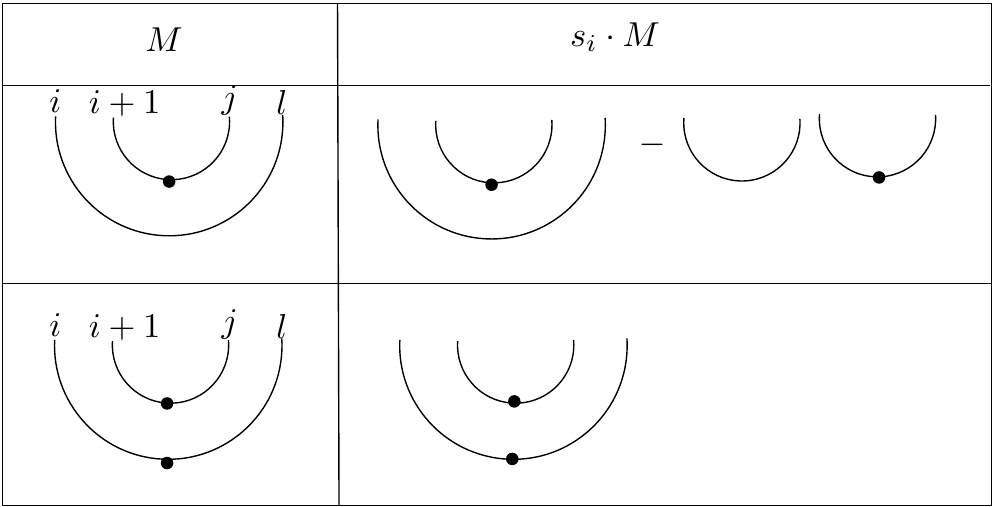}
\caption{An extension of the Russell-Tymoczko $S_{2n}$ action to $\widetilde{R}^{q=1}_{n,k}$.}
\label{Sn_action_ext}
\end{figure}

\chapter{An equivariant deformation of the Bar-Natan--Russell skein module}
\label{ch:equiv}

\section{The Russell skein module and $HH_0(H^n)$}
\label{sec:HH0}
Let $B^n$ be the set of crossingless matchings on $2n$ points. That is, elements of $B^n$ are pairings of integers from $1$ to $2n$ such that there is no quadruple $i < j < k < l$ with $(i, k)$ and $(j, l)$ paired. In \cite{crossmatch}, Khovanov defines a relation ``$\to$" on elements of $B^n$ as follows:

\begin{definition}
For $a, b \in B^n$ we write $a \to b$ if there is a quadruple $i < j < k < l$ such that $(i, j)$ and $(k, l)$ are pairs in $a$, $(i, l)$ and $(j, k)$ are pairs in $b$, and otherwise $a$ and $b$ are identical.
\end{definition}

Recall that for crossingless matchings $a$ and $b$ in $B^n$, the component ${}_b(H^n)_a$  of $H^n$ is obtained by applying the functor $\mathcal{F}$ to the collection of circles of $W(b)a$ to obtain the underlying algebra $\mathcal{A}^{\otimes k}$, where $W(b)$ is the vertical reflection of $b$ and $k$ is the number of circles in $W(b)a$. We index the tensor factors from $1$ to $k$ by numbering the circles in $W(b)a$ in the order in which the leftmost point of a circle is encountered as we traverse the center line from left to right. See Figure \ref{numConv} for an example.

\begin{figure}[ht]
   \centering
   \includegraphics[width=4in]{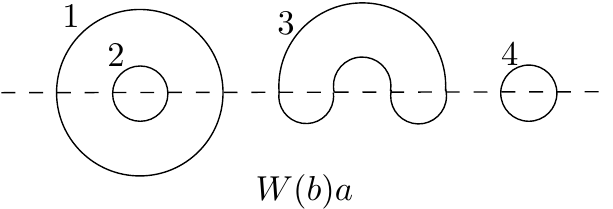}
   \caption{Numbering convention for the circles of $W(b)a$ and the corresponding tensor factors of ${}_aH^n_b$}
      \label{numConv}
\end{figure}

If $a \to b$, let $i$ and $j$ be the indices of the tensor factors of elements of ${}_a(H^n)_a$ and ${}_b(H^n)_b$ corresponding to the two circles that differ in $W(a)a$ and $W(b)b$. Define ${}_aX_b \subset {}_a(H^n)_a \oplus {}_b(H^n)_b$ to be the subspace generated by elements of one of the following forms:
\begin{itemize}
\item  ${}_ax^i_a + {}_ax^j_a - {}_bx^i_b - {}_bx^j_b$, where the ${}_*x_*^i$'s have an $x$ in tensor factor $i$ and a 1 in tensor factor $j$, the ${}_*x_*^j$'s have an $x$ in tensor factor $j$ and a 1 in tensor factor $i$, the ${}_ax^*_a$'s are in ${}_a(H^n)_a$, the ${}_bx^*_b$'s are in ${}_b(H^n)_b$, and they are all identical outside of tensor factors $i$ and $j$.
\item ${}_ax^{i,j}_a - {}_bx^{i,j}_b$, where each term has an $x$ in tensor factors $i$ and $j$ and are otherwise identical.
\end{itemize}

\begin{lemma}
\label{hncommutator}
The commutator module $[H^n, H^n]$ is generated by all elements of ${}_a(H^n)_b$ for $a \neq b$ and ${}_aX_b$ for all $a \to b$.
\end{lemma}

\begin{proof}
First, it is clear that any element of a summand ${}_a(H^n)_b$ for $a \neq b$ is in the commutator: if $y = x^{\epsilon_1} \otimes \cdots \otimes x^{\epsilon_k}$ is a generator of ${}_a(H^n)_b$ ($\epsilon_i = 0$ or $1$), and $1_b = 1 \otimes \cdots \otimes 1 \in {}_bH^n_b$, then $y = y - 0 = y \cdot 1_b - 1_b \cdot y$. 

All other elements of the commutator must belong to the component $\displaystyle{\bigoplus_{a \in B^n} {}_a(H^n)_a}$. Such elements may only come from $[{}_a(H^n)_b, {}_b(H^n)_a]$ where $a \neq b$, since $[{}_a(H^n)_a,{}_a(H^n)_a]=0.$ First suppose that $a$ and $b$ are such that $a \to b$ and let $i$ and $j$ be the indices of the tensor factors of ${}_a(H^n)_a$ and ${}_b(H^n)_b$ corresponding to the two arcs that appear in the relation $\to$. Without loss of generality take $i <  j$, so that the $i$th component of $W(b)a$ is the one that is asymmetrical. Let $p$ be an arbitrary generator of ${}_a(H^n)_b$ that has $x$'s in positions $P \subset \{1, \ldots, n-1 \}$ and $1$'s elsewhere, and similarly $q$ a generator of ${}_b(H^n)_a$ with $x$'s in positions $Q \subset \{1, \ldots, n-1\}$ and $1$'s elsewhere. Then
\[
pq - qp = \left\{
\begin{array}{ll}
0 & \mbox{if } P \cap Q \neq \emptyset \\
r^i_a + r^j_a - r^i_b -r^j_b & \mbox{if } P \cap Q = \emptyset \mbox{ and } i \notin P \cup Q\\
r^{i,j}_a - r^{i,j}_b & \mbox{if } P \cap Q = \emptyset \mbox{ and } i \in P \cup Q
\end{array}
\right.
\]
where $r^i_a$ is in ${}_a(H^n)_a$ and has $x$'s in tensor factors $i$ and those corresponding to $P \cup Q$ (some of whose indices may have shifted by one between ${}_a(H^n)_b$ and ${}_a(H^n)_a$) and 1's elsewhere, $r^j_a$ is in ${}_a(H^n)_a$ and has $x$'s in tensor factors $j$ and those corresponding to $P \cup Q$ and 1's elsewhere, and $r^i_b$ and $r^j_b$ are similar except in ${}_b(H^n)_b$. The $r_a^{i,j}$ and $r_b^{i,j}$ are the same but with $x$'s in tensor factors $i$ and $j$ as well as those corresponding to $P \cup Q$. Therefore any generator of $[{}_a(H^n)_b, {}_b(H^n)_a]$ belongs to ${}_aX_b$.

\begin{comment}
Define ${}_a1_b = 1 \otimes \cdots \otimes 1 \in {}_a(H^n)_b$ and choose ${}_a x^i_a \in {}_a(H^n)_a$ to be the element with $x$ in the $i$th tensor factor and 1's elsewhere, and similarly for ${}_ax^j_a$, ${}_bx^i_b$, and ${}_bx^j_b$. Note that $[{}_a(H^n)_b, {}_b(H^n)_a]$ contains the element
\[ {}_a1_b \cdot {}_b1_a - {}_b1_a \cdot {}_a1_b = {}_ax^i_a + {}_ax^j_a - {}_bx^i_b - {}_bx^j_b.\]

Any other generating element of $[{}_a(H^n)_b, {}_b(H^n)_a]$ may be obtained from this one by putting one or more $x$'s in the same tensor factors of each term, so $[{}_a(H^n)_b, {}_b(H^n)_a]$ is generated by ${}_aX_b$.
\end{comment}

Finally, suppose that $a \nrightarrow b$. Then there exists a sequence $a = a_0, a_1, \ldots, a_k = b$ such that for each $0 \leq i \leq k-1$, either $a_i \to a_{i+1}$ or $a_{i+1} \to a_i$. Let $i_m, j_m$ be the indices of the two arcs of $a_m$ and $a_{m+1}$ that change in the relation $a_m \to a_{m+1}$ (where arcs are numbered from left to right). Let $x^{i}_{a_m}$ be the element of ${}_{a_m}H^n_{a_m}$ with $x$ in the $i$th tensor factor and $1$s elsewhere, and similarly for $x^{j}_{a_m}$.

Observe that each product ${}_{a_0}1_{a_k} \cdot {}_{a_k}1_{a_0}$ is formed by successively merging and splitting the circles corresponding to the arcs $i_m, j_m$, so that
\[ {}_{a_0}1_{a_k} \cdot {}_{a_k}1_{a_0} = {}_{a_0}1_{a_1} \cdot ((x^{i_1}_{a_1} +x^{j_1}_{a_1}) \cdots (x^{i_{k-1}}_{a_1} + x^{j_{k-1}}_{a_1})){}_{a_1}1_{a_0}.\]

Then
\begin{eqnarray*}
 [{}_a1_b, {}_b1_a] &=& {}_a1_b {}_b1_a -{}_b1_a {}_a1_b \\
  &=& {}_{a_0}1_{a_1} \cdot ((x^{i_1}_{a_1} +x^{j_1}_{a_1}) \cdots (x^{i_{k-1}}_{a_1} + x^{j_{k-1}}_{a_1})){}_{a_1}1_{a_0} \\
  & & \qquad - {}_{a_k}1_{a_{k-1}} \cdot ((x^{i_0}_{a_{k-1}}+x^{j_0}_{a_{k-1}}) \cdots (x^{i_{k-2}}_{a_{k-1}} + x^{j_{k-2}}_{a_{k-1}})) {}_{a_{k-1}}1_{a_k}.
\end{eqnarray*} 

Observe that for any $m$,
\begin{multline*}
(x^{i_0}_{a_m} + x^{j_0}_{a_m}) \cdots \widehat{(x^{i_{m-1}}_{a_m} + x^{j_{m-1}}_{a_m})} \cdots (x^{i_{k-1}}_{a_m} + x^{j_{k-1}}_{a_m}) {}_{a_m}1_{a_{m-1}} \cdot {}_{a_{m-1}}1_{a_m} = \\
{}_{a_m}1_{a_{m+1}} \cdot (x^{i_0}_{a_{m+1}} + x^{j_0}_{a_{m+1}}) \cdots \widehat{(x^{i_{m}}_{a_{m+1}} + x^{j_{m}}_{a_{m+1}})} \cdots (x^{i_{k-1}}_{a_{m+1}} + x^{j_{k-1}}_{a_{m+1}}) {}_{a_{m+1}}1_{a_m}.
\end{multline*}

Therefore
\begin{multline*}
[{}_a1_b,{}_b1_a] = [{}_{a_0}1_{a_1},(x^{i_1}_{a_1}+ x_{a_1}^{j_1}) \cdots (x_{a_1}^{i_{k-1}} + x_{a_1}^{j_{k-1}}){}_{a_1}1_{a_0}] + \cdots \\
+ [{}_{a_{m-1}}1_{a_m}, (x_{a_m}^{i_0} + x_{a_m}^{j_0}) \cdots \widehat{(x_{a_m}^{i_{m-1}} + x_{a_m}^{j_{m-1}})} \cdots (x_{a_m}^{i_{k-1}} + x_{a_m}^{j_{k-1}}){}_{a_m}1_{a_{m-1}}] + \cdots \\
+ [{}_{a_{k-1}}1_{a_k}, (x_{a_k}^{i_0} + x_{a_k}^{j_0}) \cdots (x_{a_k}^{i_{k-2}} + x_{a_k}^{j_{k-2}}){}_{a_k}1_{a_{k-1}}].
\end{multline*}

\begin{comment}
Let ${}_{a_{i}}x^{i_1, \ldots, i_k}_{a_{i+1}}$ be the element of ${}_{a_i}(H^n)_{a_{i+1}}$ with $x$'s in the tensor factors $i_1, \ldots, i_k$ and $1$'s elsewhere. Then it can be seen by induction on $k$ that

\[{}_a1_b \cdot {}_b1_a - {}_b1_a \cdot {}_a1_b = \sum_{i_1, \ldots, i_k} ([{}_{a_0}x^{i_1, \ldots, i_k}_{a_1}, {}_{a_1}1_{a_0}] + \cdots + [{}_{a_{k-1}}x^{i_1, \ldots, i_k}_{a_k}, {}_{a_{k}}1_{a_{k-1}}]) \]

where the sum is over all choices of $k$ tensor factors from those corresponding to the arcs involved in the sequence $a_0, \ldots, a_k$. 
\end{comment}

Thus  $[{}_a(H^n)_b, {}_b(H^n)_a] \subseteq \sum_{0 \leq i \leq k-1} [{}_{a_i}H^n_{a_{i+1}}, {}_{a_{i+1}}(H^n)_{a_i}]$ and is generated by a subset of \\ $ \cup_{0\leq i \leq k-1} {}_{a_i}X_{a_{i+1}}$.
$\blacksquare$
\end{proof}

\begin{proposition}
\label{prop:HH0}
$HH_0(H^n)$ is isomorphic to the Russell skein module $R_n$ as a $\mathbb{Z}$-module.
\end{proposition}

\begin{proof}
Recall that the skein module $R_n$ is generated by crossingless matchings of $2n$ points, where arcs are allowed to carry up to one dot, modulo the local Type I and Type II relations of Figure \ref{Russell_relns}.

We get such relations for any $a,b \in B^n$ such that $a \to b$ or $b \to a$, where the relation $\to$ is defined as in \cite{crossmatch}. We will construct an explicit isomorphism $\varphi$ from this diagrammatically presented skein module to $HH_0(H^n)$.

Recall that $HH_0(H^n)$ is isomorphic to $H^n / [H^n, H^n]$. For a diagram $m$ in the Russell skein module given by a dotted crossingless matching with underlying arc structure $a$, construct the element $\varphi(m)$ as follows. First, vertically reflect the arcs of $a$ and glue the endpoints of $m$ with the reflected arcs to get $n$ closed circles with dots, which we call $d_m$. To get an element of ${}_a(H^n)_a \subset H^n$, take the generator of ${}_a(H^n)_a \subset H^n \cong \mathcal{A}^{\otimes n}$ that has an $X$ in each tensor factor corresponding a dotted circle in $d_m$ and $1$ otherwise. Define $\varphi(m)$ to be the class of this element in the quotient $H^n / [H^n, H^n]$ and extend $\varphi$ to the full Russell skein module by linearity.

It is clear from previous lemma that $\varphi$ is well-defined, as Type I and Type II Russell relations get mapped to elements of ${}_aX_b$, which belong to the commutator $[H^n, H^n]$. Further, the previous lemma guarantees that there is a one-to-one correspondence between the relations of the Russell skein module and the elements of $\cup_{a \to b} {}_aX_b$. Since all elements of ${}_a(H^n)_b$ for $a \neq b$ are in the commutator, $\varphi$ is surjective and thus an isomorphism.
$\blacksquare$
\end{proof}

\begin{remark}
Russell indirectly proves the previous proposition in \cite{russ}. She shows that the Bar-Natan skein module is isomorphic to the homology of the topological space $\widetilde{S}$ of \cite{crossmatch}. Khovanov showed that the cohomology of $\widetilde{S}$ is isomorphic to $HH^0(H^n) \cong Z(H^n)$. Therefore homology of $\widetilde{S}$ is isomorphic to $HH_0(H^n)$. Our argument gives a more direct isomorphism.
\end{remark}

\section{The equivariant deformation of the Russell skein module}
In Section \ref{sec:HH0}, we saw that the Russell skein module $R_n$ is isomorphic to $HH_0(H^n)$, the 0th Hochschild homology of the Khovanov arc ring $H^n$. In this section we construct a deformation of $R_n$ by considering the 0th Hochschild homology of a deformation of the ring $H^n$.

\subsection{Equivariant deformation of $H^n$}
\label{equi-Hn}
Recall that $H^n$, introduced by Khovanov \cite{khov}, is constructed by taking a direct sum of algebras  (with a grading shift) obtained by applying a functor $\mathcal{F}$ to a diagram obtained by gluing together two crossingless matchings of $2n$ points. The functor $\mathcal{F}$ associates to each circle in the diagram a tensor factor of the ring $\mathcal{A}$, where $\mathcal{A}$ is isomorphic to $\mathbb{Z}[X]/X^2$ with the usual multiplication,  and comultiplication and trace map as follows:
\begin{gather*}
\Delta: \mathcal{A} \to \mathcal{A} \otimes \mathcal{A} \\
\Delta(1) = 1 \otimes X + X \otimes 1, \Delta(X) = X \otimes X \\
\epsilon: \mathcal{A} \to \mathbb{Z} \\
\epsilon(1) = 0, \epsilon(X) = 1.
\end{gather*}

\begin{definition}
A Frobenius system $F = (R, A, \epsilon, \Delta)$ consists of commutative rings $R$ and $A$, with $R$ viewed as a subring of $A$ sharing the identity element, together with an $A$-bimodule map $\Delta: A \to A \otimes_R A$ and an $R$-module map $\epsilon: A \to R$ such that $\Delta$ is coassociative and cocommutative, and $(\epsilon \otimes \mbox{Id})\Delta = \mbox{Id}.$
\end{definition}

We can think of $\mathcal{A}$ as belonging to a Frobenius system $(\mathcal{R}, \mathcal{A}, \epsilon, \Delta)$ with $\epsilon$ and $\Delta$ described as above and $\mathcal{R} = \mathbb{Z} \subset \mathbb{Z}[X]/X^2$. 

\begin{definition}
A Frobenius system $F$ has rank two if there exists $X \in A$ such that $A \cong R1 \oplus RX$.
\end{definition}

Observe that $\mathcal{A}$ has rank two over $\mathcal{R}$. In \cite{frob-ext}, Khovanov describes how any rank two Frobenius system produces a cohomology theory of links. Here we investigate an alternate such system with a cohomological interpretation.

First observe that the ring $\mathcal{A}$ has a cohomological interpretation: it is isomorphic to the cohomology ring of the two-sphere $S^2$, and $\epsilon$ is the integration along the fundamental cycle on $S^2$. As described by Khovanov \cite{frob-ext}, there exist other rank two Frobenius systems with interpretations via equivariant homology.

If $G$ is a topological group acting continuously on $S^2$, we can define
\[ R_G := H^{\ast}_G(p, \mathbb{Z}) = H^{\ast}(BG, \mathbb{Z}),\]
the $G$-equivariant cohomology ring of a point $p$ (where $BG$ is the classifying space of $G$), and
\[ A_G := H^{\ast}_G(S^2, \mathbb{Z}) = H^{\ast}(S^2 \times_G EG, \mathbb{Z}),\]
the equivariant cohomology ring of $S^2$. In several cases, $(R_G, A_G)$ forms a part of a rank two Frobenius system, with $\epsilon$ induced by the integration along the fibers of the $S^2$-fibration
\[S^2 \times_G EG \to BG.\]

In particular, take $G$ to be the group $U(2)$ with the usual action on $S^2$. Then we obtain a rank two Frobenius system with
\begin{eqnarray*}
R_{U(2)} &\cong& H^{\ast}(BU(2), \mathbb{Z}) \cong H^{\ast}(\mbox{Gr}(2, \infty), \mathbb{Z}) \cong \mathbb{Z}[h,t] \\
A_{U(2)} &\cong& H^{\ast}(S^2 \times_{U(2)} EU(2), \mathbb{Z}) \cong H^{\ast}(BU(1) \times BU(1), \mathbb{Z}) \cong R_{U(2)}[X], X^2 = hX + t,
\end{eqnarray*}
where $\mbox{deg}(h) = 2 $ and $\mbox{deg}(t) = 4$, and
\begin{gather*}
\Delta(1) = 1 \otimes X + X \otimes 1 - h 1 \otimes 1, \Delta(X) = X \otimes X + t 1 \otimes 1 \\
\epsilon(1) = 0, \epsilon(X) = 1.
\end{gather*}

\begin{definition}
Let $\mathcal{A}_{h,t} = A_{U(2)}$ and define $H^n_{h,t}$ in the exact same way as $H^n$ but with $\mathcal{A}_{h,t}$ in place of $\mathcal{A}$. We call $H^n_{h,t}$ the \emph{equivariant deformation of $H^n$}.
\end{definition}

As in the usual case of $H^n$, $H^n_{h,t}$ has a grading. The grading works exactly as before, but now in the ring $\mathcal{A}_{h,t}$, we take $1$ to be in degree $0$, $X$ in degree $2$, $h$ in degree $2$, and $t$ in degree $4$, so that the relation $X^2 = hX + t$ is grading-preserving. 

\subsection{Equivariant deformation of $R_n$ and associated graphical calculus}
\label{sec:equi-Russell}
Recall that the original Russell module $R_n$ is isomorphic to $HH_0(H^n)$. We use this fact as motivation for constructing a new deformation of the Russell skein module. In this section, for mild technical simplification, instead of working over $\mathbb{Z}$, we work over a field $\Bbbk$. With this modification, the ring $\Bbbk[h,t]$ is local as a graded ring, and its graded Jacobson radical is the ideal $(h,t)$.

\begin{definition}
Define the \emph{equivariant deformation of $R_n$}
\[ R^{h,t}_n := HH_0(H^n_{h,t}).\]
\end{definition}

We now present a graphical description of $R^{h,t}_n$ analogous to the Russell calculus of dotted crossingless matchings. Let $\widetilde{R}^{h,t}_n$ be the free $\Bbbk[h,t]$-module generated by diagrams, where a diagram is a crossingless matching of $2n$ points (drawn as cups below a horizontal line containing the $2n$ fixed points) with each arc decorated by a finite number of dots.

We consider $\widetilde{R}^{h,t}_n$ as a graded module with grading inherited from $\Bbbk[h,t]$, again with $1$ in degree $0$, $h$ in degree $2$, and $t$ in degree $4$, and define the degree of a diagram to be twice its number of dots.

\begin{definition}
Let $\overline{R}^{h,t}_n$ be the quotient of $\widetilde{R}^{h,t}_n$ by the local relations shown in Figure \ref{equiv-relns}.
\end{definition}

\begin{figure}[htbp]
   \centering
   \includegraphics[scale=.8]{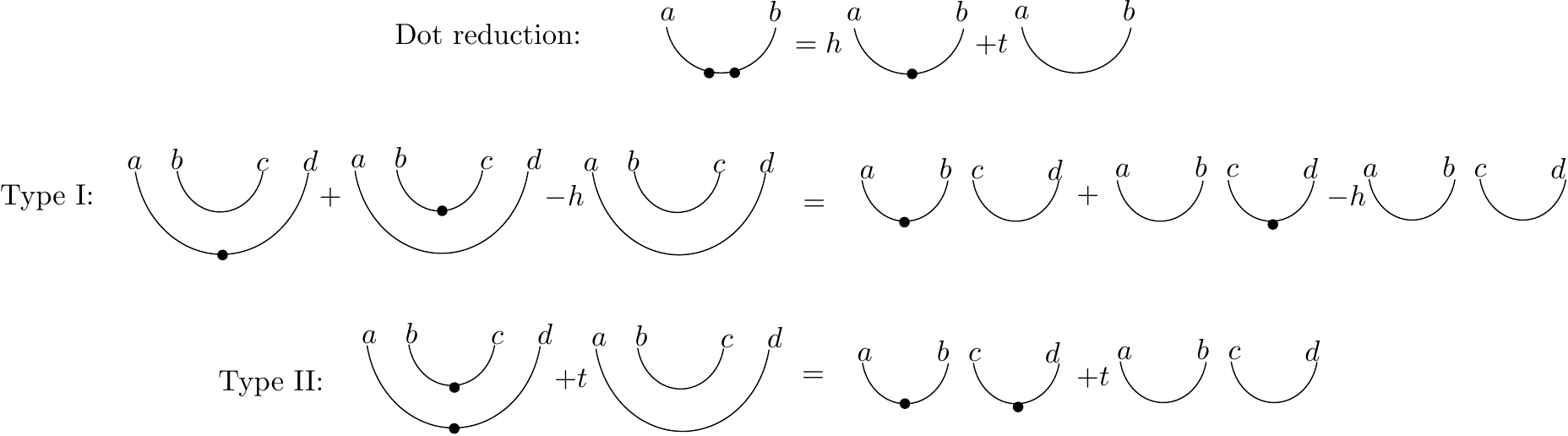}
   \caption{Local relations in the graphical deformation of the Russell skein module.}
      \label{equiv-relns}
\end{figure}

As in the original and quantum Russell skein modules, these relations are local in the sense that all diagrams in each relation are identical apart from the arcs shown. Additionally, since arcs may now carry more than one dot, we assume that for each diagram in a relation if we were to slide all dots not involved in the relation away from the region shown and cut out that region, then the resulting dotted arc segments would be identical in all diagrams.

Note that these relations are grading-preserving, so that $\overline{R}_n^{h,t}$ inherits a graded structure from $\widetilde{R}_n^{h,t}$.

\begin{proposition}
$R^{h,t}_n$ and $\overline{R}^{h,t}_n$ are isomorphic as $\Bbbk[h,t]$-modules.
\end{proposition}

\begin{proof}
First recall that $R^{h,t}_n = HH_0(H^n_{h,t}) \cong H^n_{h,t}/[H^n_{h,t}, H^n_{h,t}]$. We define a map $\varphi_{h,t}: \widetilde{R}^{h,t}_n \to R^{h,t}_n$. For a dotted crossingless matching $m \in \widetilde{R}^{h,t}_n$ with underlying arc structure $a$, construct $\varphi_{h,t}(m)$ as follows: First vertically reflect the arcs of $a$ and glue the endpoints of $m$ with those of the reflected arcs to form a diagram $d_m$. Then take the element of ${}_a(H^n_{h,t})_{a} \cong \mathcal{A}_{h,t}^{\otimes n}$ that has a factor of $X$ for each dot on the corresponding circle in $d_m$ and $1$ if no dots are present and define $\varphi_{h,t}(m)$ to be the class of this element in the quotient $R^{h,t}_n$. Extend this map to all of $R^{h,t}_n$ by linearity. Note that $\varphi_{h,t}$ is graded of degree 0: it turns dots into factors of $X$, both of which have degree 2 in their respective rings.

It is possible to show that $\varphi_{h,t}$ descends to a well-defined isomorphism from $\overline{R}_n^{h,t}$ to $HH_0(H^n_{h,t}$ using an argument completely analogous to that of Lemma \ref{hncommutator} to show that the equivariant Type I and Type II relations exactly correspond with generators of $[H^n_{h,t}, H^n_{h,t}]$ that belong to $\oplus_{a \in B^n} {}_a(H^n_{h,t})_a \subset H^n_{h,t}$ and that all elements of $\oplus_{a \neq b \in B^n} {}_b(H^n_{h,t})_a \subset H^n_{h,t}$ are in the commutator. However, to avoid such a calculation, we alternatively draw our conclusion from the $h=t=0$ case of Proposition \ref{prop:HH0} and the following graded version of Nakayama's Lemma:

\begin{lemma}[Graded Nakayma's Lemma]
Let $R$ be a ring graded by the nonnegative integers, and $I$ a homogeneous ideal whose elements are positively graded. Let $M$ be a graded $R$-module with $M_i = 0$ for $i \ll 0$. If homogeneous elements $m_1, \ldots m_n \in M$ have images in $M/IM$ that generate it as an $R$-module, then $m_1, \ldots, m_n$ generate $M$ as an $R$-module.
\end{lemma}

Note that $\varphi_{h,t}$ is surjective: it is clear that all classes of elements of $\oplus_{a \in B^n} {}_a(H^n_{h,t})_a$ are in the image of $\varphi_{h,t}$, and elements of $\oplus_{a \neq b \in B^n} {}_a(H^n_{h,t})_b$ are contained in the commutator and are thus trivial in $H^n_{h,t}/[H^n_{h,t}, H^n_{h,t}]$. 

We take $\varphi: \widetilde{R}_n \to H^n/[H^n,H^n]$ as in Proposition \ref{prop:HH0} (where there it was considered a map on the quotient space $R_n$) and modify it slightly: we enlarge the space $\widetilde{R}_n$ to a space $\widetilde{R}_n'$, where arcs in a diagram are allowed to carry more than one dot, and extend the map $\varphi$ to a map $\varphi': \widetilde{R}'_n \to H^n/[H^n, H^n]$ that sends all such new diagrams to $0$. That is, the quotient space $\widetilde{R}_n' / \ker(\varphi')$ is the same as that of $\widetilde{R}_n/\ker{\varphi}$ but with additional relations corresponding to the $h=t=0$ analogue of the dot reduction relation shown in Figure \ref{dot_red_old}.

\begin{figure}[htbp]
   \centering
   \includegraphics[scale=1]{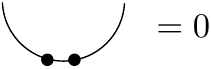}
   \caption{The $h=t=0$ dot reduction relation.}
      \label{dot_red_old}
\end{figure}

Therefore we have the following commutative diagram of exact sequences of graded $\Bbbk[h,t]$-modules. Note that we have been considering the modules in the top row as $\Bbbk$-modules, but equivalently we may consider them as $\Bbbk[h,t]$-modules where $h$ and $t$ act trivially.

\centerline{
\xymatrix{
0 \ar[r] & \Bbbk \langle \mbox{dot reduction, Type I, Type II} \rangle \ar@{^{(}->}[r] & \widetilde{R}_n '\ar@{->>}[r]^-{\varphi'} &H^n/[H^n, H^n] \ar[r] & 0\\
0 \ar[r] & \ker{\varphi_{h,t}} \ar@{^{(}->}[r] \ar[u]_{h,t=0} &\widetilde{R}_n^{h,t} \ar@{->>}[r]^-{\varphi_{h,t}} \ar[u]_{h,t=0} & H^n_{h,t}/[H^n_{h,t}, H^n_{h,t}] \ar[r] \ar[u]_{h,t=0} & 0
} 
}

Then we may apply the graded Nakayama's Lemma to get a generating set for $\ker \varphi_{h,t}$: take $R = \Bbbk[h,t]$, $I= (h,t)$, and $M = \ker \varphi_{h,t}$. $\ker \varphi_{h,t}$ inherits its grading from that of $\widetilde{R}_n^{h,t}$ and the fact that $\varphi_{h,t}$ is graded. We know that the images of the equivariant dot reduction, Type I, and Type II relations form a generating set of $M/IM$, so we can conclude that these equivariant relations generate $M$. Therefore we have
\[ \overline{R}_n^{h,t} \cong \widetilde{R}_n^{h,t}/\ker \varphi_{h,t} \cong H^n_{h,t}/[H^n_{h,t}, H^n_{h,t}]. \]
$\blacksquare$
\end{proof}

From now on, by abuse of notation, we will use $R_n^{h,t}$ to denote both $HH_0(H^n_{h,t})$ and the graphical space $\overline{R}_n^{h,t}$.

\section{Equivariant deformation has the expected rank}
\label{sec:equi-rank}
In this section, we return to working over $\mathbb{Z}$ instead of $\Bbbk$. We will show that the space $R^{h,t}_n$ has the correct size as a deformation of $R_n$, in the sense that its rank and basis are analogous to that of $R_n$.

In Section \ref{sec:psi} we saw that in the case of the quantum deformation $R_n^q$, $R_n^q$ could be decomposed into a direct sum of spaces $\oplus_{0 \leq k \leq n} R^q_{n,k}$ where each $R_{n,k}^q$ has a basis $B_{n,k}$ of diagrams with $k$ dots on outer arcs only and rank $\frac{2k+1}{n+k+1} \left( \begin{array}{c} 2n \\ n+k \end{array} \right)$ over $\mathbb{Z}[q, q^{-1}]$. In particular, specializing $q=-1$ recovers the original Russell space $R_n$, which therefore has rank $\sum_{0 \leq k \leq n} \frac{2k+1}{n+k+1} \left( \begin{array}{c} 2n \\ n+k \end{array} \right)$ over $\mathbb{Z}$.

With $t$ and $h$ present, we no longer have such a direct sum decomposition, since relations involve diagrams with different numbers of dots. Therefore instead of considering the rank of each summand separately, we consider the rank of the total space $R_n^{h,t}$.

As before, we will use $B_{n,k}$ to denote the set of diagrams in $R_n^{h,t}$ that have $k$ total dots on outer arcs only with each arc carrying at most one dot. Define $B_n$ to be $\cup_{0 \leq k \leq n} B_{n,k}$.

\begin{theorem}
$R_{n}^{h,t}$ has rank $\sum_{0 \leq k \leq n} \frac{2k+1}{n+k+1} \left( \begin{array}{c} 2n \\ n+k \end{array} \right)$ over $\mathbb{Z}[h,t]$ and $B_n$ forms a basis.
\end{theorem}

\begin{proof}
The first part of the theorem follows immediately from the second and Proposition \ref{prop:rank}, which counted the number of dotted crossingless matchings with dots on outer arcs only. We need to see that the collection of diagrams with dots on outer arcs only and each arc carrying at most one dot forms a basis.

\begin{comment}
It is clear that we can uniquely express any diagram with more than one dot on a single arc as a linear combination of diagrams with a single dot on each arc by applying the first of the relations of Figure \ref{equiv-relns}. That is, $R_{n}^{h,t} \cong (R_n^{h,t})'$, where $(R_n^{h,t})'$ is the $\mathbb{Z}[h,t]$-module generated by diagrams with at most one dot on each arc and whose relations are only the Type I and Type II relations of Figure \ref{equiv-relns}. Therefore we need only shown that $(R_n^{h,t})'$ has a basis of diagrams with dots on outer arcs only.
\end{comment}

It is easy to see that such diagrams form a spanning set. If a diagram has more than one dot on an arc, then the first of the relations in Figure \ref{equiv-relns} may be applied to reduce the number of dots on that arc. If a diagram has a dot on an inner arc, then a rearranged version of either a Type I or Type II relation can be applied to express the diagram as a linear combination of ones with strictly fewer arcs containing dotted arcs. These relations may be applied repeatedly until a linear combination of diagrams with zero or one dot on outer arcs only is attained. Therefore we need only show that such diagrams are linearly independent.

For a diagram $d$ in $R^{h,t}_n$, define the \emph{containment} $n(d)$ to be the sum over all dots in $d$ of the number of arcs in $d$ containing the arc on which the dot lies, that is, if the dot lies on an arc $a$ with endpoints $(a_1, a_2)$, the number of arcs $b$ with endpoints $(b_1, b_2)$ such that $b_1 < a_1$ and $b_2 > a_2$. We use $R_n^{h,t}(\leq N)$ to denote the space generated by diagrams $d$ in $\widetilde{R}_n^{h,t}$ with $n(d) \leq N$ and relations being the subset of those in Figure \ref{equiv-relns} that only involve diagrams with containment $\leq N$. 

We will complete the proof of the theorem by induction on $n(d)$. That is, we will assume that $R_n^{h,t}(\leq N)$ has $\cup_{0 \leq k \leq n} B_{n,k}$ as a basis and show that $R_n^{h,t}(\leq N + 1)$ has the same basis.

For the base case $R_n^{h,t}(\leq 0)$, note that if a diagram $d$ has $n(d)= 0$, that is it has no dots nested in other arcs, so it is already has dots on outer arcs only. There are no Type I or Type II relations in this space, so elements of $B_n$ are clearly linearly independent in $R_n^{h,t}(\leq 0)$.

Now suppose that $B_n$ forms a basis of $R_n^{h,t}(\leq N)$ and consider the same set inside $R_n^{h,t}(\leq N+1)$. We must show that this set remains linearly independent. To see that such diagrams are linearly independent, since we already know they form a spanning set, it is equivalent to show that expression of any diagram as a linear combination of diagrams in $B_n$ is unique. That is, when there is a choice of which local relation to apply at any stage in the process, then any possible choice must ultimately result in the same linear combination. To see that this is the case, we apply a strategy similar to that used by Kuperberg \cite{kup-G2} to show that non-elliptic webs form a basis of $G_2$-spiders.

Let $d$ be a diagram in $R_n^{h,t}(\leq N+1)$ of containment $N+1$. Suppose that $d$ can be expressed as an element of $R_n^{h,t}(\leq N)$ in more than one way by applying a local relation of Figure \ref{equiv-relns}. First note that if the arcs involved in each of the potential relations do not overlap, then it is clear that it will not matter which of the relations is applied first. Therefore we must only consider cases in which the potential relations that could be applied share an arc. The cases that must be checked are shown in Figure \ref{equiv-cases}. For each diagram, the circled regions show each of the possible containment-reducing relations that could be applied.

\begin{figure}[htbp]
   \centering
   \includegraphics[scale=1.1]{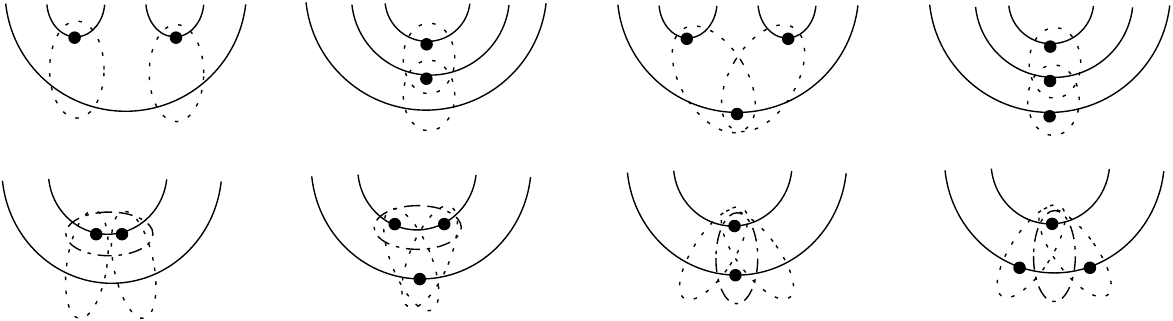}
   \caption{Cases to check that order of applying local relations does not matter.}
      \label{equiv-cases}
\end{figure}

We show the claim for the first, third, and fifth cases: the others follow by analogous argument. In the first case, there are two Type I relations that could be applied (see Figure \ref{case1}).

\begin{figure}[htbp]
   \centering
   \includegraphics[scale=.8]{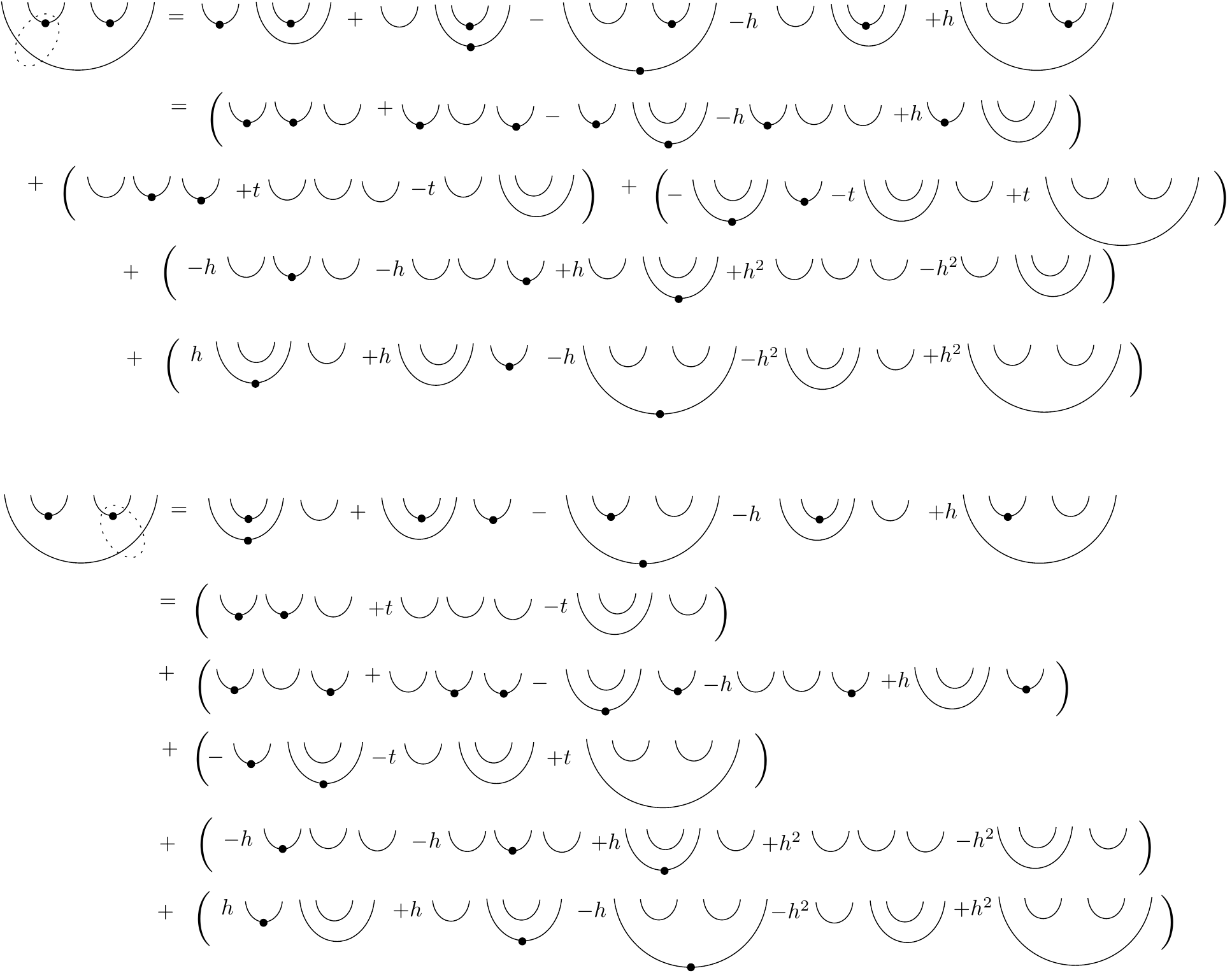}
   \caption{Order of applying two Type I relations does not matter.}
      \label{case1}
\end{figure}

The first equality in each of the two calculations of Figure \ref{case1} is obtained by applying the indicated Type I relation. As usual these are local pictures, and we assume that the regions not pictured are identical in each diagram. After applying the relation, we see that the remaining diagrams are elements of $R_n^{t,h}[\leq N]$. Therefore by the inductive hypothesis, each of those diagrams has a unique expression as a linear combination of elements of $B_n$, so the choice of which containment-reducing relation to apply next does not matter. By applying the remaining obvious Type I or Type II relations on the arcs pictured and invoking the inductive hypothesis on $R_n^{t,h}[\leq N-1]$, we see that we arrive at the same answer in both computations.

\begin{figure}[htbp]
   \centering
   \includegraphics[scale=.8]{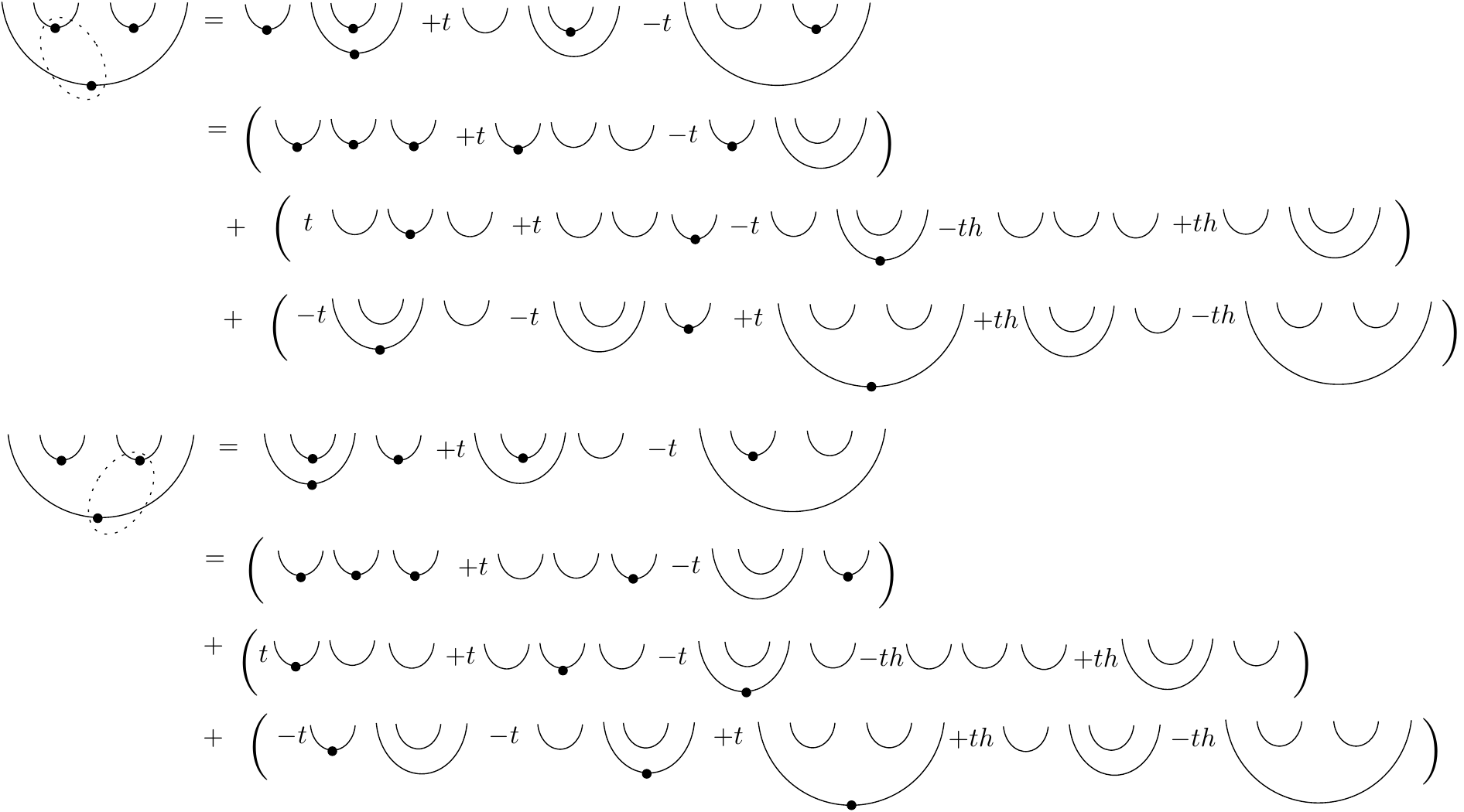}
   \caption{Order of applying two Type II relations does not matter.}
      \label{case3}
\end{figure}

Figure \ref{case3} proves the claim in the third case, following the same argument as above. In Figure \ref{case5}, we examine the fifth case. Here, there are three relations that could be applied first involving overlapping arcs: a dot reduction relation, or two possible Type I relations. $\blacksquare$

\begin{figure}[htbp]
   \centering
   \includegraphics[scale=.8]{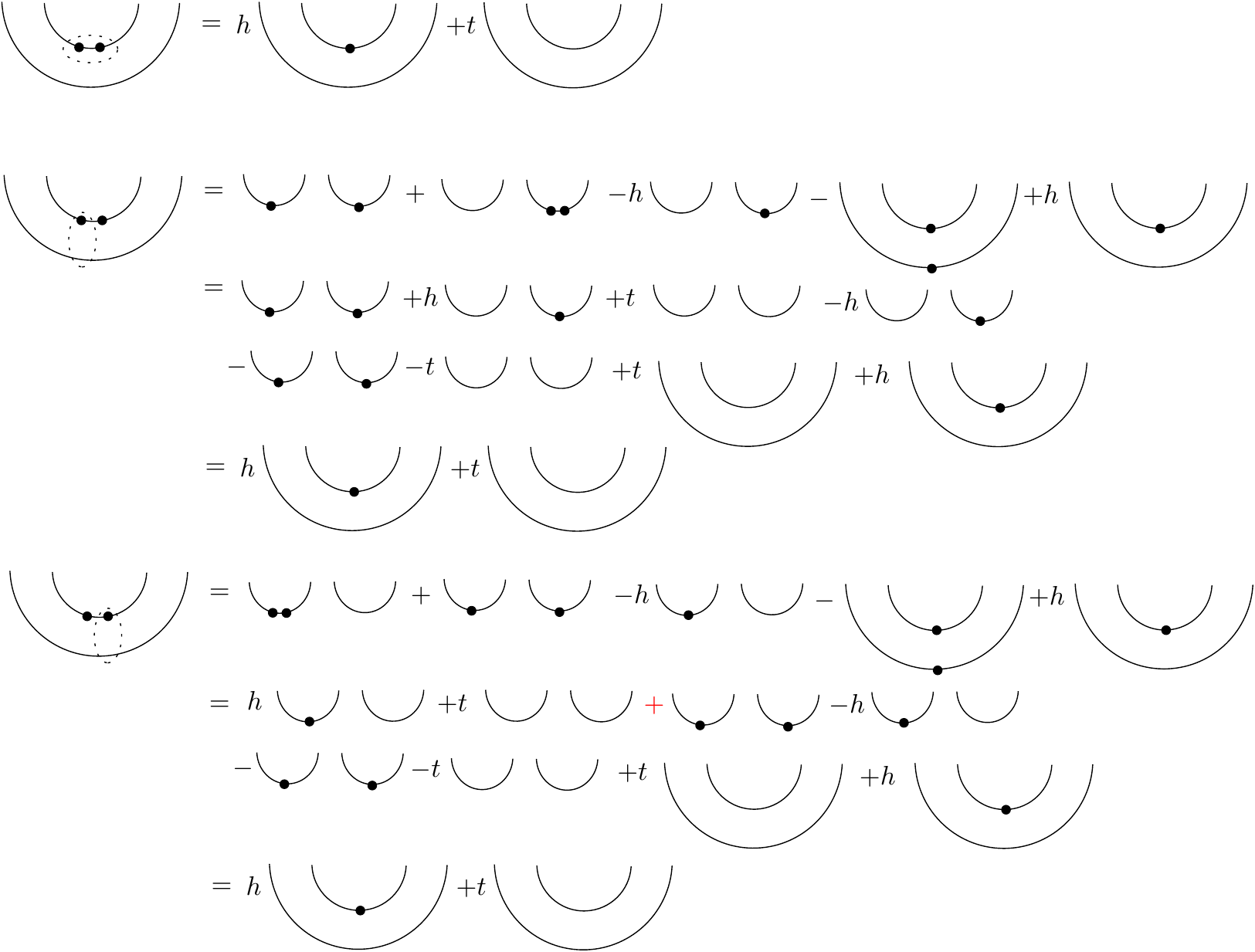}
   \caption{Order of applying dot reduction relation or two Type I relations does not matter.}
      \label{case5}
\end{figure} 
\end{proof}

\chapter{Tangle cobordisms and $(H^{\lowercase{m}},H^{\lowercase{n}})$-bimodule homomorphisms}
\label{ch:tangle-cob}

\section{Definitions}
Let $\mathbb{TL}$ be the Temperley-Lieb 2-category. The objects of this category are nonnegative even integers. The one-morphisms from $2n$ to $2m$ will be denoted $\widehat{B}^m_n$ and are given by flat tangles with $2m$ top and $2n$ bottom endpoints. For $a, b \in \widehat{B}^m_n$, the two-morphisms from $a$ to $b$ are isotopy classes of admissible cobordisms from $a$ to $b$, that is, cobordisms such that every generic horizontal cross-section is a $(2m,2n)$-tangle.

Recall that Khovanov defines a 2-functor $\mathcal{F}$ from the 2-category of tangle cobordisms to the 2-category of bimodules over the rings $H^n$, which in particular restricts to a 2-functor on $\mathbb{TL}$. We briefly review the definition of $\mathcal{F}$ on one- and two-morphisms of $\mathbb{TL}$.

When we restrict to planar $(2m, 2n)$-tangles, instead of a chain complex of $(H^m, H^n)$-bimodules, we just get a single $(H^m, H^n)$-bimodule, and a cobordism between planar tangles simply becomes a homomorphism of bimodules. For a planar $(2m, 2n)$-tangle $T$, the functor $\mathcal{F}$ associates an $(H^m, H^n)$-bimodule $\mathcal{F}(T)$ by
\[ \mathcal{F}(T) = \bigoplus_{a, b} {}_b\mathcal{F}(T)_a,\]
where the sum is over all $a \in B^n$ and $b \in B^m$, and
\[ {}_b \mathcal{F}(T)_a := \mathcal{F}(W(b)Ta)\{n\}. \]
The left action $H^m \times \mathcal{F}(T) \to \mathcal{F}(T)$ comes from maps
\[ {}_c(H^m)_b \times {}_b\mathcal{F}(T)_a \to {}_c\mathcal{F}(T)_a\]
induced by the cobordism from $W(c)bW(b)Ta$ to $W(c)Ta$ which is the composition of the identity cobordisms on $W(c)$ and on $Ta$ with the cobordism $S(b) : bW(b) \to \mbox{Vert}_{2m}$ defined in Section \ref{sec:Hn-def}. The right action of $H^n$ is similarly defined.

Now let $S$ be a cobordism between two planar $(2m, 2n)$-tangles $T_1, T_2$. Then, ignoring gradings, $S$ induces a map 
\[ \mathcal{F}(S): \mathcal{F}(T_1) \to \mathcal{F}(T_2)\]
as follows. Recall that $\mathcal{F}(T_1) = \oplus_{a,b} \mathcal{F}(W(b)T_1a)$ and $\mathcal{F}(T_2) = \oplus_{a,b}\mathcal{F}(W(b)T_2a)$, where both sums are over $a \in B^n$ and $b \in B^m$. Then for any such $a$ and $b$, $S$ induces a cobordism from $W(b)T_1a$ to $W(b)T_2a$ given by composing $S$ with the identity cobordisms on $W(b)$ and on $a$, which in turn induces a map of abelian groups $\mathcal{F}(W(b)T_1a) \to \mathcal{F}(W(b)T_2a)$. $\mathcal{F}(S)$ is defined to be the sum over all $a,b$ of these maps.

Fix a commutative ring $R$, usually $\mathbb{Z}$ or $\mathbb{Q}$. Let $\mbox{Hom}_{BN}(T_1, T_2)$ be the free $R$-module generated by tangle cobordisms with dots from $T_1$ to $T_2$ modulo the local Bar-Natan relations of Figure \ref{BN-relns}. Addition of tangle cobordisms is defined formally.

Let  $\mbox{Hom}_{(m,n)}(T_1, T_2)$ be the free $R$-module generated by $(H^m, H^n)$-bimodule homomorphisms from $\mathcal{F}(T_1)$ to $\mathcal{F}(T_2)$, with addition given by usual pointwise addition of bimodule homomorphisms.  Then Khovanov's functor gives a homomorphism of $R$-modules
\[ \phi_{T_1,T_2}: \mbox{Hom}_{BN}(T_1, T_2) \to \mbox{Hom}_{(m,n)}(T_1, T_2). \]

We now introduce two new rings, $\mbox{Hom}_{BN}(m,n)$ and $\mbox{Hom}(m,n)$.
\begin{definition}
For any nonnegative integers $m$ and $n$, define 
\[ \mbox{Hom}_{BN}(m,n) := \bigoplus_{T_1,T_2} \mbox{Hom}_{BN}(T_1,T_2)\]
and
\[ \mbox{Hom}(m,n) := \bigoplus_{T_1,T_2} \mbox{Hom}_{(m,n)}(T_1,T_2) \]
where both sums are over all $T_1$ and $T_2$ in $B_n^m$. 

Multiplication in $\mbox{Hom}_{BN}(m,n)$ is induced by the maps
\[ m_{BN}: \mbox{Hom}_{BN}(T_1,T_2) \times \mbox{Hom}_{BN}(T_2,T_3) \to \mbox{Hom}_{BN}(T_1,T_3) \]
given by vertical stacking of cobordisms. Multiplication in $\mbox{Hom}(m,n)$ is induced by the maps
\[ M: \mbox{Hom}(T_1,T_2) \times \mbox{Hom}(T_2,T_3) \to \mbox{Hom}(T_1,T_3) \]
given by composition of homomorphisms.
\end{definition}

Then the set of maps $\phi_{T_1,T_2}$ induce a ring homomorphism
\[ \phi_{m,n}: \mbox{Hom}_{BN}(m,n) \to \mbox{Hom}(m,n). \]

Note that when $T_1 = T_2 = T$, the same multiplication maps above give $\mbox{End}_{BN}(T)$ and $\mbox{End}_{(m,n)}(T)$ natural ring structures as well.

\section{Surjectivity of $\phi_{m,n}$}
\label{main}

In this section, we will show that the map $\phi_{T_1, T_2}$ is surjective for any $T_1, T_2$ and give an explicit description of its kernel. We first analyze the simpler case in which $m=n$ and $T_1, T_2$ are both the identity tangle $\mbox{Vert}_{2n}$ consisting of $2n$ vertical lines. In this case, $\mbox{End}_{(n,n)}(\mbox{Vert}_{2n})$ is just the center of the ring $H^n$. Recall that Khovanov showed in \cite{crossmatch} that
\[ Z(H^n) \cong \mathbb{Z}[x_1, \ldots, x_{2n}]/(x_1^2, \ldots, x_{2n}^2, e_1(x_1, \ldots, x_{2n}), \ldots, e_{2n}(x_1, \ldots, x_{2n}))
\]
where $e_i(x_1, \ldots, x_{2n})$ is the $i$th elementary symmetric function in $x_1, \ldots, x_{2n}$. The statement was proven through a roundabout argument which expresses $Z(H^n)$ as the cohomology of a certain topological space $\widetilde{S}$, which has the same cohomology as the $(n,n)$-Springer variety $\mathcal{B}_{n,n}$. The argument does not give an explicit map in either direction, so the $x_i$'s cannot be interpreted as any elements of $H^n$.

We claim that $e_{n+1}(x_1, \ldots, x_{2n}), \ldots, e_{2n}(x_1, \ldots, x_{2n})$ are actually redundant under the relations $x_i^2 = 0$, as each can be expressed in terms of the first $n$ elementary symmetric polynomials.
\begin{lemma}
\label{lem:polys}
For $1 \leq k \leq n$ and $x_i^2  = 0$,
\begin{multline*}
e_{n+k}(x_1, \ldots, x_{2n}) = e_1(x_{n+k}, \ldots, x_{2n})e_{n+k-1}(x_1, \ldots, x_{2n}) - e_{2}(x_{n+k}, \ldots, x_{2n})e_{n+k-2}(x_1, \ldots, x_{2n}) \\  + -  \cdots + (-1)^{n-k} e_{n-k+1}(x_{n+k}, \ldots, x_{2n})e_{2k-1}(x_1, \ldots, x_{2n}).
\end{multline*}
\end{lemma}

\begin{proof}
Note that the elementary symmetric function $e_m(x_1, \ldots, x_{2n})$ can be expressed as 
\[ e_m(x_1, \ldots, x_{2n}) = \sum_{\{i_1, \ldots, i_m\} \subset \{1, \ldots, 2n\}} x_{i_1}\cdots x_{i_k}.\]
Also observe that because of the relations $x_i^2 = 0$, a product of elementary symmetric functions
\[ e_{m}(x_{n+k}, \ldots, x_{2n})e_{n+k-m}(x_{1}, \ldots, x_{2n}) \]
can be expressed as
\[ \sum_{\{i_1, \ldots, i_{n+k}\} \subset \{1, \ldots, 2n\} } c^m_{i_1, \ldots i_{n+k}} x_{i_1} \cdots x_{i_{n+k}},\]
where
\[ c^m_{i_1, \ldots, i_{n+k}} = \left( \begin{array}{c} \left| \{i_1, \ldots, i_{n+k} \} \cap \{ 1, \ldots, 2n\} \right| \\ m \end{array} \right). \]
Note that each $\left| \{i_1, \ldots, i_{n+k} \} \cap \{ 1, \ldots, 2n\} \right|$ is guaranteed to be at least one, because the first set of that intersection has size $n+k$, the second has size $n-k+1$, and both are subsets of the same $2n$-element set. Therefore $c^m_{i_1, \ldots, i_{n+k}}$ is nonzero for all $m$ such that $1 \leq m \leq \left| \{i_1, \ldots, i_{n+k} \} \cap \{ 1, \ldots, 2n\} \right|$.

Let $n_{i_1, \ldots, i_{n+k}} =  \left| \{i_1, \ldots, i_{n+k} \} \cap \{1, \ldots, 2n\} \right|$. Then on the right-hand side of the equation in the statement of the Lemma, the coefficient of $x_{i_1} \cdots x_{i_{n+k}}$ is
\[ \sum_{m =1}^{n_{i_1, \ldots, i_{n+k}}} (-1)^{m+1} \left( \begin{array}{c} n_{i_1, \ldots, i_{n+k}} \\ m \end{array} \right) = 1,\]
which is the same as its coefficient on the left-hand side. $\blacksquare$
\end{proof}

The previous lemma allows each $e_{n+k}$ to be written in terms of $e_i$, where $i < n+k$. In particular, $e_{n+1}$ can be written in terms of $e_1, \ldots, e_n$, and by induction so can $e_{n+k}$ for any $k$.

For example, for $n=2$,
\begin{eqnarray*}
e_3(x_1, x_2, x_3, x_{4}) &=& (x_3+x_4)(x_1x_2+x_1x_3+x_1x_4+x_2x_3+x_2x_4+x_3x_4) -x_3x_4(x_1+x_2+x_3+x_4) \\
&=& (x_3+x_4)e_2(x_1,x_2,x_3,x_4)-x_3x_4e_1(x_1,x_2,x_3,x_4) \\
e_4(x_1, x_2, x_3, x_4) &=& x_4(x_1x_2x_3+x_1x_2x_4+x_1x_3x_4+x_2x_3x_4) \\
&=& x_4e_3(x_1,x_2,x_3,x_4) \\
&=& x_4((x_3+x_4)e_2(x_1,x_2,x_3,x_4)-x_3x_4e_1(x_1,x_2,x_3,x_4)).
\end{eqnarray*}

Note that with the Bar-Natan relations present, any cobordism from $\mbox{Vert}_{2n}$ to itself can be represented by one consisting of $2n$ parallel, vertical sheets with 0 or 1 dots on each sheet. Denote the cobordism with one dot on the $i$th sheet and a coefficient of $(-1)^{i+1}$ by $c_i$. Then, as a module over $\mathbb{Z}$, $\mbox{End}_{BN}(\mbox{Vert}_{2n}) = \mbox{Hom}_{BN}(\mbox{Vert}_{2n}, \mbox{Vert}_{2n})$ is generated by $c_1, \ldots, c_{2n},$ with defining relations $c_i^2 = 0$ for all $i$.

\begin{proposition}
\label{surjection}
The map $\phi_{\mbox{Vert}_{2n}, \mbox{Vert}_{2n}}: \mbox{Hom}_{BN}(\mbox{Vert}_{2n}, \mbox{Vert}_{2n}) \to \mbox{Hom}_{(n,n)}(\mbox{Vert}_{2n}, \mbox{Vert}_{2n})$ is a surjection.
\end{proposition}

We prove Proposition \ref{surjection}  with the aid of the following key proposition:
\begin{proposition}
\label{kernelprop}
Over $\mathbb{Z}$, the kernel of the ring homomorphism $\phi_{\mbox{Vert}_{2n}, \mbox{Vert}_{2n}}$ is the two-sided ideal generated by the first $n$ elementary symmetric functions in $c_i, 1 \leq i \leq 2n$.
\end{proposition}
\begin{proof}
We proceed by induction on $n$.
First suppose $n=1$. There is only one crossingless matching in this case. When a cobordism from $\mbox{Vert}_2$ to itself is capped off by this arc on both ends, the two sheets are joined into a single sheet and therefore cobordisms with a single dot on either sheet produce the same bimodule homomorphism. Thus $c_1+c_2$ is in the kernel, and it is clear that this generates the whole kernel.

Now we proceed with the induction step. Let $a$ denote an arbitrary crossingless matching in $B^n$. We write  $a=((i_1, j_1), \ldots (i_n, j_n))$, where each pair represents the endpoints of an arc in $a$. Observe that when a cobordism from $\mbox{Vert}_{2n}$ to itself is closed up by $a$ on both sides, each sheet $i_k$ becomes merged with the sheet $j_k$ so that dots on either the $i_k$th or $j_k$th sheets induce the same map ${}_{a}(H^n)_a \to {}_{a}H^n_a$. Therefore if we just look at the image of $\phi_{\mbox{Vert}_{2n}, \mbox{Vert}_{2n}}$ projected onto the summand ${}_{a}(H^n)_a$ of $H^n$, the kernel will be given by the ideal $I_a = (c_{i_1}+c_{j_1}, \ldots, c_{i_n}+c_{j_n})$. More generally, the kernel of $\phi_{\mbox{Vert}_{2n}, \mbox{Vert}_{2n}}$ projected onto the summand ${}_{a}(H^n)_b$ is $I_a \cap I_b$. The kernel $I$ of $\phi_{\mbox{Vert}_{2n}, \mbox{Vert}_{2n}}$ is then just the intersection of the ideals $I_a$ over all $a$ in $B^n$:

\[ I = \bigcap_{((i_1, j_1), \ldots, (i_n, j_n)) \in B^n} (c_{i_1}+c_{j_1}, \ldots, c_{i_n}+c_{j_n}). \]

Our claim is that $I = (e_1(c_1, c_2, \ldots, c_{2n}), \ldots, e_n(c_1, c_2, \ldots, c_{2n}))$. Now for some $k$ in each crossingless matching, $|i_k - j_k| = 1$. Therefore, invoking the inductive hypothesis, we have
\begin{eqnarray*}
I &=& \bigcap_{((i_1, j_1), \ldots, (i_n, j_n)) \in B^n} (c_{i_k} + c_{j_k}, c_{i_1} + c_{j_1}, \ldots, \widehat{c_{i_k} + c_{j_k}}, \ldots, c_{i_n} + c_{j_n}) \\
&=& \bigcap_{i=1}^n \left( \bigcap_{((i_1, j_1) \cdots (i, i+1) \cdots (i_n, j_n))} (c_i + c_{i+1}, c_{i_1} + c_{j_1}, \ldots, \widehat{c_{i_k} + c_{j_k}}, \ldots, c_{i_n} + c_{j_n}) \right) \\
&=& \bigcap_{i=1}^n \left( (c_i + c_{i+1}) \cap \bigcap_{((i_1, j_1) \cdots (i, i+1) \cdots (i_n, j_n))} (c_{i_1} + c_{j_1}, \ldots, \widehat{c_{i_k} + c_{j_k}}, \ldots, c_{i_n} + c_{j_n}) \right) \\
&=& \bigcap_{i=1}^n (c_i+ c_{i+1}, e_1(c_1, \ldots, \widehat{c_i, c_{i+1}}, \ldots, c_{2n}), \ldots, e_{n-1}(c_1, \ldots, \widehat{c_i, c_{i+1}}, \ldots, c_{2n})),
\end{eqnarray*}
\noindent where $\widehat{x}$ means that the variable $x$ is omitted. We only need to intersect up to $n$ rather than $2n$ since any crossingless matching with an arc from $i$ to $i+1$, $i > n$, must also have an arc from $l$ to $l+1$ for some $l \leq n$. Define
\[ e_j(c:i) = e_j(c_1, \ldots, \widehat{c_i, c_{i+1}}, \ldots, c_{2n}).\]

Observe that
\[ e_j(c_1, \ldots, c_{2n}) = e_j(c:i) + (c_i+c_{i+1})e_{j-1}(c : i) + (c_ic_{i+1})e_{j-2} (c : i), \]
where $e_i(c : k) := 0$ if $i < 0$. Therefore by manipulating the generators, we see that:
\begin{eqnarray*}
I &=& \bigcap_{i=1}^n(c_i+c_{i+1}, e_1(c:i)+(c_i+c_{i+1}), \ldots, e_{n-1}(c:i)  +(c_i+c_{i+1})e_{n-2}(c:i) + c_ic_{i+1}e_{n-3}(c:i)) \\
&=& \bigcap_{i=1}^n(c_i+c_{i+1}, e_1(c_1, \ldots, c_{2n}), \ldots, e_{n-1}(c_1, \ldots, c_{2n})) \\
&=& ((c_1+c_2)(c_2+c_3)\cdots(c_n+c_{n+1}), e_1(c_1, \ldots, c_{2n}), \ldots, e_{n-1}(c_1, \ldots, c_{2n})).
\end{eqnarray*}

To see the last equality, first look at the intersection of ideals in the quotient ring  
\[ \mathbb{Z}[c_1, \ldots, c_{2n}]/(e_1(c_1, \ldots, c_{2n}), \ldots, e_{n-1}(c_1, \ldots, c_{2n})).\]
This ring is a unique factorization domain, so it is clear that $\cap_{k=1}^n(\overline{c_k}+\overline{c_{k-1}}) = ((\overline{c_1}+\overline{c_2})\cdots(\overline{c_{n}}+\overline{c_{n+1}}))$ since the $\overline{c_k}+\overline{c_{k+1}}$ are pairwise relatively prime. Then, when we quotient by $c_1^2, \ldots, c_{2n}^2$, we get the same statement, since all ideals are now principle. Finally, we can lift the statement back to $\mathbb{Z}[c_1, \ldots, c_{2n}]/(c_1^2, \ldots, c_{2n}^2)$ to get the desired equality.

Now observe the following formula, whose proof is analogous to that of Lemma \ref{lem:polys}:
\begin{multline*}
(c_1+c_2)(c_2+c_3)\cdots(c_n+c_{n+1}) = e_n(c_1, \ldots, c_{2n})-e_1(c_{n+2}, \ldots c_{2n})e_{n-1}(c_1, \ldots, c_{2n}) 
\\+ \cdots  + (-1)^{n-1} e_{n-1}(c_{n+2}, \ldots, c_{2n})e_1(c_1, \ldots, c_{2n}).
\end{multline*}

Therefore we now have that $I = (e_1(c_1, c_2, \ldots, c_{2n}), \ldots, e_n(c_1, c_2, \ldots, c_{2n}))$ as desired.
$\blacksquare$
\end{proof}

\begin{remark}
Note that we may interpret each generator of the kernel as a ``divided power" of the first elementary symmetric function in the following sense, where the last equality makes sense over $\mathbb{Q}$:
\[ e_k(c_1,c_2, \ldots, c_{2n}) = e_1(c_1, c_2, \ldots, c_{2n})^{(k)} = \frac{1}{k!}e_1(c_1, c_2, \ldots, c_{2n})^k. \]

Over $\mathbb{Q}$, we may apply a similar argument to show that the kernel of $\phi_{\mbox{Vert}_{2n}, \mbox{Vert}_{2n}}$ is generated by just the first elementary symmetric function, $c_1+c_2+\ldots+c_{2n}$.
\end{remark}

Now we can complete the proof of Proposition \ref{surjection}:
\begin{proof}
Now we have that
\begin{eqnarray*}
\mbox{Im}(\phi_{\mbox{Vert}_{2n}, \mbox{Vert}_{2n}}) &\cong& \mbox{Hom}_{BN}(\mbox{Vert}_{2n}, \mbox{Vert}_{2n})/\mbox{ker}(\phi_{\mbox{Vert}_{2n}, \mbox{Vert}_{2n}}) \\
&\cong& \mathbb{Z}[c_1, \ldots, c_{2n}]/((c_1^2, \ldots, c_{2n}^2, e_1(c_1, \ldots, c_{2n}), \ldots, e_n(c_1, \ldots, c_{2n})) \\
&\cong& \mbox{Hom}_{(n,n)}(\mbox{Vert}_{2n}, \mbox{Vert}_{2n}).
\end{eqnarray*}
So $\phi_{\mbox{Vert}_{2n}, \mbox{Vert}_{2n}}$ is surjective.  $\blacksquare$
\end{proof}

Now we may move on to the more general case of arbitrary planar $(2m,2n)$-tangles $T_1$ and $T_2$. To establish some notation, denote by $\cap_{i,n} $ the element of $B^{n-1}_n$ and by $\cup_{i, n-1}$ the element of $B^{n}_{n-1}$ pictured in Figure \ref{capcup} below.
\begin{figure}[h]
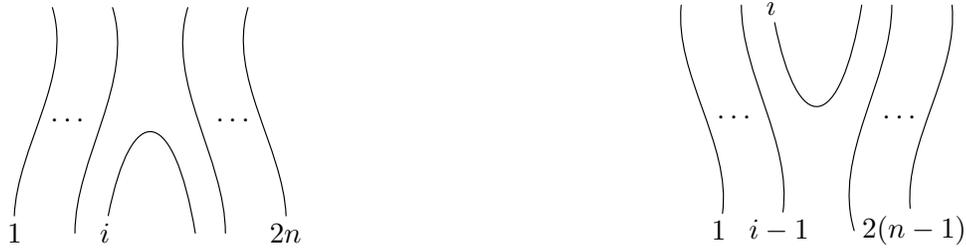

\begin{minipage}{0.5 \linewidth}
\[ 
\xy
(-13,0)*{}="a";
(-18,-30)*{}="b";
(-5,0)*{}="c";
(-10,-30)*{}="d";
"a"*{};"b"*+{1}  **\crv{(-10,-10) & (-19,-20)};
"c"*{}; "d"*{} **\crv{(-2,-10) & (-10,-20)};
(13,0)*{}="a'";
(18,-30)*{}="b'";
(5,0)*{}="c'";
(10,-30)*{}="d'";
"a'"*{};"b'"*+{2n}  **\crv{(10,-10) & (19,-20)};
"c'"*{}; "d'"*{} **\crv{(2,-10) & (10,-20)};
(-6,-30)*{}="e";
(6,-30)*{} = "f";
"e"*+{i}; "f"*{} **\crv{(-3,-12)&(3,-12)};
(-11,-15)*{\cdots};
(11,-15)*{\cdots};
\endxy \]
\end{minipage}
\hspace{.5cm}
\begin{minipage}{0.5 \linewidth}
\[ \xy
(-13,-30)*{}="a";
(-18,0)*{}="b";
(-5,-30)*{}="c";
(-10,0)*{}="d";
"a"*+{1};"b"*{}  **\crv{(-10,-20) & (-19,-10)};
"c"*+{i-1}; "d"*{} **\crv{(-2,-20) & (-10,-10)};
(13,-30)*{}="a'";
(18, 0)*{}="b'";
(5,-30)*{}="c'";
(10,0)*{}="d'";
"a'"*+{2(n-1)};"b'"*{}  **\crv{(10,-20) & (19,-10)};
"c'"*{}; "d'"*{} **\crv{(2,-20) & (10,-10)};
(-6,-0)*{}="e";
(6,0)*{} = "f";
"e"*+{i}; "f"*{} **\crv{(-3,-18)&(3,-18)};
(-11,-15)*{\cdots};
(11,-15)*{\cdots};
\endxy \]
\end{minipage}
\caption{$\cap_{i,n}$ in $B_n^{n-1}$ and $\cup_{i,n-1}$ in $B_{n-1}^n$}
\label{capcup}
\end{figure}

In  \cite{invtcob}, Khovanov proves the following proposition related to homomorphisms of $(m,n)$-bimodules, where $F_{\cup}$ is the functor of tensoring with $\mathcal{F}(\cup_{i, n-1})$ and $F_{\cap}$ is the functor of tensoring with $\mathcal{F}(\cap_{i,n})$.
\begin{proposition}[Khovanov]
\label{adjoint}
$F_{\cup}\{1\}$ is left adjoint to $F_{\cap}$, and $F_{\cap}\{-1\}$ is left adjoint to $F_{\cup}$.
\end{proposition}
The statement holds whether the tensor products are taken on the left or on the right. This adjointness comes from isotopies between compositions of cobordisms between $\mbox{Vert}_{2(n-1)}$ and $\cap_{i,n}\cup_{i,n-1}$ and between $\mbox{Vert}_{2n}$ and $\cup_{i,n-1} \cap_{i,n}$. More concretely, and ignoring gradings, the proposition gives the isomorphisms
\begin{itemize}
\item $\mbox{Hom}_{(m,n)} (\cup_{i, m-1} T_1, T_2) \cong \mbox{Hom}_{(m-1, n)} (T_1, \cap_{i, m} T_2), $
where $T_1$ is a planar $(2(m-1), 2n)$-tangle and $T_2$ is a planar $(2m, 2n)$-tangle
\item $\mbox{Hom}_{(m,n)}(T_1, T_2 \cap_{i,n}) \cong \mbox{Hom}_{(m,n-1)}(T_1 \cup_{i,n-1}, T_2)$, where $T_1$ is a planar $(2m,2n)$-tangle and $T_2$ is a planar $(2m, 2(n-1))$-tangle
\item $\mbox{Hom}_{(m,n)}(T_1, \cup_{i,m-1}T_2) \cong \mbox{Hom}_{(m-1,n)}(\cap_{i,m}T_1, T_2)$, where $T_1$ is a planar $(2m,2n)$-tangle and $T_2$ is a planar $(2(m-1), 2n)$-tangle, and
\item $ \mbox{Hom}_{(m,n)}(T_1 \cap_{i, n}, T_2) \cong \mbox{Hom}_{(m, n-1)}(T_1, T_2 \cup_{i, n-1} ),$
where $T_1$ is a planar $(2m, 2(n-1))$-tangle and $T_2$ is a planar $(2m, 2n)$-tangle.
\end{itemize}

We briefly explain the explicit maps underlying the first of these isomorphisms. The maps in the other isomorphisms are analogous. In the first isomorphism, define 
\[ \alpha: \mbox{Hom}_{(m,n)}(\cup_{i, m-1} T_1, T_2) \to \mbox{Hom}_{(m-1, n)} (T_1, \cap_{i,m} T_2) \] as follows. For a map $\varphi \in \mbox{Hom}_{(m,n)} (\cup_{i,m-1}T_1, T_2)$, $\alpha(\varphi)$ is defined to be 
\[ (\mbox{Id}_{\cap_{i,m}} \otimes \varphi) \circ (\eta \otimes \mbox{Id}_{T_1})\]
where $\eta$ is the map from $\mathcal{F}(\mbox{Vert}_{2(m-1)})$ to $\mathcal{F}(\cap_{i,m} \cup_{i,m-1})$ induced by the ``birth'' cobordism from the empty $(0,0)$-tangle to the closed circle. Conversely, we define
\[ \beta: \mbox{Hom}_{(m-1, n)} (T_1, \cap_{i, m} T_2) \to \mbox{Hom}_{(m,n)} (\cup_{i, m-1} T_1, T_2) \]
that takes $\psi \in \mbox{Hom}_{(m-1, n)} (T_1, \cap_{i, m} T_2)$ to
\[ (\nu \otimes \mbox{Id}_{T_2}) \circ (\mbox{Id}_{\cup_{i, m-1}} \otimes \psi) \]
where $\nu$ is the map from $\mathcal{F}(\cup_{i, m-1} \cap_{i, m})$ to $\mathcal{F}(\mbox{Vert}_{2m})$ induced by the saddle cobordism. Khovanov explains in \cite{invtcob} that $\alpha$ and $\beta$ are mutually inverse.

\begin{example}
Disregarding grading shifts, Proposition \ref{adjoint} gives 
\[\mbox{Hom}_{(1,1)}(\cup_{1,0}\cap_{1,1}, \mbox{Vert}_2) \cong \mbox{Hom}_{(1,0)}(\cup_{1,0}, \cup_{1,0}).\]
For bimodule maps $\varphi \in \mbox{Hom}_{(1,1)}(\cup_{1,0}\cap_{1,1}, \mbox{Vert}_2)$ and $\psi \in \mbox{Hom}_{(1,0)}(\cup_{1,0}, \cup_{1,0})$, the corresponding maps $\alpha(\varphi)$ and $\beta(\psi)$ under the isomorphism are pictured below in Figures \ref{adjoint1} and \ref{adjoint2}. Here, the blank space around $\varphi$ and $\psi$ represents a homomorphism of bimodules, while the new maps $\alpha(\varphi)$ and $\beta(\psi)$ are obtained by composing the olds maps with the pictured cobordisms in the manner illustrated.

\begin{figure}[htb]
\begin{minipage}[b]{0.5 \linewidth}
\begin{center}
\scalebox{.5}{\includegraphics[scale=1.2]{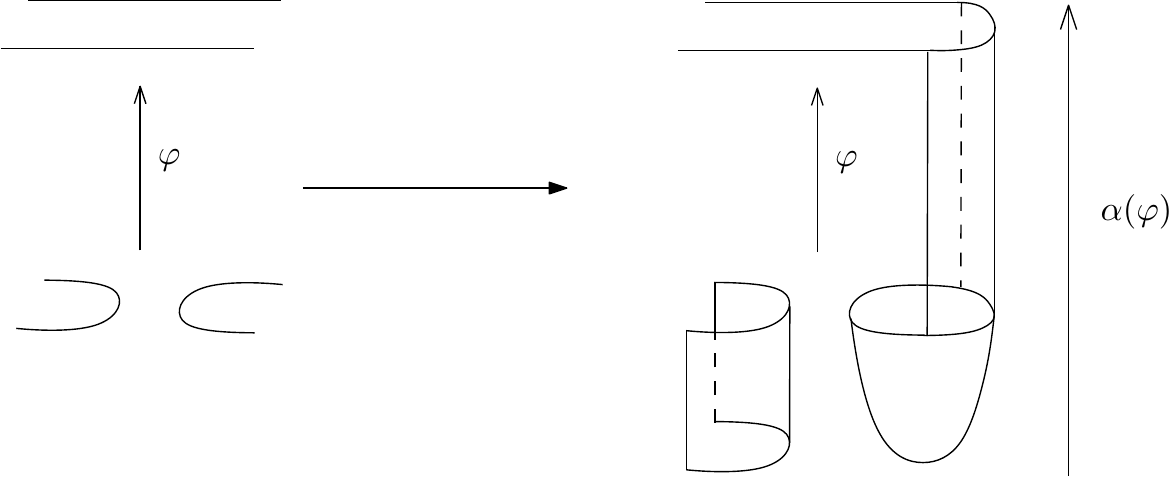}}
\caption{Image of $\varphi$ in $\mbox{Hom}_{(1,0)}(\cup_{1,0}, \cup_{1,0})$}
\label{adjoint1}
\end{center}
\end{minipage}
\hspace{.5 cm}
\begin{minipage}[b]{0.5 \linewidth}
\begin{center}
\scalebox{.6}{\includegraphics[scale=1.2]{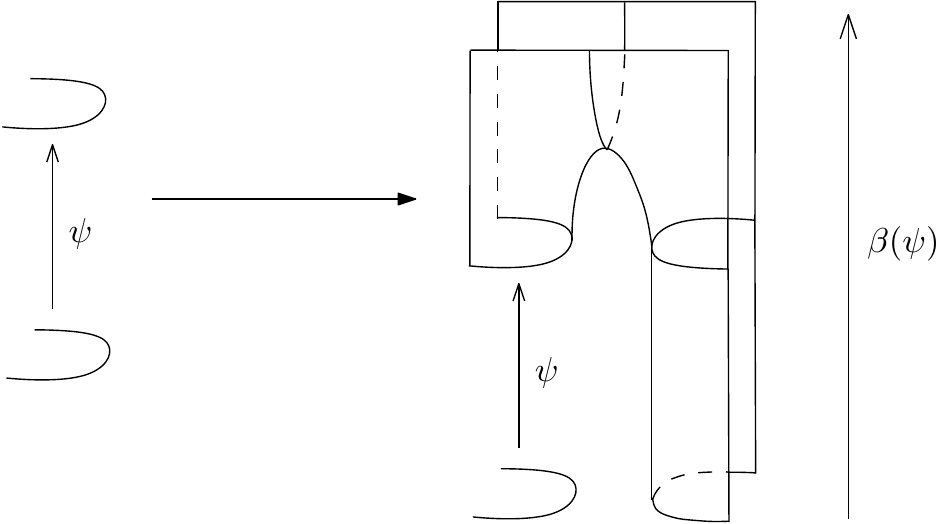}}
\caption{\mbox{Image of $\psi$ in $\mbox{Hom}_{(1,1)}(\cup_{1,0}\cap_{1,1}, \mbox{Vert}_2)$}}
\label{adjoint2}
\end{center}
\end{minipage}
\end{figure}
\end{example}

It is easy to see that we have a similar statement of adjointness for tangle cobordisms modulo Bar-Natan relations, given by bending a cobordism to move a cup or cap at one boundary to the opposite boundary, as shown in Figure \ref{cobadjoint} below.
\begin{figure}[htb]
\begin{center}
\scalebox{.8}{\includegraphics{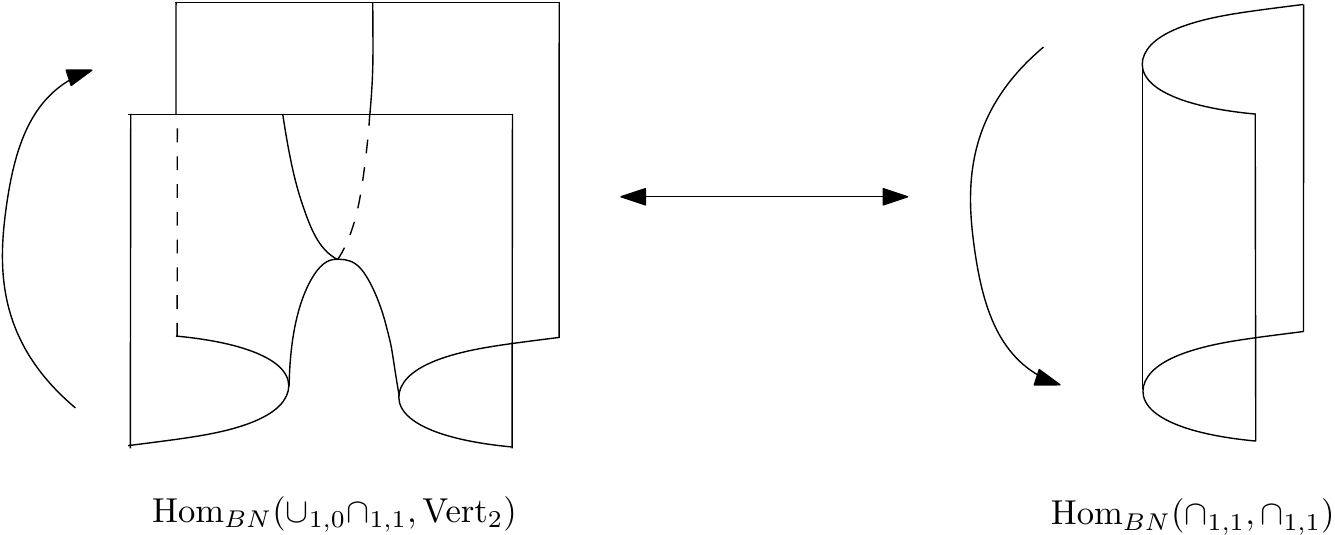}}
\caption{Adjointness of cobordisms}
\label{cobadjoint}
\end{center}
\end{figure}

\begin{lemma}
The maps $\phi_{T_1,T_2}$ respect adjointness.
\end{lemma}

\begin{proof}
We show the statement in the case of $\cup$ being left adjoint to $\cap$; the other case is analogous. Now if $T_1, T_2$ are planar $(2m,2n)$-tangles, such that $T_1 = \cup_{i,m-1} T_1'$ with $T_1'$ a $(2(m-1), n)$-tangle, we wish to show that the following diagram commutes:

\centerline{
\xymatrixcolsep{5pc}
\xymatrix{
\mbox{Hom}_{BN}(T_1,T_2) \ar[r]^{\phi_{T_1,T_2}} \ar[d]^{\cong} & \mbox{Hom}_{(m,n)}(T_1, T_2) \ar[d]^{\cong} \\
\mbox{Hom}_{BN}(T_1', \cap_{i,m} T_2) \ar[r]^{\phi_{T_1',\cap_{i,m} T_2}} & \mbox{Hom}_{(m-1,n)}(T_1', \cap_{i,m} T_2)}
}

Let $S \in \mbox{Hom}_{BN}(T_1, T_2)$ be a cobordism in the Bar-Natan module, in the top left corner of the above commutative diagram. Taking the path down and then right, we first turn $S$ into an element of $\mbox{Hom}_{BN}(T_1', \cap_{i,m}T_2)$ by bending the component of $S$ whose boundary is $\cup_{i, m-1}$ in $T_1$ up to the top boundary and then looking at the induced bimodule homomorphism:
\begin{figure}[H]
\begin{center}
{\includegraphics[scale=.8]{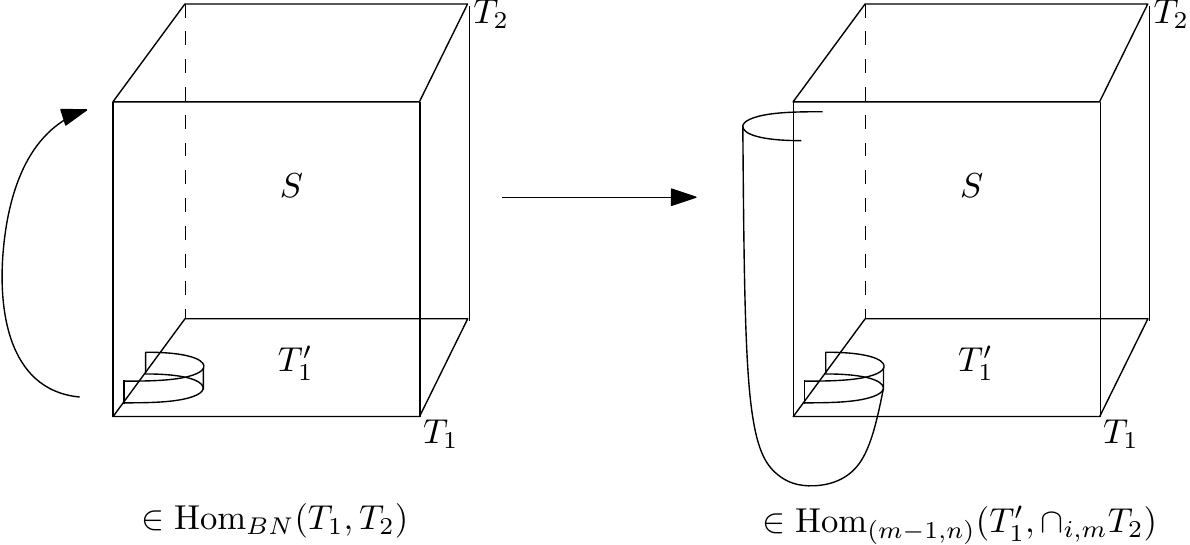}}
\end{center}
\end{figure}

On the other hand, taking the path right and then down amounts to first considering $S$ as a bimodule map and then applying adjointness on the bimodule side. However, as we discussed above, adjointness of bimodule maps can be visualized as in Figure \ref{adjoint1}, which consists of attaching a new identity cobordism from an arc to itself and capping off the resulting circle on the bottom boundary with a cobordism from the empty manifold:
\begin{figure}[H]
\begin{center}
{\includegraphics[scale=.8]{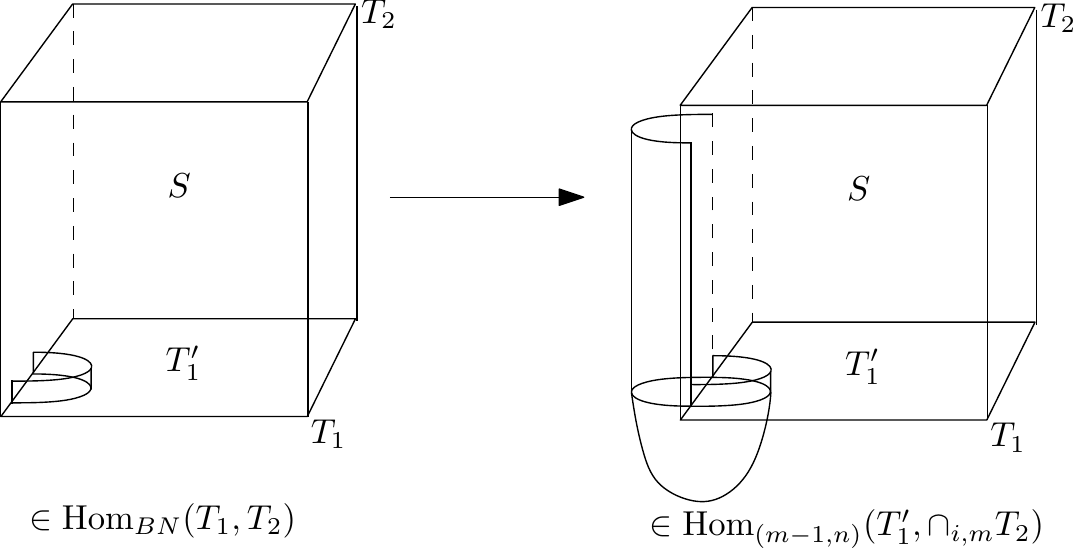}}
\end{center}
\end{figure}

It's clear that these two processes of bending the existing cobordism and attaching the cobordism described above are isotopic, and thus they induce the same bimodule homomorphisms. $\blacksquare$
\end{proof}

Using adjointness of bimodule homomorphisms, we can show the following:

\begin{proposition}
\label{decomp}
Any bimodule hom-space $\mbox{Hom}_{(m,n)}(T_1,T_2)$ is isomorphic to a direct sum of hom-spaces of the form $\mbox{Hom}_{(k,k)}(\mbox{Vert}_{2k}, \mbox{Vert}_{2k}), k \leq m, n.$
\end{proposition}

\begin{proof}
First note that by adjointness, any hom-space $\mbox{Hom}_{(m,n)}(T_1,T_2)$ is isomorphic to \\ $\mbox{Hom}_{(m,m)}(T_1W(T_2), \mbox{Vert}_{2m})$, where $W(T_2)$ is the reflection of $T_2$. If $T_1W(T_2)$ contains a circle, then the hom-space is isomorphic to $\mbox{Hom}_{(m,m)}((T_1W(T_2))', \mbox{Vert}_{2m})^{\oplus 2}$, where $(T_1W(T_2))'$ is $T_1W(T_2)$ with the circle removed. This follows from a result of Khovanov in \cite{khov}, which specialized to our case says that
\[ \mathcal{F}(T_1W(T_2)) \cong \mathcal{F}((T_1W(T_2))') \otimes \mathcal{A} \cong \mathcal{F}((T_1W(T_2))') \{1\} \oplus \mathcal{F}((T_1W(T_2))') \{-1\}.\]
So we may reduce to the case that $T_1W(T_2)$ has no circles. Therefore we have reduced the problem to showing that $\mbox{Hom}_{(m,m)}(T, \mbox{Vert}_{2m})$ has the desired property for any planar $(2m,2m)$-tangle $T$.

Now we proceed by induction on $m$. If $m=0$, the statement is clear. If $T= \mbox{Vert}_{2m}$, the statement is also clear. Otherwise, $T$ can be written in the form $T= \cup_{i, m-1}T'$ for some $i$ and some $T'$. Then by adjointness,
\begin{eqnarray*}
\mbox{Hom}_{(m,m)}(T, \mbox{Vert}_{2m}) & \cong & \mbox{Hom}_{(m-1,m)}(T', \cap_{i, m}) \\
& \cong& \mbox{Hom}_{(m-1,m-1)}(T'\cup_{i,m-1}, \mbox{Vert}_{2(m-1)})
\end{eqnarray*}
and we are done by induction.
$\blacksquare$
\end{proof}

Together with the matching isomorphism on the Bar-Natan side and using that the maps $\phi_{T_1,T_2}$ respect adjointness, Proposition \ref{decomp} gives us the following commutative diagram:

\centerline{
\xymatrix{
\mbox{Hom}_{BN}(T_1,T_2) \ar[r]^-{\cong} \ar[d]^{\phi_{T_1,T_2}} & \bigoplus_k \mbox{Hom}_{BN}(\mbox{Vert}_{2k}, \mbox{Vert}_{2k}) \ar[d]^{\oplus \phi_{\mbox{Vert}_{2k}, \mbox{Vert}_{2k}}} \\
\mbox{Hom}_{(m,n)}(T_1, T_2) \ar[r]^-{\cong} & \bigoplus_k \mbox{Hom}_{(k,k)}(\mbox{Vert}_{2k}, \mbox{Vert}_{2k})}
}

In other words, the decomposition of $\mbox{Hom}_{(m,n)}(T_1,T_2)$ is compatible with the homomorphism $\phi_{T_1,T_2}$.

Given this decomposition, we see that $\phi_{T_1,T_2}: \mbox{Hom}_{BN}(T_1,T_2) \to \mbox{Hom}_{(m,n)}(T_1,T_2)$ is a surjective map of $R$-modules, and we can describe its kernel. We know that the kernel of each map $\phi_{\mbox{Vert}_{2k}, \mbox{Vert}_{2k}}$ is generated by the first $k$ elementary symmetric functions in $c_1, \ldots, c_{2k}$. The kernel of $\phi_{T_1,T_2}$ is then given by the elements which correspond under adjointness to the kernel elements of $\phi_{\mbox{Vert}_{2k}, \mbox{Vert}_{2k}}$ for each $k$ that appears in the decomposition of Proposition \ref{decomp}.

\begin{example}
Consider the map $\phi_{T_1,T_2}$ for $T_1 = T_2 = \cup_{1,1}$. Note that via adjointness we have
\[ \mbox{Hom}_{(2,1)}(\cup_{1,1}, \cup_{1,1}) \cong \mbox{Hom}_{(1,1)} (\mbox{Vert}_2, \cap_{1, 2} \cup_{1,1}) \cong \mbox{Hom}_{(1,1)}(\mbox{Vert}_2, \mbox{Vert}_2)^{\oplus 2}, \]
with the same isomorphisms on the Bar-Natan side. We know that the kernel of $\phi_{\mbox{Vert}_2, \mbox{Vert}_2}$ is generated by $e_1(c_1,c_2) \in \mbox{End}_{BN}(\mbox{Vert}_2)$. Using adjointness to find the corresponding elements of $\mbox{End}_{BN}(\cup_{1,2})$, we see that the kernel of $\phi_{T_1,T_2}$ is given by $R$-linear combinations of the rightmost cobordisms in the  correspondences of Figure \ref{kernelex}. Recall that in the direct sum decomposition $ \mbox{Hom}_{(1,1)} (\mbox{Vert}_2, \cap_{1, 2} \cup_{1,1}) \cong \mbox{Hom}_{(1,1)}(\mbox{Vert}_2, \mbox{Vert}_2)^{\oplus 2}$, the map from the right-hand side to the left-hand side assigns a factor of $1$ to the circle for each map in the first summand and a factor of $X$ for each map in the second summand.
\begin{figure}[htb]
\begin{center}
\includegraphics[scale=1.2]{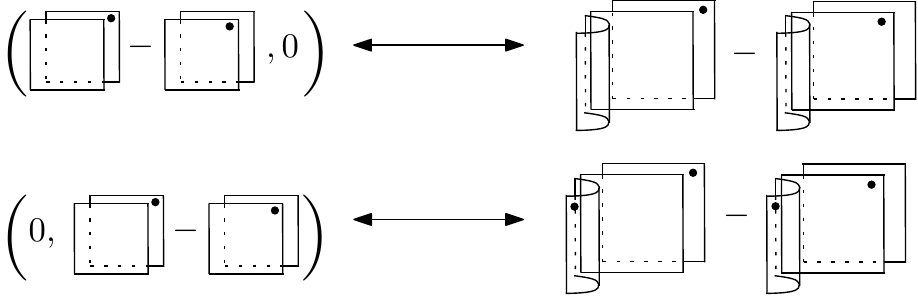}
\end{center}
\caption{Kernel generators of $\phi_{\mbox{Vert}_2, \mbox{Vert}_2}^{\oplus 2}$ and the corresponding elements of $\mbox{ker}(\phi_{\cup_{1,2}, \cup_{1,2}})$.}
\label{kernelex}
\end{figure}

\end{example}

Considering all isotopy classes of tangles $T_1, T_2$ with the same number of endpoints, we also have that the map $\phi_{m,n}$ is surjective as a map of vector spaces, with kernel spanned by the union of elements spanning the kernel of each $\phi_{T_1,T_2}$. We now proceed to give a better description of these kernel elements.

\begin{definition}
Let $a$ be any element of $B^m_n$. For $1 \leq i \leq 2m$, let $c_i(a)$ be the element of $\mbox{End}_{BN}(a)$ obtained by putting a single dot on top of the $i$th endpoint of $a$, crossing with the interval $[0,1]$, and introducing a sign of $(-1)^{i+1}$. For $1 \leq i \leq 2n$, let $c_i'(a)$ be the analogous element of $\mbox{End}_{BN}(a)$ with dots placed at the bottom of $a$.  Then define $e_k(a) \in \mbox{End}_{BN}(a)$ to be the $k$th elementary symmetric function in $c_1(a), \ldots, c_{2m}(a)$ and $e_k(a)'$ to be the $k$th elementary symmetric function in $c_1'(a), \ldots, c_{2n}'(a)$.
\end{definition}

\begin{lemma}
$e_k(a) =  e_k'(a)$ for $1 \leq k \leq \min \{2m,2n \}$.
\end{lemma}

\begin{proof}
Any dots placed on strands of $a$ which don't belong to a cup or cap may be slid from one end to the other, so $e_k(a)$ and $e_k(a)'$ might only differ due to dots placed on cups or caps. If $a$ contains a cup $\cup_{i,m-1}$, then $c_i(a) + c_{i+1}(a) = 0$, so any terms in $e_k(a)$ involving $c_i(a)$ or $c_{i+1}(a)$ will disappear for all possible $k$. Similarly, if $a$ contains a cap $\cap_{i,n}$, then $c_i'(a) + c_{i+1}'(a) = 0$, so any terms involving $c_i'(a)$ or $c_{i+1}'(a)$ in $e_k'(a)$ disappear. Therefore both $e_k(a)$ and $e_k'(a)$ are only left with terms involving strands that run from the top to the bottom of $a$.
\end{proof}

\begin{definition}
Given $a \in B^m_n$, define the \emph{width} $w(a)$ to be the number of through-strands of $a$, that is, the number of strands that run from the top of $a$ to the bottom.
\end{definition}

\begin{corollary}
$e_k(a) = 0$ for $k > w(a)$.
\end{corollary}

\begin{proof}
In the proof of the previous proposition, we saw that for any $k$, any terms of $e_k(a)$ corresponding to placing dots on a cup or cap of $a$ disappeared. Therefore the only surviving terms come from placing $k$ dots on $w(a)$ strands. So if $k > w(a)$, then $e_k(a) = 0$.
\end{proof}

\begin{proposition}
The kernel of $\phi_{m,n}$ is the two-sided ideal $I$ generated by the elements $e_k(a), 1 \leq k \leq w(a)$, for all $a$ in $B^m_n$.
\end{proposition}

\begin{proof}
We know that $I$ consists of those elements of $\mbox{Hom}_{BN}(a,b)$, $a,b \in B^m_n$, which correspond under adjointness to the first $k$ elementary symmetric functions in $c_1, \ldots, c_{2k}$ for each $k$ appearing in the decomposition $\mbox{Hom}_{BN}(a,b) \cong \bigoplus_{k} \mbox{End}_{BN}(\mbox{Vert}_{2k})$.

Each $e_k(a)$ is clearly one of these elements, since $e_k(a)$ corresponds under adjointness to $e_k(c_1, \ldots, c_{2l}) \in \ker (\phi_{\mbox{Vert}_{2l}})$, where $l$ is the number of through-strands in $a$. On the other hand, suppose $S \in \mbox{Hom}_{BN}(a,b)$ corresponds under adjointness to $e_p(c_1, \ldots, c_{2k}), 1 \leq p \leq k$, for some $k$ appearing in the above decomposition. If the $k$-summand of the decomposition did not come from removing a circle that arose under adjointness, then by sliding all dots on $S$ to the boundary containing $a$, it is clear that $S$ is a multiple of $e_p(a)$. If the $k$-summand did come from removing a circle, then $a$ and $b$ both contain some $\cup_{i,m}$ or $\cap_{i,n}$, and $S$ might have a dot on a component of the form $\cup_{i,m} \times [0,1]$ or $\cap_{i,n} \times [0,1]$, as in the second correspondence of Figure \ref{kernelex}. In that case, slide all dots except those to the bottom, and it is clear that $S$ is again a multiple of $e_p(a)$.
$\blacksquare$

\begin{comment}
Slide any dots on $S$ to the boundary containing $a$, except potentially any dots that appear on components of the form $\cup_{i,m} \times [0,1]$ or $\cap_{i,n} \times [0,1]$. 

Slide all dots in $e_p(c_1, \ldots, c_{2k})$ to the bottom. 

Note that if $a$ and $b$ both contain some $\cup_{i,m}$ or $\cap_{i,n}$, then there is a piece of the decomposition $\mbox{Hom}_{BN}(a,b) \cong \bigoplus_{k} \mbox{End}_{BN}(\mbox{Vert}_k)$ 

If $a$ and $b$ do not both contain some $\cup_{i,m}$ or $\cap_{i,n}$, then $S$ must be a multiple of $e_p(a)$. 

Slide any dots on $S$ to the boundary containing $a$. 
\end{comment}
\end{proof}

In summary, we have shown the following Theorem:

\begin{theorem}
\label{thm:bimod-kernel}
$\phi_{m,n}$ is a surjective ring homomorphism. Its kernel is the two-sided ideal generated by the elements $e_k(a)$ for $a \in B^m_n$ and $1 \leq k \leq w(a)$.
\end{theorem}

\begin{comment}
WANT TO EXPLAIN WHY THIS IS ACTUALLY A 2-SIDED IDEAL.
It is clear that the two-sided ideal generated by the elements of each $Hom_{BN}(T_1,T_2)$ corresponding under adjointness to the elementary symmetric functions in the $(-1)^{i+1}c_i$ really does give the kernel of $\phi_{m,n}$: The previous work shows any element in the kernel belongs to this ideal.
\end{comment}

\section{Properties of the rings $\mbox{Hom}(m,n)$}
Now that we can express $\mbox{Hom}(m,n)$ as a quotient of the better-understood ring $\mbox{Hom}_{BN}(m,n)$, we can investigate some of its properties. First, it will be useful to note the following:

\begin{proposition}
\label{nBNiso}
 The rings $\mbox{Hom}_{BN}(m,n)$ and $H^{m+n}$ are isomorphic.
\end{proposition}

\begin{proof}
Define a map $\gamma$ from $H^{m+n}$ to $\mbox{Hom}_{BN}(m,n)$ as follows. Consider a generator of ${}_b(H^{m+n})_a$ geometrically as the diagram $W(b)a$ with a dot on each circle for each $X$ in the corresponding tensor factor. The map $\gamma$ takes such an element, caps off each of the $m+n$ circles with disks, slices it where $W(b)$ meets $a$ and bends $a$ to the bottom of the cobordism (producing an element of $\mbox{Hom}_{BN}(0, m+n)$) and then bending the first $m$ strands over to the other side.

This map is a ring homomorphism, since the merging or splitting of circles that occurs in the multiplication of $H^{m+n}$ corresponds with the vertical stacking of the resulting cobordisms. The process described above is clearly reversible, showing that it is an isomorphism.
$\blacksquare$
\end{proof}

\begin{example}
Figure \ref{m+niso} shows an example for $m=n=1$.
\begin{figure}[htb]
\begin{center}
\includegraphics{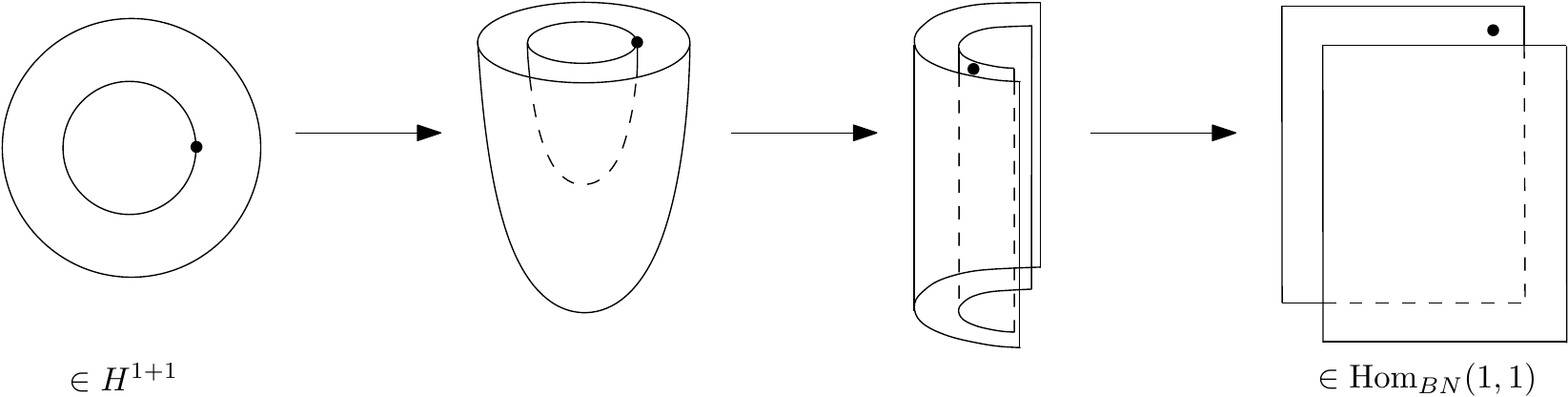}
\end{center}
\caption{Isomorphism $\gamma$ from $H^2$ to $\mbox{Hom}_{BN}(1,1)$}
\label{m+niso}
\end{figure}
\end{example}

The isomorphism of Proposition \ref{nBNiso} allows us to give a set of mutually orthogonal idempotents in $\mbox{Hom}(m,n)$.
\begin{proposition}
\label{idems}
For $a \in B^m_n$, let $\mbox{id}_a$ in $\mbox{Hom}_{BN}(m,n)$ be the identity cobordism given by $a \times [0,1]$. Then the set of elements $\{ \phi_{m,n}(\mbox{id}_a) \}_{a \in B^m_n}$ in $\mbox{Hom}(m,n)$ forms a complete set of mutually orthogonal idempotents.
\end{proposition}

\begin{proof}
In $H^{m+n}$, idempotents are given by $1_{\overline{a}} = 1^{\otimes (m+n)} \in {}_{\overline{a}}(H^{m+n})_{\overline{a}}$ for all $\overline{a} \in B^{m+n}$. They are mutually orthogonal and satisfy 
\[ 1 = \sum_{\overline{a} \in B^{m+n}} 1_{\overline{a}}.\]
Under the isomorphism $\gamma$, the element $1_{\overline{a}}$ corresponds to the cobordism $id_a$, where $a$ is the $(2m,2n)$-tangle obtained from $\overline{a}$ by bending. Since $\gamma$ and $\phi_{m,n}$ are both ring homomorphisms and take $1$ to $1$, this completes the proof.
\end{proof}

Therefore as a left module over itself, we have the following decomposition of $\mbox{Hom}(m,n)$;
\[ \mbox{Hom}(m,n) \cong \bigoplus_{a \in B^m_n} \mbox{Hom}_{(m,n)}( -, a).\]

We can also partially describe the center of the rings $\mbox{Hom}(m,n)$. First, we give an explicit description of the center of the rings $H^n$.
\begin{proposition}
Define $\overline{c_i}$ to be the element of $H^n$ given by $\sum_{a \in B^n} x_i(a)$, where $x_i(a)$ is the element of ${}_a(H^n)_a$ corresponding to the diagram $W(a)a$ with a dot on the $i$th endpoint on the center line and a sign of $(-1)^{i+1}$. Then there is an isomorphism
\begin{eqnarray*}
\mathbb{Z}[x_1,\ldots, x_{2n}]/ ( x_1^2, \ldots, x_{2n}^2, e_1(x_1,\ldots,x_{2n}), \ldots, e_n(x_1, \ldots, x_{2n})) &\to& Z(H^n) \\
x_i &\mapsto& \overline{c_i}.
\end{eqnarray*}
\end{proposition}

\begin{proof}
We know from \cite{crossmatch} that the two spaces are isomorphic. The claim here is that the above map is an explicit isomorphism. It easy to see that the $\overline{c_i}$ are in $Z(H^n)$. The fact that the $\overline{c_i}$ satisfy only the relations $\overline{c_i}^2 = 0$ and $e_i(\overline{c_1}, \ldots, \overline{c_{2n}}) = 0$ for $1 \leq i \leq 2n$ follows from the same argument used in the proof of Proposition \ref{kernelprop}.
\end{proof}

Therefore we may identify each $x_i$ with the explicit element $\overline{c_i}$. Because $\phi_{m,n}$ is surjective, it maps the center of $\mbox{Hom}_{BN}(m,n)$ into the center of $\mbox{Hom}(m,n)$.

\begin{proposition}
\label{prop:Z-Hom}
The image of the map $\phi_{m,n}$ restricted to $Z(\mbox{Hom}_{BN}(m,n))$, which is isomorphic to the center of $H^{m+n}$, is isomorphic to $Z(H^m) \otimes Z(H^n)$. That is,
\[ \mbox{im}(\phi_{m,n} |_Z) \cong Z(H^m) \otimes Z(H^n).\]
\end{proposition}

\begin{proof}
The generators $x_1, \ldots, x_{2(m+n)}$ of $Z(H^{m+n})$ can be translated into generators $\gamma(x_1), \ldots, \gamma(x_{2(m+n)})$ of $Z(\mbox{Hom}_{BN}(m,n))$, where $\gamma$ is the bending map of Proposition \ref{nBNiso}. For convenience, since each $\gamma(x_i)$ is a sum of dotted tangles times the interval $[0,1]$, we can restrict to two dimensions by considering the horizontal cross-section of each cobordism, obtaining a sum of dotted $(2m,2n)$-tangles. Similarly, in the definition of the generators $e_k(a)$ of $\ker(\phi_{m,n})$, we can replace $c_i(a)$ with its cross-section, so that $e_k(a)$ can be thought of as a linear combination of dotted $(2m,2n)$-tangles.

Note that because $\ker(\phi_{m,n})$ is generated by the elementary symmetric functions of dots on through-strands, relations on the $\phi_{m,n}(\gamma(x_1)), \ldots, \phi_{m,n}(\gamma(x_{2(m+n)}))$ will be generated by the first $n$ elementary symmetric functions in $\phi_{m,n}(\gamma(x_1)), \ldots, \phi_{m,n}(\gamma(x_n))$ and the first $m$ elementary symmetric functions in $\phi_{m,n}(\gamma(x_{n+1})), \ldots, \phi_{m,n}(\gamma(x_{2(m+n)}))$. The $\phi_{m,n}(\gamma(x_i))$ are then completely determined by their components involving split $(2m,2n)$-tangles, where a split tangle is one with no through-strands. 

For each generator $\phi_{m,n}(\gamma(x_i))$ of $\mbox{im}(\phi_{m,n}|_Z)$, express the generator as a sum
\[\phi_{m,n}(\gamma(x_i)) = \sum_j S_{i_j} + \sum_k T_{i_k},\]
where the $S_{i_j}$ are the components of $\phi_{m,n}(\gamma(x_i))$ that are dotted split tangles and the $T_{i_k}$ contain through-strands. Each dotted split tangle $S_{i_j}$ is of the form $a_{i_j} W(b_{i_j})$, with $a_{i_j} \in B^m$ and $b_{i_j} \in B^n$, where each arc may carry up to one dot. Define a map
\[ \alpha: \mbox{im}(\phi_{m,n}|_Z) \to Z(H^m) \otimes Z(H^n) \]
that behaves on generators by ignoring the components $T_{i_k}$, and turns each $S_{i_j}$ into an element of $Z(H^m) \otimes Z(H^n)$ by attaching $b_{i_j}$ to the bottom of $S_{i_j}$ and $W(a_{i_j})$ to the top to get $n$ dotted circles at the bottom half of the diagram and $m$ dotted circles at the top. Each collection of dotted circles gets mapped to the element of ${}_{a_{i_j}}(H^m)_{a_{i_j}}$ and ${}_{b_{i_j}}(H^n)_{b_{i_j}}$, respectively, that has a $1$ in each tensor factor corresponding to an undotted circle and $X$ for each dotted circle.

Then $\alpha$ is an isomorphism by the above discussion. $\blacksquare$
\end{proof}

\begin{example}
When $m=n=1$, $Z(H^{1+1})$ is generated by $1, x_1, x_2, x_3, x_4$, subject to $x_i^2 = 0$ and the elementary symmetric functions on the $x_i$. The composition of $\phi_{1,1}$ with $\alpha$ on generators $\gamma(x_1), \ldots, \gamma(x_4)$ of $Z(\mbox{Hom}_{BN}(1,1))$ is pictured below. The diagrams on the left represent two-dimensional identity cobordisms from the pictured tangle to itself with dots on the indicated sheets. In the image $\mbox{im}(\phi_{1,1}|_Z)$, $\phi_{1,1}(\gamma(x_1)) = \phi_{1,1}(\gamma(x_2))$ and $\phi_{1,1}(\gamma(x_3)) = \phi_{1,1}(\gamma(x_4))$. Therefore $\alpha$ is a well-defined isomorphism.
\begin{figure}[H]
\begin{center}
\includegraphics{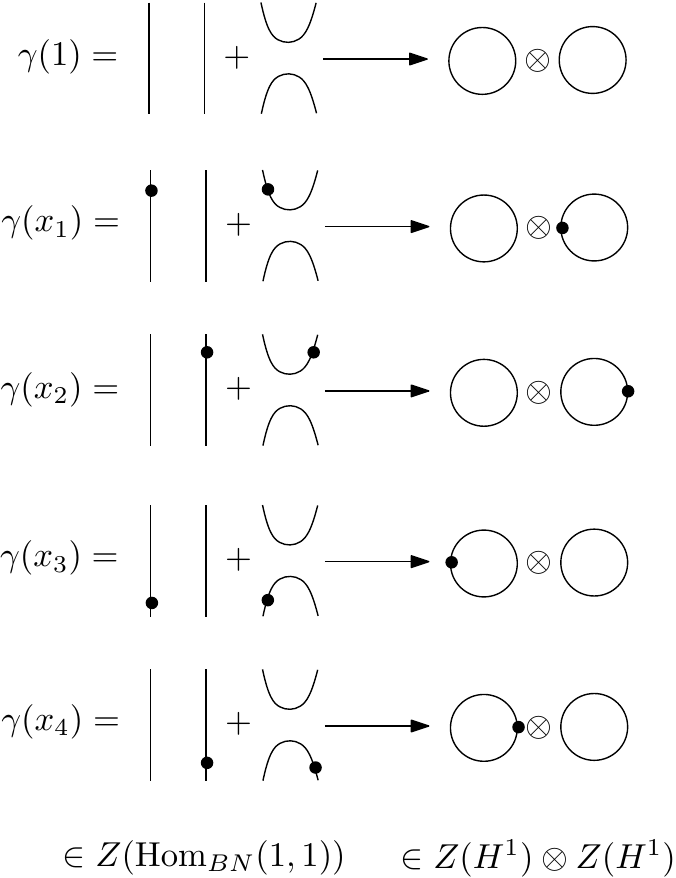}
\end{center}
\end{figure}
\end{example}

\begin{conjecture}
We conjecture that the map $\phi_{m,n}|_{Z}: Z(\mbox{Hom}_{BN}(m,n)) \to Z(\mbox{Hom}(m,n))$ is surjective, so that
\[ Z(\mbox{Hom}(m,n)) \cong Z(\mbox{Hom}_{BN}(m,n))  / (\mbox{ker}(\phi_{m,n}) \cap Z(\mbox{Hom}_{BN}(m,n))) \]
\end{conjecture}
\begin{remark}
The conjecture was confirmed via Magma for $m=n=1$.
\end{remark}

\begin{comment}
\begin{proof}
Start of proof: We wish to show that if $\overline{r}$ is in $Z(\mbox{Hom}(m,n))$, then for some preimage $r$ under $\phi_{m,n}$, $r$ is an element of $Z(\mbox{Hom}_{BN}(m,n))$. Note that the condition $\overline{r} \in Z(\mbox{Hom}(m,n))$ is equivalent to saying that $rs-sr$ is in $\mbox{ker}(\phi_{m,n})$ for all preimages $r$ of $\overline{r}$ and all elements $s$ of $\mbox{Hom}_{BN}(m,n)$. Therefore, we look to show the contrapositive: If $\overline{r}$ is such that $r \notin Z(\mbox{Hom}_{BN}(m,n))$ for any preimage $r$, then $r's-sr' \notin \mbox{ker}(\phi_{m,n})$ for some preimage $r$ and some $s \in \mbox{Hom}_{BN}(m,n)$.
\end{proof}
\end{comment}

\section{$HH_0(\mbox{Hom}(m,n))$ as a quotient of the Russell skein module}
\label{sec:quotient-Russell}

In Chapter \ref{ch:equiv}, we saw that $HH_0(H^n)$ was isomorphic to the Bar-Natan--Russell skein module. Now we turn to $HH_0(\mbox{Hom}(m,n))$. We define an analogue of the Russell skein module and show that it is isomorphic to $HH_0(\mbox{Hom}(m,n))$.

\begin{definition}
Define the diagrammatic skein module $R_{(m,n)}$ to be the $\mathbb{Z}$-module generated by crossingless matchings of $2(m+n)$ points, where arcs may carry up to one dot, modulo the following relations:
\begin{itemize}
\item Type I and Type II relations as before.
\item For each generator of $\phi_{m,n}$ over $\mathbb{Z}$, pass to an element of $H^{m+n}$ via the inverse of the map $\gamma$ of Proposition \ref{nBNiso}. Slide all dots to the bottom half of the diagram and cut off the top half, leaving a (sum of) dotted crossingless matchings.
\end{itemize}
\end{definition}

\begin{example}
$R_{(1,1)}$ has one Type I relation, one Type II relation, and two relations from $\ker(\phi_{1,1})$:
\begin{figure}[H]
\begin{center}
\includegraphics{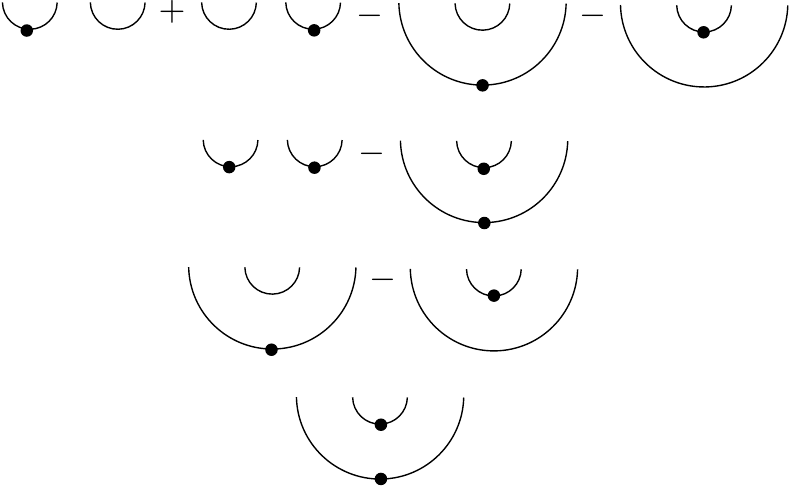}
\end{center}
\end{figure}
\end{example}

\begin{proposition}
$HH_0(\mbox{Hom}(m,n))$ and $R_{(m,n)}$ are isomorphic as $\mathbb{Z}$-modules.
\end{proposition}

\begin{proof}
Recall that $\mbox{Hom}(m,n)$ can be expressed as a quotient of $\mbox{Hom}_{BN}(m,n)$, which is isomorphic to $H^{m+n}$. For any ring $R$ and any ideal $I$, there is an isomorphism $(R/I)/[R/I,R/I] \cong R/([R,R]+I)$. Therefore
\[ HH_0(\mbox{Hom}(m,n)) \cong H^{m+n}/([H^{m+n},H^{m+n}]+\gamma^{-1}(\ker(\phi_{m,n})))\]
and the statement follows from our earlier description of the commutator of the rings $H^n$ and the previous proposition.
\end{proof}

\section{The case $X^2 = t$}
\label{sec:x^2=t}
We now extend the results of Section \ref{main} to the $SU(2)$-equivariant deformation of the rings $H^n$. This can be viewed as a specialization of the $U(2)$-equivariant case of Section \ref{equi-Hn} where instead of the $\mathcal{A}_{h,t} \cong \mathbb{Z}[X]/(X^2 = hX + t)$, we specialize $h$ to $0$, so the relations simply becomes $X^2 = t$. We denote the deformed arc rings in this case by $H^n_t$. The local Bar-Natan relations in this setting are the same as above, except that now
\[
 \begin{array}{c}\includegraphics[height=10mm]{ddot.pdf}\end{array}
  \hspace{-4mm}=t.
\]

First, we establish analogues of the elementary symmetric functions. For any natural number $n$, we define functions $e_k^t(x_1, \ldots, x_{2n})$, $1 \leq k \leq n$ inductively. As the base, set $e_1^t(x_1, \ldots, x_{2n}) = e_1(x_1,\ldots, x_{2n})$. To find $e_k^t(x_1, \ldots, x_{2n})$, consider the expansion of  $(x_1 + \cdots + x_{2n})^k$ under the conditions $x_i^2 = t$. The expansion will be of the form
\[ (x_1 + \cdots + x_{2n})^k = k! e_k(x_1,\ldots, x_{2n}) + \sum_{i=1}^{k-1} c_{i,k}t^{(k-i)/2} e_i(x_1, \ldots, x_{2n}) + c_{0,k} t^{k/2}\]
where $c_{i,k} = 0$ if $i \neq k$ mod $2$, $0 \leq i \leq k$.
Inductively define
\[ e^t_k(x_1, \ldots, x_{2n}) := \frac{1}{k!} \left( (x_1 + \cdots x_{2n})^k - \sum_{i=1}^{k-1} c_{i,k}t^{(k-i)/2} e^t_i(x_1, \ldots, x_{2n})  \right). \]

For example, when $n=4$ we obtain
\begin{eqnarray*}
e_1^t(x_1, \ldots, x_8) &=& e_1(x_1, \ldots, x_8) \\
e_2^t(x_1, \ldots, x_8) &=& e_2(x_1, \ldots, x_8) + 4t \\
e_3^t(x_1, \ldots, x_8) &=& e_3(x_1, \ldots, x_8) \\
e_4^t(x_1, \ldots, x_8) &=& e_4(x_1, \ldots, x_8) -6t^2.
\end{eqnarray*}

Note in particular that $e_k^t(x_1, \ldots, x_{2n}) = e_k(x_1, \ldots, x_{2n})$ when $k$ is odd.

We will show that the map $\phi_{m,n}^t: \mbox{Hom}^t_{BN}(m,n) \to \mbox{Hom}^t(m,n)$ is surjective and describe its kernel using an argument parallel to that in Section \ref{main}.

\begin{proposition}
\label{kernelpropt}
As a map of $\mathbb{Z}$-modules, the kernel of $\phi^t_{\mbox{Vert}_{2n}, \mbox{Vert}_{2n}}$ is the two-sided ideal generated by $e_1^t(c_1, \ldots, c_{2n})$, \ldots, $e_n^t(c_1, \ldots, c_{2n})$.
\end{proposition}

\begin{proof}
The argument is completely analogous to that of Proposition \ref{kernelprop}. Our kernel $I^t$ is again given by
\[ I^t = \bigcap_{((i_1, j_1), \ldots, (i_n, j_n)) \in B^n} (c_{i_1}+c_{j_1}, \ldots, c_{i_n}+c_{j_n}). \]

Again we assume inductively that
\[
{I^t = \bigcap_{k=1}^n (c_k + c_{k+1}, e_1^t(c_1, \ldots, \widehat{c_k, c_{k+1}}, \ldots, c_{2n}), \ldots, e^t_{n-1}(c_1, \ldots, \widehat{c_k, c_{k+1}}, \ldots, c_{2n}))}.
\]

By manipulating generators, we see that
\begin{eqnarray*}
I ^t &=& \bigcap_{k=1}^n(c_k+c_{k+1}, e_1^t(c_1, \ldots, c_{2n}), \ldots, e^t_{n-1}(c_1, \ldots, c_{2n})) \\
&=& ((c_1+c_2)(c_2+c_3)\cdots(c_n+c_{n+1}), e_1^t(c_1, \ldots, c_{2n}), \ldots, e^t_{n-1}(c_1, \ldots, c_{2n})).
\end{eqnarray*}
\begin{comment}
&=& \bigcap_{k=1}^n(c_k+c_{k+1}, e_1^t(c_1, \ldots, \widehat{c_k, c_{k+1}}, \ldots, c_{2n})+(c_k+c_{k+1}), \ldots, \\
& & \qquad e^t_{n-1}(c_1, \ldots, \widehat{c_k, c_{k+1}}, \ldots, c_{2n})  +(c_k+c_{k+1})e_{n-2}(c_1, \ldots, \widehat{c_k, c_{k+1}}, \ldots, c_{2n})+ \\
& & \qquad (c_kc_{k+1}+t)e^t_{n-3}(c_1, \ldots, \widehat{c_k, c_{k+1}}, \ldots, c_{2n})) \\
\end{comment}

To see the first equality, note that if $j$ is odd, then $e^t_j(c_1, \ldots, c_{2n}) = e_j(c_1, \ldots, c_{2n})$, so
\begin{eqnarray*}
e^t_j(c_1, \ldots, c_{2n}) &=& e^t_j(c_1, \ldots, \widehat{c_k, c_{k+1}}, \ldots, c_{2n}) + (c_k + c_{k+1})e_{j-1}^t(c_1, \ldots, \widehat{c_k, c_{k+1}}, \ldots, c_{2n})+\\
& & \qquad (c_kc_{k+1})e_{j-2}^t(c_1, \ldots, \widehat{c_k, c_{k+1}}, \ldots, c_{2n}) - mt^{(j-1)/2}(c_k + c_{k+1})
\end{eqnarray*}
where $m$ is the integer such that 
\[ e_{j-1}^t(c_1, \ldots, \widehat{c_k, c_{k+1}}, \ldots, c_{2n}) = e_{j-1}(c_1, \ldots, \widehat{c_k, c_{k+1}}, \ldots, c_{2n}) + mt^{(j-1)/2}.\]
If $j$ is even, define $m_n^j, m_{n-1}^j,$ and $m_{n-1}^{j-2}$ so that
\begin{eqnarray*}
e^t_j(c_1, \ldots, c_{2n}) &=& e_j(c_1, \ldots, c_{2n}) + m^j_n t^{(j-1)/2} \\
e^t_j(c_1, \ldots, \widehat{c_k, c_{k+1}}, \ldots, c_{2n}) &=& e_j(c_1, \ldots, \widehat{c_k, c_{k+1}}, \ldots, c_{2n}) + m^j_{n-1}t^{(j-1)/2} \\
e^t_{j-2}(c_1, \ldots, \widehat{c_k, c_{k+1}}, \ldots, c_{2n}) &=& e_{j-2}(c_1, \ldots, \widehat{c_k, c_{k+1}}, \ldots, c_{2n}) + m^{j-2}_{n-1}t^{(j-3)/2} 
\end{eqnarray*}
Then
\begin{eqnarray*}
e_j^t(c_1, \ldots, c_{2n}) &=& e^t_j(c_1, \ldots, \widehat{c_k, c_{k+1}}, \ldots, c_{2n}) + (c_k + c_{k+1})e_{j-1}^t(c_1, \ldots, \widehat{c_k, c_{k+1}}, \ldots, c_{2n})+\\
& & \qquad (c_kc_{k+1} + \frac{m_n^j-m_{n-1}^j}{m_{n-1}^{j-1}}t) e^t_{j-2}(c_1, \ldots, \widehat{c_k, c_{k+1}}, \ldots, c_{2n}).
\end{eqnarray*}

Finally, by observing that
\begin{eqnarray*}
(c_1 +c_2)\cdots(c_n+c_{n+1}) &=& e_n^t(c_1, \ldots, c_{2n}) - e_1(c_{n+2}, \ldots, c_{2n})e_{n-1}^t(c_1, \ldots, c_{2n}) + \cdots + \\
& & (-1)^{n-1}e_{n-1}(c_{n+2}, \ldots, c_{2n})e_1^t(c_1, \ldots, c_{2n}) +  \\
& & nt(e_{n-2}^t(c_1, \ldots, c_{2n}) + \cdots + (-1)^{n-3}e_{n-3}(c_{n+2}, \ldots, c_{2n})e_1^t(c_1, \ldots, c_{2n}))
\end{eqnarray*}
we see that
\[ I^t = (e_1^t(c_1, \ldots, c_{2n}), \ldots, e_{n}^t(c_1, \ldots, c_{2n})) \]
as desired.
$\blacksquare$
\end{proof}

\begin{proposition}
The center of the ring $H^n_t$ is given by 
\[ \mathbb{Z}[x_1, \ldots, x_{2n}]/(x_1^2 -t,\ldots, x_{2n}^2-t, e_1^t(x_1, \ldots, x_{2n}), \ldots, e_n^t(x_1, \ldots, x_{2n})).\]
\end{proposition}

\begin{proof}
There is an obvious surjective map $\Phi: H^n_t \to H^n$ given by setting $t$ to $0$. $\Phi$ restricts to a map on centers, $\Phi |_Z: Z(H^n_t) \to Z(H^n)$. From the proof of Proposition \ref{kernelprop} it follows that the $x_i$ generating $Z(H^n)$ can be realized as $\sum_{a \in B^n} x_i(a)$, where $x_i(a)$ is the element of ${}_aH^n_a$ with an $x$ in the $i$th tensor factor and $1$s elsewhere, and a sign of $(-1)^{i+1}$. Viewing these elements in $H^n_t$, it is clear that they are also central there. That they satisfy only the generating relations $x_i^2 -t = 0$ and $e^t_1(x_1, \ldots, x_{2n}), \ldots, e^t _{n}(x_1, \ldots, x_{2n})$ follows from Proposition \ref{kernelpropt}. Therefore $\Phi |_Z$ is surjective.

Now if $Z(H^n_t)$ had some other generator independent from the $x_1, \ldots, x_{2n}$, then there exists some $y$, a generator of $Z(H^n_t)$ independent from $x_1, \ldots, x_{2n}$, such that $\Phi(y) = 0$. Therefore $y$ must be a multiple of some power of $t$, i.e., $y = t^k y'$ where $y'$ has no factor of $t$ and is also central in $H^n_t$. But $\Phi(y') \neq 0$ since $y'$ has no factor of $t$, so $y'$ cannot be independent from the $x_i$, contradicting our assumption.
$\blacksquare$
\end{proof}

All other propositions from the previous section may be directly carried over to the case $X^2 = t$, and elements $e_k^t(a)$ may be defined analogously, giving the following theorem.
\begin{theorem}
$\phi_{m,n}^t$ is a surjective ring homomorphism. Its kernel is the two-sided ideal generated by the elements $e^t_k(a)$ for $a \in B^m_n$ and $1 \leq k \leq \min \{m,n\}$.
\end{theorem}

%\part{Title for the second part}
%\label{sec:corpus}
%\input{sample-part}

%\part{Conclusions}
%\label{sec:conclusions}
%\input{conclusions}

%%%
%%% Appendices
%%%
\begin{comment}
\part{Appendices}
\appendix
\input{sample-appendix}
\input{sample-appendix}
\end{comment}

%%%
%%% Bibliography
%%%
%\part{Bibliography}
\addcontentsline{toc}{chapter}{Bibliography}
\bibliographystyle{plain}
\bibliography{thesis-bib}
\nocite{*}

\end{document}